\documentclass[a4paper]{amsart}
\usepackage[utf8]{inputenc}
\usepackage{amssymb}
\usepackage{stmaryrd}
\usepackage[new]{old-arrows}
\usepackage{enumitem}
\usepackage{tikz-cd}
\usepackage{url}
\usepackage[doi=false, backend=biber, style=alphabetic, sorting=nty, maxbibnames=5]{biblatex}
\DeclareFieldFormat[article]{title}{#1}
\DeclareFieldFormat{postnote}{#1}
\DeclareFieldFormat{multipostnote}{#1}
\usepackage[colorlinks, pdfpagelabels, urlcolor = blue, pdfstartview = FitH, bookmarksopen = true, bookmarksnumbered = true, linkcolor = black, plainpages = false, hypertexnames = false, citecolor = blue]{hyperref}
\usepackage{cleveref}

\newtheorem{theorem}{Theorem}[section]
\newtheorem{corollary}[theorem]{Corollary}
\newtheorem{lemma}[theorem]{Lemma}
\newtheorem{proposition}[theorem]{Proposition}
\newtheorem{conjecture}[theorem]{Conjecture}
\newtheorem{claim}{Claim}

\theoremstyle{definition}
\newtheorem{definition}[theorem]{Definition}
\newtheorem{construction}[theorem]{Construction}
\newtheorem{case}{Case}
\newtheorem{assumption}[theorem]{Assumption}

\theoremstyle{remark}
\newtheorem{remark}[theorem]{Remark}
\newtheorem*{warning}{Warning}

\newcommand{\Z}{\mathbb{Z}}
\newcommand{\C}{\mathbb{C}}
\newcommand{\Q}{\mathbb{Q}}
\newcommand{\OO}{\mathcal{O}}

\newcommand{\Hom}{\mathrm{Hom}}

\newcommand{\id}{\mathrm{id}}

\title{Holomorphic 1-forms without zeros on K\"ahler threefolds}
\author{Simon Pietig}
\date{\today}

\begin{document}
\begin{abstract}
    We classify all smooth compact connected K\"ahler threefolds that admit the structure of a $C^\infty$-fiber bundle over the circle. This generalizes the work of Hao and Schreieder \cite{Hao_Schreieder_3-folds} in the projective case. In contrast to the projective case, there cannot always exist a smooth morphism to a positive-dimensional torus. Instead, we show that such a compact K\"ahler threefold admits a finite \'etale cover that is bimeromorphic to a $\mathbb{P}^1$-, $\mathbb{P}^2$-, or Hirzebruch surface-bundle over a locally trivial torus-fiber bundle over a smooth compact connected K\"ahler base. Our results prove Kotschick's conjecture in dimension 3.
\end{abstract}
\maketitle
\setcounter{tocdepth}{1}
\tableofcontents
\section{Introduction}
In \cite{Kotschick_1-forms_fibrations_char-numbers}, Kotschick conjectured the following:
\begin{conjecture}\label{conjecture_Kotschick}
    Let $X$ be a compact K\"ahler manifold. Then the following conditions are equivalent:
    \begin{enumerate}[label=\textit{(\Alph*)}]
        \item\label{conditionA} $X$ admits a holomorphic one-form without zeros;
        \item\label{conditionB} $X$ admits a real closed 1-form without zeros.
    \end{enumerate}
\end{conjecture}
Note that by Tischler \cite{Tischler}, the second condition is equivalent to $X$ being a $C^\infty$-fibre bundle over $S^1$. In \cite{Schreieder_topology}, Schreieder introduced a third condition:
\begin{enumerate}[label=\textit{(\Alph*)}]
    \setcounter{enumi}{2}
    \item\label{conditionC} \textit{There is a holomorphic 1-form $\omega\in H^0(X,\Omega^1_X)$, such that for any finite \'etale cover $\tau\colon Y\rightarrow X$, the sequence
    \begin{equation*}
        H^{i-1}(Y,\C)\stackrel{\wedge\tau^*\omega}{\longrightarrow}H^i(Y,\C)\stackrel{\wedge\tau^*\omega}{\longrightarrow}H^{i+1}(Y,\C)
    \end{equation*}
    given by the cup product, is exact for all $i$.}
\end{enumerate}
The implication $\ref{conditionA}\Rightarrow\ref{conditionB}$ was proven by Kotschick \cite[][Prop. 7]{Kotschick_1-forms_fibrations_char-numbers}, and the implication $\ref{conditionB}\Rightarrow\ref{conditionC}$ was proven by Schreieder \cite[][Thm. 1.2]{Schreieder_topology}. Our main result is as follows:
\begin{theorem}\label{conjecture_Kotisch_true_in_dim_3}
    Let $X$ be a compact K\"ahler manifold of dimension 3. Then all three conditions above are equivalent to each other: $\ref{conditionA}\Leftrightarrow\ref{conditionB}\Leftrightarrow\ref{conditionC}$.
\end{theorem}
The equivalence of \ref{conditionA}, \ref{conditionB} and \ref{conditionC} has already been proven in dimension 2 by Schreieder \cite[][Thm. 1.3]{Schreieder_topology} and for projective manifolds of dimension 3 by Hao and Schreieder \cite[][Thm 1.1]{Hao_Schreieder_3-folds}. Their result is a consequence of the following classification result:
\begin{theorem}\label{structure_theorem_projective_case}\cite[][Thm. 1.3]{Hao_Schreieder_3-folds}
    Let $X$ be a smooth projective threefold that satisfies condition \ref{conditionC} from \Cref{conjecture_Kotschick}. Then the minimal model program for $X$ yields a birational morphism $\sigma\colon X\rightarrow X^{min}$ to a smooth projective threefold $X^{min}$, such that:
    \begin{enumerate}[label=(\roman*)]
        \item\label{structure_theorem_projective_case_birational} $\sigma\colon X\rightarrow X^{min}$ is a sequence of blow-ups along elliptic curves that are not contracted via the natural map to the Albanese $Alb(X^{min})$.
        \item\label{structure_theorem_projective_case_smooth_morphism} There is a smooth morphism $\pi\colon X^{min}\rightarrow A$ to a positive-dimensional abelian variety $A$.
        \item\label{structure_theorem_projective_case_Kodaira-dim_non-neg} If $\kappa(X)\geq 0$, then a finite \'etale cover $\tau\colon X'\rightarrow X^{min}$ splits into a product $X'\cong S'\times A'$, where $S'$ is smooth projective and $A'$ is an abelian variety such that for all $s'\in S'$, the natural composition
        \begin{equation*}
            A'\cong\{s'\}\times A'\hookrightarrow S'\times A'\stackrel{\tau}{\longrightarrow}X^{min}\stackrel{\pi}{\longrightarrow}A
        \end{equation*}
        is a finite \'etale cover.
        \item If $\kappa(X)=-\infty$, then one of the following holds:
        \begin{enumerate}
            \item $X^{min}$ admits a del Pezzo fibration over an elliptic curve;
            \item $X^{min}$ has the structure of a conic bundle $f\colon X^{min}\rightarrow S$ over a smooth projective surface $S$ which satisfies $\ref{conditionA}\Leftrightarrow\ref{conditionB}\Leftrightarrow\ref{conditionC}$. Moreover, $f$ is either smooth, or $A$ is an elliptic curve and the degeneration locus of $f$ is a disjoint union of smooth elliptic curves on $S$ which are \'etale over $A$ via the map $S\rightarrow A$ induced by $\pi$.
        \end{enumerate}
    \end{enumerate}
\end{theorem}
Note that the \cref{structure_theorem_projective_case_birational} and \cref{structure_theorem_projective_case_smooth_morphism} above imply that the pullback of a general 1-form from $A$ to $X$ has no zeros.
\begin{warning}
    \Cref{structure_theorem_projective_case_smooth_morphism} and \cref{structure_theorem_projective_case_Kodaira-dim_non-neg} in \Cref{structure_theorem_projective_case} fail in the K\"ahler case. This is also the case for non-projective K\"ahler surfaces that satisfy condition \ref{conditionC}. For an example, see \Cref{remark_example_no_splitting} below.
\end{warning}
The key to generalizing the above result to the K\"ahler case is the insight that after a finite \'etale cover, $X^{min}$ resp.~ $S$, should admit the structure of a locally trivial torus fiber bundle with trivial monodromy if $\kappa(X)\geq 0$, resp.~ $\kappa(X)=-\infty$. This leads to our main result. For a more precise version, see \Cref{structure_theorem_Kaeler_3-folds} below.
\begin{theorem}\label{structure_theorem_introduction}
    Let $X$ be a compact connected K\"ahler manifold of dimension 3 that satisfies condition \ref{conditionC} from \Cref{conjecture_Kotschick}. Then there is a finite \'etale cover $X'\rightarrow X$ and a bimeromorphic morphism $\sigma\colon X'\rightarrow X^{min}$ such that
    \begin{enumerate}[label=(\roman*)]
        \item $\sigma\colon X'\rightarrow X^{min}$ is a sequence of blow-ups along elliptic curves that are not contracted via the natural map to the Albanese $Alb(X^{min})$.
        \item\label{structure_theorem_introduction_Kodaira-dim_non-neg} If $\kappa(X)\geq 0$, then $X^{min}$ admits the structure of a locally trivial fiber bundle $Z:=X^{min}\rightarrow B$ over a compact connected K\"ahler manifold $B$ whose fibers are all isomorphic to a positive-dimensional torus $A$. The manifold $B$ does not satisfy condition \ref{conditionC}.
        \item\label{structure_theorem_introduction_Kodaira-dim_neg} If $\kappa(X)=-\infty$, then one of the following holds:
        \begin{enumerate}[label=(\alph*)]
            \item\label{structure_theorem_introduction_fiber-dim2} $X^{min}$ admits a smooth morphism $X^{min}\rightarrow E$ to an elliptic curve $E$ whose fibers are all isomorphic to $\mathbb{P}^2$ or to a Hirzebruch surface;
            \item\label{structure_theorem_introduction_fiber-dim1} $X^{min}$ admits a smooth morphism $X^{min}\rightarrow S$ to a compact connected K\"ahler manifold $S$ of dimension 2 whose fibers are all isomorphic to $\mathbb{P}^1$. $S$ admits the structure of a locally trivial fiber bundle $Z:=S\rightarrow B$ over a compact connected K\"ahler manifold $B$ whose fibers are all isomorphic to a positive-dimensional torus $A$. The manifold $B$ does not satisfy condition \ref{conditionC}.
        \end{enumerate}
        \item\label{structure_theorem_introduction_omega_restricts_non_trivial} If $\omega\in H^0(X,\Omega^1_X)$ is a holomorphic 1-form such that $(X,\omega)$ satisfies condition \ref{conditionC}, then the induced holomorphic 1-forms on the successive blow-downs $\sigma\colon X'\rightarrow X^{min}$ restrict non-trivially on the elliptic centers and the induced holomorphic 1-form $\omega_Z\in H^0(Z,\Omega^1_Z)$ restricts non-trivially on the fibers of $Z\rightarrow B$.
    \end{enumerate}
\end{theorem}
\begin{corollary}
    In \Cref{structure_theorem_introduction} one can choose the finite \'etale cover $X'\rightarrow X$ such that depending on the Kodaira dimension $\kappa(X)$ one of the following possibilities holds true:
    \begin{enumerate}[label=(\roman*)]
        \item $\kappa=-\infty$: in \cref{structure_theorem_introduction_fiber-dim2} $X^{min}$ admits a smooth morphism $X^{min}\rightarrow E$ to an elliptic curve $E$ whose fibers are all isomorphic to $\mathbb{P}^2$ or to a Hirzebruch surface; in \cref{structure_theorem_introduction_fiber-dim1} $B$ is either a point and $Z$ a 2-torus or $B$ is a curve of genus at least 2 and $Z\rightarrow B$ is a locally trivial elliptic fiber bundle;
        \item $\kappa=0$: $Z$ splits into a product where one factor is a positive-dimensional torus and the other factor is $B$. The fiber bundle $Z\rightarrow B$ is the projection;
        \item $\kappa=1$: $B$ is either a curve of genus at least 2 and $Z\rightarrow B$ is a locally trivial 2-torus fiber bundle or $B$ is a surface of Kodaira dimension 1 and $Z\rightarrow B$ is a locally trivial elliptic fiber bundle;
        \item $\kappa=2$: $B$ is a surface of Kodaira dimension 2 and $Z\rightarrow B$ is an elliptic fiber bundle.
    \end{enumerate}
    Part (ii) follows from the Beauville--Bogomolov decomposition \cite[][Thm. 2]{Beauville_decomposition}.
\end{corollary}
As in the work by Hao and Schreieder our main result \Cref{structure_theorem_introduction} allows us to prove Kotschick's conjecture in dimension 3.
\begin{proof}[Proof of \Cref{conjecture_Kotisch_true_in_dim_3}]
    By \cite[][Thm. 1.2]{Schreieder_topology}, it suffices to prove the implication $\ref{conditionC}\Rightarrow\ref{conditionA}$. Let $X$ be a compact K\"ahler manifold of dimension 3 that satisfies condition \ref{conditionC} via the holomorphic 1-form $\omega\in H^0(X,\Omega^1_X)$. We can assume that $X$ is connected. Let $X'$, $X^{min}$, $Z$ and $B$ be as in \Cref{structure_theorem_introduction}. To prove that $\omega$ has no zeros, it suffices to prove that its pullback $\omega'\in H^0(X',\Omega^1_{X'})$ has no zeros. As the induced holomorphic 1-forms on the successive blow-downs $\sigma\colon X'\rightarrow X^{min}$ restrict non-trivially on the elliptic centers, it suffices to prove that the induced holomorphic 1-form on $X^{min}$ has no zeros. This in turn can be reduced to showing that $\omega_Z\in H^0(Z,\Omega^1_Z)$ has no zeros, which is true since $\omega_Z$ restricts non-trivially on the torus fibers of $Z\rightarrow B$ by \Cref{structure_theorem_introduction} \cref{structure_theorem_introduction_omega_restricts_non_trivial}.
\end{proof}
\begin{remark}
    Note that in \Cref{structure_theorem_introduction} \cref{structure_theorem_introduction_Kodaira-dim_non-neg} and \cref{structure_theorem_introduction_Kodaira-dim_neg}, it can be easily arranged that $B$ does not satisfy condition \ref{conditionC} by an induction argument. However, without this arrangement \cref{structure_theorem_introduction_omega_restricts_non_trivial} fails. As an example, take $X=X'=X^{min}=Z$ to be the product of three elliptic curves, and let $Z\rightarrow B$ be the projection onto the last factor. The holomorphic 1-form coming from the last factor has no zeros but is trivial on the fibers of $Z\rightarrow B$. Since the above proof of \Cref{conjecture_Kotisch_true_in_dim_3} relies on \cref{structure_theorem_introduction_omega_restricts_non_trivial}, the detail that $B$ does not satisfy condition \ref{conditionC} is crucial.
\end{remark}
\subsection{Outline of the argument for \Cref{structure_theorem_introduction}}
Let $X$ be a compact connected K\"ahler manifold of dimension 3 that satisfies condition \ref{conditionC} from \Cref{conjecture_Kotschick}. Recall that, in contrast to the result of Hao and Schreieder, there cannot always be a smooth morphism to a positive-dimensional torus. We overcome this obstacle by shifting the perspective on the splitting $X'\cong S'\times A'$ in \Cref{structure_theorem_projective_case}: the morphism $X'\rightarrow A'$ does not always exist. However, the morphism $X'\rightarrow S'$ can be constructed from the Iitaka fibration and is the equivalent of the torus fiber bundle $Z\rightarrow B$ in \Cref{structure_theorem_introduction}.\par
In the projective case, the logical order is as follows: suppose that $X$ is a minimal smooth projective 3-dimensional manifold of Kodaira dimension 1 or 2 and satisfies condition \ref{conditionC}. Then one shows that after replacing $X$ by a finite \'etale cover if necessary, the general fiber of the Iitaka fibration $f\colon X\rightarrow Y$ is not contracted by the Albanese morphism. The general fiber is thus mapped to a translate of a subtorus $A\subset Alb(X)$. One then dualizes this inclusion to obtain $X\rightarrow Alb(X)\rightarrow A$. The projective manifold $S'$ as well as the splitting $X'\cong S'\times A'$ of a finite \'etale cover $X'\rightarrow X$ are then constructed from a general fiber of $X\rightarrow Alb(X)\rightarrow A$. As a consequence, one concludes that $X'\rightarrow S'$ is a torus fiber bundle. We omit the details here, the point being that since short exact sequences of complex tori in general do not split after a finite \'etale cover, one cannot expect to apply similar methods in the K\"ahler case \cite[][Sec. 1.5]{Hao_Schreieder_3-folds}. Instead, we use the theory of twisted, equivariant Weierstraß models and tautological models developed in \cite{Nakayama_Weierstrass} \cite{Nakayama_elliptic-fibrations}, \cite{Claudon_Hoering_Lin_fundamental_group} and \cite{Lin_algebraic_approximation_codim_1}.\par
In the following, we sketch the argument for minimal smooth compact connected 3-dimensional K\"ahler manifolds $X$ of Kodaira dimension 1 and 2. For the reduction to smooth minimal models as well as for the cases of Kodaira dimension 3, 0, and $-\infty$, we refer to Hao and Schreieder. Their arguments remain valid also in the K\"ahler case since the minimal model program for K\"ahler threefolds \cite{Hoering_Peternell_Kaehler-MMP}, \cite{Hoering_Peternell_Kaehler-MFS} is fully established and provides the same bimeromorphic morphisms as the projective minimal model program. Furthermore, the abundance theorem also holds for K\"ahler threefolds \cite[][Thm. 1.1]{Campana_Hoering_Peternell_Kaehler-abundance}, \cite{Erratum_Addendum_Kaehler-abundance}.\par
Suppose $\kappa(X)=2$. By the abundance theorem, we obtain the Iitaka fibration $f\colon X\rightarrow S$, which is an elliptic fibration over a normal projective surface $S$. One can show that the local monodromies are finite and are trivialized by a finite \'etale cover $S'\rightarrow S$. The induced elliptic fibration $f'\colon X'\rightarrow S'$ is then generically isotrivial, and all but finitely many singular fibers are multiples of smooth elliptic curves. We aim to remove the multiple fibers by a finite \'etale cover. Suppose that $S'$ is smooth and that we can find a Galois cover $\overline{S}\rightarrow S'$ with Galois group $G$ and $\overline{S}$ smooth such that the induced elliptic fibration $\overline{f}\colon\overline{X}\rightarrow\overline{S}$ has at most finitely many singular fibers. (In general, this can only be achieved after replacing $S'$ by a log-desingularization of $S'$ together with the discriminant divisor of $f'$. However, to simplify our situation, let us assume that this is not an issue.) Note that $\overline{X}\rightarrow X'$ is finite but may no longer be \'etale since the branch locus of $\overline{S}\rightarrow S$ may be too large. Our assumptions allow to replace $\overline{f}\colon\overline{X}\rightarrow\overline{S}$ by a twisted $G$-equivariant Weierstraß model $p\colon W\rightarrow\overline{S}$, which is an explicitly constructed elliptic fibration that is $G$-equivariantly bimeromorphic to $\overline{f}\colon\overline{X}\rightarrow\overline{S}$. One can show that $p\colon W\rightarrow\overline{S}$ is a locally trivial elliptic fiber bundle, where $G$ acts on the elliptic fibers by translations. Taking the quotient by the normal subgroup $H\subset G$ consisting of those elements that do not act on the fibers of $p\colon W\rightarrow\overline{S}$ results in the locally trivial elliptic fiber bundle with trivial monodromy $Y'':=W/H\rightarrow S'':=\overline{S}/H$ over the smooth surface $S''$. Suppose, for simplicity, that $S''$ is minimal. Note that the group $G/H$ acts fixed-point freely on $Y''$. Define $Y':=Y''/(G/H)\cong W/G$. Then $Y'\rightarrow S'$ is bimeromorphic to $f'\colon X'\rightarrow S'$ and $Y''\rightarrow Y'$ is finite \'etale. As $Y'$ is smooth and bimeromorphic to $X'$ by construction, we can find a finite \'etale cover $X''\rightarrow X'$ bimeromorphic to $Y''\rightarrow Y'$. Since $X$ was assumed to be minimal, and since $X''\rightarrow X'\rightarrow X$ is a finite \'etale cover, also $X''$ is minimal. Moreover, as $S''$ is smooth and minimal, the same holds for $Y''$. Thus, by \cite[][Thm. 4.9]{Kollar_flops}, $Y''$ and $X''$ are connected by a sequence of flops. However, one can show that $Y''$ does not admit any non-trivial flops as $Y''\rightarrow S''$ is a locally trivial elliptic fiber bundle. Therefore, $Y''$ and $X''$ are isomorphic. Hence $Y''\cong X''\rightarrow X'\rightarrow X$ is the desired finite \'etale cover and $Y''\rightarrow S''$ is the desired fiber bundle.\par
Recall that we assumed that $S'$ is smooth and that we can find a Galois cover $\overline{S}\rightarrow S'$ with Galois group $G$ and $\overline{S}$ smooth such that the induced elliptic fibration $\overline{f}\colon\overline{X}\rightarrow\overline{S}$ has at most finitely many singular fibers. Without these simplifications, we can adapt the above argument as follows. We replace $S'$ by a log-desingularization of $S'$ together with the discriminant divisor of $f'$, and construct the finite \'etale cover $Y''\rightarrow Y'$ as well as the locally trivial elliptic fiber bundle $Y''\rightarrow S''$ with trivial monodromy as before. One can show that there are unique smooth minimal models $Y''_{min}$, $Y'_{min}$, $S_{min}$ of $Y''$, $Y'$, resp.~ $S''$. Moreover, $Y''\rightarrow Y'$, resp.~ $Y''\rightarrow S''$ descends to $Y''_{min}\rightarrow Y'_{min}$, resp.~ $Y''_{min}\rightarrow S_{min}$ and still is finite \'etale, resp.~ a locally trivial elliptic fiber bundle with trivial monodromy:
\[\begin{tikzcd}
	{S''} & {Y''} & {Y'} \\
	{S_{min}} & {Y''_{min}} & {Y'_{min}.}
	\arrow[from=1-1, to=2-1]
	\arrow[from=1-2, to=1-1]
	\arrow[from=1-2, to=1-3]
	\arrow[from=1-2, to=2-2]
	\arrow[from=1-3, to=2-3]
	\arrow[from=2-2, to=2-1]
	\arrow[from=2-2, to=2-3]
\end{tikzcd}\]
Since $Y'_{min}$ is smooth and bimeromorphic to $X'$ by construction, there is a finite \'etale cover $X''\rightarrow X'$ which is bimeromorphic to $Y''_{min}\rightarrow Y'_{min}$. As before, one can show that $Y''_{min}$ does not admit any flops, and thus, by \cite[][Thm. 4.9]{Kollar_flops}, $Y''_{min}$ and $X''$ are isomorphic. Hence, $Y''_{min}\cong X''\rightarrow X'\rightarrow X$ is the desired finite \'etale cover and $Y''_{min}\rightarrow S^{min}$ is the desired fiber bundle.\par
Suppose $\kappa(X)=1$. By the abundance theorem, we obtain the Iitaka fibration $f\colon X\rightarrow C$, where $C$ is a smooth compact curve. By \cite[][Thm. 6.1]{Hao_Schreieder_3-folds}, the general fiber is a 2-torus or a bielliptic surface. We distinguish between three cases. The first case is that the general fiber of $f$ is not contracted to a curve by the Albanese morphism. A result of Campana and Peternell \cite[][Thm. 4.2 and the following remark]{Campana_Peternell_2-forms} shows that $f$ is quasi smooth. We then remove the multiple fibers by finding a finite \'etale cover $X'\rightarrow X$ induced by a finite morphism $C'\rightarrow C$ such that the induced morphism $f'\colon X'\rightarrow C'$ is smooth. One can show that the monodromy $R^1f'_*\Z$ of $f'$ is finite. In particular, it vanishes after a further finite \'etale cover $X''\rightarrow X'$ induced by a finite \'etale cover $C''\rightarrow C'$. Thus, $X''\rightarrow C''$ is the desired fiber bundle.\par
The second case is that for every finite \'etale cover $\tau\colon\tilde{X}\rightarrow X$, the general fiber of the induced fibration $\tilde{f}\colon\tilde{X}\rightarrow\tilde{C}$ by the Stein factorization is contracted to a curve by the Albanese morphism. If the genus of $C$ is positive, we can factor $f\colon X\rightarrow C$ through the normalization of the image of the Albanese morphism $Y\rightarrow Alb(X)$. We then construct a finite \'etale cover $X'\rightarrow X$ to obtain the diagram
\[\begin{tikzcd}
	{X'} && {Y'} \\
	& {C',}
	\arrow["g'", from=1-1, to=1-3]
	\arrow["f'"', from=1-1, to=2-2]
	\arrow["\pi'", from=1-3, to=2-2]
\end{tikzcd}\]
where $\pi'\colon Y'\rightarrow C'$ is a locally trivial elliptic fiber bundle with trivial monodromy, $g'$ is an elliptic fibration without multiple fibers, and the general fiber of $f'$ is a 2-torus. A similar diagram exists if the genus of $C$ is zero. However, as there are bad orbifold structures on $\mathbb{P}^1$, one is forced to work with a $G$-equivariant diagram. Let us assume for simplicity that this is not an issue. One can show that there is a finite \'etale cover $Y''\rightarrow Y'$ induced by a finite \'etale cover on the fibers of $Y'\rightarrow C'$ such that the general fiber of the induced fibration $f''\colon X'':=X'\times_{Y'}Y''\rightarrow C'$ is a product of two elliptic curves. This implies that the twisted Weierstraß model $p\colon W\rightarrow Y''$ of the induced elliptic fibration $X''\rightarrow Y''$ descends to a twisted Weierstraß model $q\colon V\rightarrow C'$, i.e., $W\cong V\times_{C'}Y''$. In particular, the morphism $W\rightarrow V$ is a locally trivial elliptic fiber bundle with trivial monodromy over a surface. Suppose for simplicity that $V$ (and hence $W$) is smooth and minimal. Then, as in the case of $\kappa(X)=2$, $W$ does not admit any flops and is thus isomorphic to $X''$ by \cite[][Thm. 4.9]{Kollar_flops}. The desired fiber bundle is then $X''\cong W\rightarrow V$.\par
To conclude the case of $\kappa(X)=1$, we show that the Albanese morphism cannot contract the general fiber onto a point since in this case $C$ must be elliptic and $f\colon X\rightarrow C$ smooth. A finite \'etale cover is thus a torus, which cannot have Kodaira dimension 1.
\subsection{Conventions and notation}
By a complex analytic variety, we mean a Hausdorff, second-countable, irreducible, and reduced complex space. A curve, surface, resp.~ threefold is a complex analytic variety of dimension 1, 2, resp.~ 3. A complex analytic variety is called K\"ahler if it admits a K\"ahler metric in the sense of \cite[][II. 1]{Varouchas_Kaehler_spaces_proper_open_morphisms}. A compact complex analytic K\"ahler variety $X$ is called minimal if it is $\Q$-factorial, has terminal singularities, and the canonical divisor $K_X$ is nef. A torus is a complex manifold isomorphic to the quotient $\C^g/\Lambda$, where $\Lambda\subset \C^g$ is a lattice of rank $2g$. A 1-dimensional torus is called an elliptic curve. If a torus is projective, we call it an abelian variety. (This should not be confused with the term complex analytic variety.) The dual of a vector space $V$ is denoted by $V^\vee$. Similarly, the dual of a vector bundle $\mathcal{V}$ on a complex manifold is denoted by $\mathcal{V}^\vee$.
\subsection*{Acknowledgments}
I would like to thank Stefan Schreieder for suggesting this problem to me, for very helpful discussions, and for comments on this manuscript. The research was conducted in the framework of the DFG-funded research training group RTG 2965: From Geometry to Numbers.
\section{Preliminaries}
\subsection{Singular K\"ahler spaces}
A complex analytic variety is called K\"ahler if it admits a K\"ahler metric in the sense of \cite[][II. 1]{Varouchas_Kaehler_spaces_proper_open_morphisms}. We recall some standard results about compact K\"ahler varieties.
\begin{proposition}\cite[][II. Prop. 1.3.1, Cor. 3.2.2]{Varouchas_Kaehler_spaces_proper_open_morphisms}, \cite[][Prop. 3.5]{Graf_Kirschner_finite_quotients_3-tori}
    Let $f\colon X\rightarrow Y$ be a morphism between compact complex analytic varieties. Then the following holds true.
    \begin{enumerate}[label=(\roman*)]
        \item If $Y$ is K\"ahler and $f$ is a closed embedding, then $X$ is K\"ahler.
        \item If $Y$ is K\"ahler and $f$ is a projective morphism, then $X$ is K\"ahler. In particular, by Hironaka's result on resolution of singularities, compact complex analytic K\"ahler varieties admit desingularizations by compact K\"ahler manifolds.
        \item If $f$ is a finite surjective morphism, then $X$ is K\"ahler if and only if $Y$ is K\"ahler.
    \end{enumerate}
\end{proposition}
\subsection{Torus fiber bundles}
Most of the following material can be found in \cite{Claudon_tori_fiber_bundles}.
\begin{definition}\label{definition_torus_fiber_bundle}
    A torus fiber bundle is defined as a smooth proper surjective morphism $f\colon X\rightarrow Y$ from a connected complex manifold $X$ to a connected complex manifold $Y$ such that every fiber is reduced and isomorphic to a positive-dimensional torus.
\end{definition}
Let $A=\C^g/\Lambda$ be a torus. Recall that the Albanese morphism with base point $0\in A$
\begin{equation*}
    A\stackrel{\sim}{\longrightarrow} Alb(A):=H^0(A,\Omega^1_A)^\vee/H_1(A,\Z)
\end{equation*}
is a canonical isomorphism of tori. Poincar\'e duality yields a canonical isomorphism $H_1(A,\Z)\cong H^{2g-1}(A,\Z)$, and Serre duality yields a canonical isomorphism $H^0(A,\Omega^1_A)^\vee\cong H^{g-1,g}(A)$. Writing furthermore $H^{g-1,g}(A)\cong H^{2g-1}(A,\C)/F^g$, where $F^g$ is the degree $g$-part of the Hodge filtration, gives rise to a canonical isomorphism
\begin{equation*}
    A\stackrel{\sim}{\longrightarrow}Alb(A)\cong (H^{2g-1}(A,\C)/F^g)/H^{2g-1}(A,\Z).
\end{equation*}
In particular, $A$ is determined by its Hodge structure in degree $2g-1$.
Let $f\colon X\rightarrow Y$ be a torus fiber bundle of fiber dimension $g>0$. Let $B\subset Y$ be an open subset isomorphic to an open ball $B\subset\C^n$ centered at the origin. Denote by $X_b$ the fiber over $b\in B$. Ehresmann's theorem \cite[][Thm. 9.3]{Voisin_HT_1} implies that the restriction maps $H^{2g-1}(f^{-1}(B),\C)\stackrel{\sim}{\rightarrow}H^{2g-1}(X_b,\C)$ are isomorphisms for all $b\in B$. We hence obtain canonical isomorphisms
    \begin{equation*}
        H^{2g-1}(X_b,\C)\cong H^{2g-1}(f^{-1}(B),\C)\cong H^{2g-1}(X_0,\C).
    \end{equation*}
    Denote by $F^g(b)$ the $g$-th part of the Hodge filtration $F^g(b)\subset H^{2g-1}(X_b,\C)$ for $b\in B$. By \cite[][Thm. 10.9]{Voisin_HT_1}, the period map
    \begin{align*}
        B&\longrightarrow Gr(g, H^{2g-1}(X_0,\C))\\
        b&\longmapsto F^g(b)
    \end{align*}
    to the Grassmannian $Gr(g, H^{2g-1}(X_0,\C))$ of $g$-dimensional subspaces in $H^{2g-1}(X_0,\C)$ is holomorphic. Thus, the subspaces $F^g(b)\subset H^{2g-1}(X_0,\C)$ vary holomorphically, and hence the quotient
    \begin{equation*}
        (R^{2g-1}f_*\Z\otimes\OO_Y)/(R^{g-1}f_*\Omega^g_{X/Y})
    \end{equation*}
    is a holomorphic vector bundle on $Y$. This gives rise to the following definitions.
\begin{definition}\label{def_variation_HS_Jacobian_fibration}
    Let $f\colon X\rightarrow Y$ be a torus fiber bundle of fiber dimension $g>0$.
    \begin{enumerate}[label=(\roman*)]
        \item The variation of Hodge structures $H$ associated to $f$ consists of the following data: the local system $H_\Z:=R^{2g-1}f_*\Z$ on $Y$ of free abelian groups of rank $2g$, the associated vector bundle $\mathcal{H}:=H_\Z\otimes\OO_Y$ on $Y$, the holomorphic subbundle $\mathcal{F}^{g}:=R^{g-1}f_*\Omega^g_{X/Y}$, and the quotient $\mathcal{H}^{g-1,g}:=\mathcal{H}/\mathcal{F}^{g}$.
        \item We say that $f$ has trivial monodromy if the local system $H_\Z$ is isomorphic to $\Z^{2g}$.
        \item We define the Jacobian fibration associated to $f$ as the complex manifold
        \begin{equation*}
            J_H:=\mathcal{H}^{g-1,g}/H_\Z
        \end{equation*}
        obtained by taking the quotient of the vector bundle $\mathcal{H}^{g-1,g}\rightarrow Y$ by the $H_\Z$-action together with the natural morphism $\pi\colon J_H\rightarrow Y$.
        \item The zero section of $\mathcal{H}^{g-1,g}\rightarrow Y$ induces the global section $\sigma_0\colon Y\rightarrow J_H$. The sheaf of local sections of $\pi\colon J_H\rightarrow Y$ is a sheaf of abelian groups with addition of local sections and neutral element $\sigma_0$. It is denoted by $\mathcal{J}_H=\mathcal{H}^{g-1,g}/H_\Z$, where we identify the vector bundle $\mathcal{H}^{g-1,g}$ with its sheaf of local sections.
        \item Let $\eta\in H^1(Y,\mathcal{J}_H)$ be a cohomology class. The $\mathcal{J}_H$-torsor associated to $\eta$ is denoted by
        \begin{equation*}
            \pi^\eta\colon J^\eta_H\rightarrow Y.
        \end{equation*}
        \item We denote by $\mathrm{exp}\colon H^1(Y,\mathcal{H}^{g-1,g})\rightarrow H^1(Y,\mathcal{J}_H)$ and $c\colon H^1(Y,\mathcal{J}_H)\rightarrow H^2(Y,H_\Z)$ the maps induced by the short exact sequence $0\rightarrow H_\Z\rightarrow\mathcal{H}^{g-1,g}\rightarrow\mathcal{J}_H\rightarrow 0$, and call them the exponential, resp.~ the Chern class map.
    \end{enumerate}
\end{definition}
\begin{proposition}\label{classification_torus_fiber_bundles}\cite[][Prop. 2.1]{Claudon_tori_fiber_bundles}
    Let $f\colon X\rightarrow Y$ be a torus fiber bundle with associated variation of Hodge structures $H$. Then there is a unique cohomology class $\eta\in H^1(Y,\mathcal{J}_H)$ and an isomorphism $\varphi\colon X\stackrel{\sim}{\rightarrow}J_H^\eta$ over $Y$. Moreover, if $f\colon X\rightarrow Y$ admits a global section $\sigma\colon Y\rightarrow X$ then $\eta=0$ and $\varphi\colon X\rightarrow J_H$ can be chosen such that $\varphi\circ\sigma=\sigma_0$. Conversely, any cohomology class $\eta\in H^1(Y,\mathcal{J}_H)$ defines a torus fiber bundle $\pi^\eta\colon J_H^\eta\rightarrow Y$.
\end{proposition}
\begin{corollary}\label{corollary_isotrivial_and_trivial_vhs}
    Let $f\colon X\rightarrow Y$ be a torus fiber bundle of fiber dimension $g$. Suppose that $f$ is isotrivial with fiber $A\cong\C^g/\Lambda$ and that its monodromy is trivial. Then $J_H=Y\times A$ and $\mathcal{J}_H\cong\OO_Y^{\oplus g}/\Lambda$.
\end{corollary}
\begin{definition}\label{definition_isotrivial_and_trivial_vhs}
    In the situation of \Cref{corollary_isotrivial_and_trivial_vhs}, we denote the $\mathcal{J}_H$ torsor associated to $\eta\in H^1(Y,\mathcal{J}_Y)$ by $\pi^\eta\colon(Y\times A)^\eta\rightarrow Y$.
\end{definition}
\begin{proposition}\label{fundamental_properties_special_torus_bundles}
    Let $Y$ be a compact connected K\"ahler manifold, let $A=\C^g/\Lambda$ be a torus, and let $\eta\in H^1(Y,\OO_Y^{\oplus g}/\Lambda)$.
    \begin{enumerate}[label=(\roman*)]
        \item\label{torus_bundle_kaehler} The complex manifold $(Y\times A)^\eta$ is K\"ahler if and only if $c(\eta)\in H^2(Y,\Lambda)$ is torsion.
        \item\label{1-forms_on_torus_bundles} Suppose that $(Y\times A)^\eta$ is K\"ahler. Let $\pi\colon (Y\times A)^\eta\rightarrow Y$ be the projection. Then for every $y\in Y$ the restriction map
        \begin{equation*}
            H^0((Y\times A)^\eta,\Omega^1_{(Y\times A)^\eta})\longrightarrow H^0(\pi^{-1}(y),\Omega^1_{\pi^{-1}(y)})\cong H^0(A,\Omega^1_A)
        \end{equation*}
        fits into a short exact sequence
        \begin{equation*}
            0\longrightarrow H^0(Y,\Omega^1_Y)\longrightarrow H^0((Y\times A)^\eta,\Omega^1_{(Y\times A)^\eta})\longrightarrow H^0(A,\Omega^1_A)\longrightarrow 0.
        \end{equation*}
        In particular, any holomorphic 1-form $\omega\in H^0((Y\times A)^\eta,\Omega^1_{(Y\times A)^\eta})$ that does not come from $Y$ has no zeros.
        \item\label{torus_bundle_multisection} The following is equivalent:
        \begin{enumerate}[label=(\alph*)]
            \item\label{torus_bundle_multisection_a} The torus fiber bundle $(Y\times A)^\eta\rightarrow Y$ admits a multisection.
            \item\label{torus_bundle_multisection_b} There is a finite \'etale cover $Y'\rightarrow Y$ such that $(Y\times A)^\eta\times_YY'\cong Y'\times A$.
            \item\label{torus_bundle_multisection_c} The cohomology class $\eta\in H^1(Y,\OO_Y^{\oplus n}/\Lambda)$ is torsion.
        \end{enumerate}
    \end{enumerate}
\end{proposition}
\begin{remark}
    \Cref{1-forms_on_torus_bundles} above fails if $c(\eta)$ is non-torsion. See \cite[][V. Prop. 5.3 (ii)]{Barth_Hulek_Peters_vandeVen_surfaces} for more details.
\end{remark}
\begin{proof}
    To prove \cref{torus_bundle_kaehler}, suppose first that $(Y\times A)^\eta$ is K\"ahler. Then \cite[][Prop. 2.5]{Claudon_tori_fiber_bundles} implies that $c(\eta)$ is torsion. Suppose now that $c(\eta)$ is torsion. Let $m$ be a positive integer such that $c(m\eta)=0$. The multiplication by $m$-map $m\colon A\rightarrow A$ induces a finite \'etale cover
    \begin{equation*}
        (Y\times A)^\eta\stackrel{m}{\longrightarrow}(Y\times A)^{m\eta}.
    \end{equation*}
    To show that $(Y\times A)^\eta$ is K\"ahler, it hence suffices to show that $(Y\times A)^{m\eta}$ is K\"ahler. We can thus assume $c(\eta)=0$ or, equivalently, we can assume that $\eta$ is in the image of the exponential map
    \begin{equation*}
        \mathrm{exp}\colon V:=H^1(Y,\OO_Y^{\oplus g})\longrightarrow H^1(Y,\OO_Y^{\oplus n}/\Lambda).
    \end{equation*}
    It hence suffices to show that $(Y\times A)^{\mathrm{exp}(v)}$ is K\"ahler for all $v\in V$. Let
    \begin{equation*}
        \Pi\colon\mathcal{X}\stackrel{q}{\longrightarrow} Y\times V\stackrel{pr_2}{\longrightarrow} V
    \end{equation*}
    be the universal family \cite[][Prop. 2.4]{Claudon_tori_fiber_bundles}, i.e. over every $v\in V$ lies the torus fiber bundle $\Pi^{-1}(v)\cong (Y\times A)^{\mathrm{exp}(v)}\rightarrow Y$. In particular, $\Pi^{-1}(0)\cong Y\times A$, which is K\"ahler. As $\Pi$ is a smooth morphism, there is an open neighbourhood $U\subset V$ of $0\in V$ such that $\Pi^{-1}(u)$ is K\"ahler for all $u\in U$ \cite[][9. Thm. 9.23]{Voisin_HT_1}. Let $v\in V$ be an arbitrary element. Then there is a positive integer $n$ such that $\frac{1}{n}v\in U$. The multiplication by $n$-map $n\colon A\rightarrow A$ induces a finite \'etale cover
    \begin{equation*}
        (Y\times A)^{\mathrm{exp}(\frac{1}{n}v)}\stackrel{n}{\longrightarrow}(Y\times A)^{\mathrm{exp}(v)}.
    \end{equation*}
    As the left-hand side is K\"ahler by assumption and the morphism is finite \'etale, it follows that also the right-hand side must be K\"ahler. This proves \cref{torus_bundle_kaehler}. A similar argument can be found in \cite[][Prop. 3.23]{Claudon_Hoering_Lin_fundamental_group}.\par
    To prove \cref{1-forms_on_torus_bundles}, note that $R^1\pi_*\C$ is a trivial local system. Indeed, as $R^{2g-1}\pi_*\Z=\Lambda\cong\Z^{2g}$ by construction also $R^{2g-1}\pi_*\C\cong\C^{2g}$. The cup product yields
    \begin{equation*}
        R^1\pi_*\C\otimes R^{2g-1}\pi_*\C\stackrel{\smile}{\longrightarrow}R^{2g}\pi_*\C\cong\C.
    \end{equation*}
    The local system $R^1\pi_*\C$ is thus trivial. The Leray spectral sequence 
    \begin{equation*}
        E^{p,q}_2=H^p(Y,R^q\pi_*\C)\Longrightarrow H^{p+q}((Y\times A)^\eta,\C)
    \end{equation*}
    induces an exact sequence
    \begin{equation*}
        0\longrightarrow H^1(Y,\C)\longrightarrow H^1((Y\times A)^\eta,\C)\longrightarrow H^0(Y,R^1\pi_*\C).
    \end{equation*}
    Suppose we have proved that the last map is surjective. Then for any $y\in Y$ the composition
    \begin{equation*}
        H^1((Y\times A)^\eta,\C)\longrightarrow H^0(Y,R^1\pi_*\C)\stackrel{\sim}{\longrightarrow} H^1(\pi^{-1}(y),\C)
    \end{equation*}
    is also surjective, where the last isomorphism uses that $R^1\pi_*\C$ is a trivial local system. Moreover, as this can be identified with the restriction map $H^1((Y\times A)^\eta,\C)\rightarrow H^1(\pi^{-1}(y),\C)$, we would then obtain the desired short exact sequence
    \begin{equation*}
        0\longrightarrow H^0(Y,\Omega^1_Y)\longrightarrow H^0((Y\times A)^\eta,\Omega^1_{(Y\times A)^\eta})\longrightarrow H^0(A,\Omega^1_A)\longrightarrow 0.
    \end{equation*}
    It is hence enough to show the surjectivity of $H^1((Y\times A)^\eta,\C)\rightarrow H^0(Y,R^1\pi_*\C)$. This is automatic if we can show
    \begin{equation}\label{inequality_cohomology_dim}
        h^1((Y\times A)^\eta,\C)\geq h^1(Y,\C)+ h^0(Y,R^1\pi_*\C)=h^1(Y,\C)+h^1(A,\C),
    \end{equation}
    where the last equality uses that $R^1\pi_*\C$ is a trivial local system with fiber isomorphic to $H^0(A,\C)$. To prove the inequality \cref{inequality_cohomology_dim}, note that since $(Y\times A)^\eta$ is K\"ahler, $c(\eta)$ is torsion. Let $m$ be a positive integer such that $c(m\eta)=0$. The multiplication by $m$-map $m\colon A\rightarrow A$ induces a finite \'etale cover
    \begin{equation*}
        (Y\times A)^\eta\stackrel{m}{\longrightarrow}(Y\times A)^{m\eta}.
    \end{equation*}
    This can also be written as follows: since $(Y\times A)^\eta\rightarrow Y$ is an $A$-torsor over $Y$, there is a canonical action of $A$ on $(Y\times A)^\eta$ over $Y$. Denote by $T_a\colon (Y\times A)^\eta\rightarrow (Y\times A)^\eta$ the automorphism induced by $a\in A$. On each fiber of $(Y\times A)^\eta\rightarrow Y$, the automorphism $T_a$ is given by translation with $a$. In particular, as $A$ is path-connected, $T_a$ is homotopic to the identity. Moreover, if we denote by $A[m]$ the $m$-torsion points of $A$, the \'etale cover $(Y\times A)^\eta\stackrel{m}{\rightarrow}(Y\times A)^{m\eta}$ can be identified with the quotient map $(Y\times A)^\eta\rightarrow (Y\times A)^\eta/A[m]$. Therefore, there is an isomorphism
    \begin{equation*}
        (Y\times A)^{m\eta}\cong (Y\times A)^\eta/A[m].
    \end{equation*}
    The Hochschild--Serre spectral sequence
    \begin{equation*}
        F^{p,q}=H^p(A[m],H^q((Y\times A)^\eta,\C))\Longrightarrow H^{p+q}((Y\times A)^{m\eta},\C)
    \end{equation*}
    induces an exact sequence
    \begin{equation}\label{exact_sequence_proof_1-forms_on_torus_bundles}
        0\longrightarrow H^1(A[m],H^0((Y\times A)^\eta,\C))\longrightarrow H^1((Y\times A)^{m\eta},\C)\longrightarrow H^0(A[m],H^1((Y\times A)^\eta,\C)).
    \end{equation}
    Since the action of $A[m]$ on $(Y\times A)^\eta$ is homotopic to the identity, the action of $A[m]$ on the cohomology groups $H^i((Y\times A)^\eta,\C)$ is trivial. Hence,
    \begin{align*}
        H^0(A[m],H^1((Y\times A)^\eta,\C))&=H^1((Y\times A)^\eta,\C),\\
        H^1(A[m],H^0((Y\times A)^\eta,\C))&\cong\Hom(A[m],H^0((Y\times A)^\eta,\C))=0,
    \end{align*}
    where the last equality in the second line uses that $H^0((Y\times A)^\eta,\C)$ is torsion free. Thus, the exact sequence \cref{exact_sequence_proof_1-forms_on_torus_bundles} implies
    \begin{equation*}
        h^1((Y\times A)^{m\eta},\C)\leq h^1((Y\times A)^\eta,\C).
    \end{equation*}
    Since the universal family $\Pi\colon\mathcal{X}\stackrel{q}{\rightarrow} Y\times V\stackrel{pr_2}{\longrightarrow} V$ allows to deform $(Y\times A)^{m\eta}$ to $Y\times A$, we obtain the desired inequality
    \begin{equation*}
        h^1(Y,\C)+h^1(A,\C)=h^1((Y\times A)^{m\eta},\C)\leq h^1((Y\times A)^\eta,\C),
    \end{equation*}
    which proves \cref{1-forms_on_torus_bundles}.\par
    To prove \cref{torus_bundle_multisection}, note that the equivalence of \ref{torus_bundle_multisection_a} and \ref{torus_bundle_multisection_c} is proven in \cite[][Prop. 2.2]{Claudon_tori_fiber_bundles}. It is also proven that in the case of $\ref{torus_bundle_multisection_a}\Leftrightarrow\ref{torus_bundle_multisection_c}$ one can choose the multisection $Y'\subset (Y\times A)^\eta$ such that $Y'\rightarrow Y$ is finite \'etale. Note that $(Y\times A)^\eta\times_YY'\rightarrow Y'$ has a section. It is hence isomorphic to $Y'\times A$. Thus, $\ref{torus_bundle_multisection_a}\Rightarrow\ref{torus_bundle_multisection_b}$. Conversely, suppose there is a finite \'etale cover $Y'\rightarrow Y$ such that $(Y\times A)^\eta\times_YY'\cong Y'\times A$. Then
    \begin{equation*}
        Y'\cong Y'\times\{0\}\longhookrightarrow Y'\times A\longrightarrow (Y\times A)^\eta\rightarrow Y
    \end{equation*}
    defines a multisection. Thus, $\ref{torus_bundle_multisection_b}\Rightarrow\ref{torus_bundle_multisection_a}$.
\end{proof}
\begin{lemma}\label{example_no_splitting}
    Let $Y$ be a compact connected K\"ahler manifold and let $E=\C/\Lambda$ be an elliptic curve. Suppose that $\eta\in H^1(Y,\OO_Y/\Lambda)$ is a cohomology class such that $(Y\times E)^\eta$ is K\"ahler. Then the following are equivalent:
    \begin{enumerate}[label=(\roman*)]
        \item\label{example_no_splitting1} There is a smooth morphism $\varphi\colon (Y\times E)^\eta\rightarrow A$ to a positive-dimensional torus $A$.
        \item\label{example_no_splitting2} There is a smooth morphism $Y\rightarrow B$ to a positive-dimensional torus $B$ or $\eta$ is torsion.
    \end{enumerate}
    In particular, if $Y$ does not admit any smooth morphisms to positive-dimensional tori then $(Y\times E)^\eta$ admits a smooth morphism to a positive-dimensional torus if and only if $\eta$ is torsion.
\end{lemma}
\begin{proof}
    To prove the implication $\ref{example_no_splitting1}\Rightarrow\ref{example_no_splitting2}$, suppose there is a smooth morphism $\varphi\colon (Y\times E)^\eta\rightarrow A$ to a positive-dimensional torus $A$. If the elliptic curves $E=\pi^{-1}(y)\subset (Y\times E)^\eta$ are contracted to a point in $A$, the smooth morphism $\varphi$ factors through a smooth morphism $\psi\colon Y\rightarrow A$. And hence \cref{example_no_splitting2} holds true with $B:=A$. If the elliptic curves are not contracted by $\varphi$, they are mapped to a translate of a positive-dimensional subtorus $F\subset A$. The composition $(Y\times E)^\eta\rightarrow A\rightarrow A/F$ hence factors through a smooth morphism $\psi\colon Y\rightarrow A/F$. If $A/F$ is positive dimensional, \cref{example_no_splitting2} holds true with $B:=A/F$. If $A/F$ is not positive-dimensional, $A=F$ and hence a general fiber of $(Y\times E)^\eta\rightarrow A=F$ defines a multisection of $(Y\times E)^\eta\rightarrow Y$. And so $\eta$ is torsion by \Cref{fundamental_properties_special_torus_bundles}. Thus, \cref{example_no_splitting2} holds true.\par
    To prove the implication $\ref{example_no_splitting2}\Rightarrow\ref{example_no_splitting1}$, suppose first that there is a smooth morphism $Y\rightarrow B$ to a positive-dimensional torus $B$. Then the composition $(Y\times E)^\eta\rightarrow Y\rightarrow B$ is also smooth. This shows that \cref{example_no_splitting1} holds true with $A:=B$. Suppose now that $\eta$ is torsion with order $m$, and consider the finite \'etale morphism
    \begin{equation*}
        (Y\times E)^\eta\longrightarrow (Y\times E)^{m\eta}\cong Y\times E.
    \end{equation*}
    Composing this with the projection $Y\times E\rightarrow E$ also defines a smooth morphism to a positive-dimensional torus. This shows that \cref{example_no_splitting1} holds true with $A:=E$.
\end{proof}
\begin{remark}\label{remark_example_no_splitting}
    We now obtain examples which show that \cref{structure_theorem_projective_case_smooth_morphism} and \cref{structure_theorem_projective_case_Kodaira-dim_non-neg} in \Cref{structure_theorem_projective_case} fail in the K\"ahler case: let $Y$ be of general type such that $H^1(Y,\OO_Y)\neq 0$. By a result of Popa and Schnell \cite[][Cor. 3.1]{Popa_Schnell_Kodaira_dimension_zeros_1-forms}, $Y$ does not admit a smooth morphism to a torus. Note that since $H^1(Y,\OO_Y)\neq 0$ and $\Lambda\cong\Z^2$, we have $H^1(Y,\Lambda)\neq 0$. Then any class $v\in H^1(Y,\OO_Y)$ that is not in the image of
    \begin{equation*}
        H^1(Y,\Lambda)\otimes\Q\longrightarrow H^1(Y,\OO_Y)
    \end{equation*}
    satisfies the property that $(Y\times E)^{\mathrm{exp}(v)}$ is K\"ahler and $\mathrm{exp}(v)$ is not torsion. Thus, by \Cref{fundamental_properties_special_torus_bundles}, $(Y\times E)^\eta$ admits a holomorphic 1-form without zeros but no smooth morphism to a positive-dimensional torus.
\end{remark}
In the following, we will state and prove some technical results that will be relevant later on.
\begin{lemma}\label{lemma_extension_of_etale_cover_of_elliptic_curves}
    Let $Y$ be a compact connected K\"ahler manifold and let $\varphi\colon A'=\C^g/\Lambda'\rightarrow A=\C^g/\Lambda$ be a finite \'etale cover of tori. Suppose $\eta\in H^1(Y,\OO_Y^{\oplus g}/\Lambda)$ is a cohomology class with trivial Chern class. Then there exists a non-unique cohomology class $\eta'\in H^1(Y,\OO_Y^{\oplus g}/\Lambda')$ that lifts $\eta$, i.e. the map
    \begin{equation*}
        H^1(Y,\OO_Y^{\oplus g}/\Lambda')\longrightarrow H^1(Y,\OO_Y^{\oplus g}/\Lambda)
    \end{equation*}
    obtained by applying cohomology to the short exact sequence $0\rightarrow\Lambda/\Lambda'\rightarrow\OO_Y^{\oplus g}/\Lambda'\rightarrow\OO_Y^{\oplus g}/\Lambda\rightarrow 0$ sends $\eta'$ to $\eta$. In particular, we can extend $\varphi\colon A'\rightarrow A$ to a finite \'etale cover
    \begin{equation*}
        (Y\times A')^{\eta'}\longrightarrow(Y\times A)^{\eta}.
    \end{equation*}
\end{lemma}
\begin{proof}
    This follows from the commutative diagram
    \[\begin{tikzcd}
	   {H^1(Y,\mathcal{O}_Y^{\oplus g})} & {H^1(Y,\mathcal{O}_Y^{\oplus g}/\Lambda')} & {H^2(Y,\Lambda')} \\
	   {H^1(Y,\mathcal{O}_Y^{\oplus g})} & {H^1(Y,\mathcal{O}_Y^{\oplus g}/\Lambda)} & {H^2(Y,\Lambda).}
	   \arrow["\mathrm{exp}", from=1-1, to=1-2]
	   \arrow[equals, from=1-1, to=2-1]
	   \arrow["c", from=1-2, to=1-3]
	   \arrow[from=1-2, to=2-2]
	   \arrow[from=1-3, to=2-3]
	   \arrow["\mathrm{exp}", from=2-1, to=2-2]
	   \arrow["c", from=2-2, to=2-3]
    \end{tikzcd}\]
\end{proof}
\begin{definition}\label{definition_extension_of_etale_cover_of_elliptic_curves}
    In the situation of \Cref{lemma_extension_of_etale_cover_of_elliptic_curves}, we call $(Y\times A')^{\eta'}\rightarrow(Y\times A)^{\eta}$ an extension of $\varphi\colon A'\rightarrow A$.
\end{definition}
\subsection{Elliptic fibrations}
If not stated otherwise, the following material can be found in \cite[][1.1]{Nakayama_elliptic-fibrations}.
\begin{definition}
    Let $f\colon X\rightarrow Y$ be a proper surjective morphism between normal complex analytic varieties.
    \begin{enumerate}[label=(\roman*)]
        \item The morphism $f$ is called an elliptic fibration if it has connected fibers and the general fiber is a smooth curve of genus 1.
        \item The discriminant locus of an elliptic fibration is the not necessarily reduced closed complex subspace $\Delta\subset Y$ over which $f$ is not smooth. The smooth locus is denoted by $Y^*:=Y\setminus\Delta$.
        \item The elliptic fibration $f\colon X\rightarrow Y$ is called projective, if there exists an $f$-ample line bundle on $X$. Similarly, $f$ is called locally projective, if there exists an open cover $Y=\bigcup U_i$ such that the elliptic fibrations $f^{-1}(U_i)\rightarrow U_i$ are projective for each $U_i$.
    \end{enumerate}
\end{definition}
In the following, we will only consider elliptic fibrations $f\colon X\rightarrow Y$ over a connected complex manifold $Y$.
\subsubsection{Equivariant Weierstraß models}\label{section_Weierstrass_models}
\begin{definition}\label{definition_Weierstrass_models}\cite[][Def. 1.1, 2]{Nakayama_Weierstrass}, \cite[][3E]{Claudon_Hoering_Lin_fundamental_group}
    Let $Y$ be a connected complex manifold and let $\mathcal{L}$ be a line bundle on $Y$. Choose global sections $\alpha\in H^0(Y,\mathcal{L}^{-4})$ and $\beta\in H^0(Y,\mathcal{L}^{-6})$ such that $4\alpha^3+27\beta^2\neq 0$ in $H^0(Y,\mathcal{L}^{-12})$.
    \begin{enumerate}[label=(\roman*)]
        \item The Weierstraß model $p\colon W(\mathcal{L},\alpha,\beta)\rightarrow Y$ is then defined as
        \begin{equation*}
            W(\mathcal{L},\alpha,\beta):=\{y^2z-(x^3+\alpha xz^2+\beta z^3)=0\}\subset\mathbb{P}:=\mathbb{P}(\OO_Y\oplus\mathcal{L}^2\oplus\mathcal{L}^3)\longrightarrow Y,
        \end{equation*}
        where $x\in H^0(\mathbb{P},\OO_\mathbb{P}(1)\otimes p^*\mathcal{L}^{-2})$, $y\in H^0(\mathbb{P},\OO_\mathbb{P}(1)\otimes p^*\mathcal{L}^{-3})$, $z\in H^0(\mathbb{P},\OO_\mathbb{P}(1))$ are given by the direct sum inclusions $\mathcal{L}^2\hookrightarrow\OO_Y\oplus\mathcal{L}^2\oplus\mathcal{L}^3$, $\mathcal{L}^3\hookrightarrow\OO_Y\oplus\mathcal{L}^2\oplus\mathcal{L}^3$, $\OO_Y\hookrightarrow\OO_Y\oplus\mathcal{L}^2\oplus\mathcal{L}^3$.
        \item The subspace $W(\mathcal{L},\alpha,\beta)^\#\subset W(\mathcal{L},\alpha,\beta)$ is defined as the locus where $p$ is smooth.
        \item The canonical section of $W(\mathcal{L},\alpha,\beta)\rightarrow Y$ is defined as $\Sigma:=\{x=z=0\}$. The sheaf of local meromorphic sections of $p\colon W(\mathcal{L},\alpha,\beta)\rightarrow Y$ is denoted by $\mathcal{W}(\mathcal{L},\alpha,\beta)^{mer}$ and the sheaf of local sections with image in $W(\mathcal{L},\alpha,\beta)^\#$ is denoted by $\mathcal{W}(\mathcal{L},\alpha,\beta)^\#$. Both sheaves are sheaves of abelian groups with addition of local (meromorphic) sections with respect to $\Sigma$.
        \item The Weierstraß model $p\colon W(\mathcal{L},\alpha,\beta)\rightarrow Y$ is called minimal if there is no prime divisor $\Delta$ on $Y$ such that $\mathrm{div}(\alpha)-4\Delta$ and $\mathrm{div}(\beta)-6\Delta$ are both effective divisors on $Y$.
        \item Let $G$ be a finite group. The Weierstraß model $p\colon W(\mathcal{L},\alpha,\beta)\rightarrow Y$ is called $G$-equivariant if $G$ acts on $W(\mathcal{L},\alpha,\beta)$ and on $Y$ such that $p$ is $G$-equivariant and the canonical section $\Sigma$ is fixed by the $G$-action.
    \end{enumerate}
    In the case where $\mathcal{L}$, $\alpha$, and $\beta$ are evident from the context, we omit these symbols in the notation.
\end{definition}
\begin{remark}
    The morphism $p\colon W(\mathcal{L},\alpha,\beta)\rightarrow Y$ is projective, flat, and surjective. Its discriminant locus is given by $\mathrm{div}(4\alpha^3+27\beta^2)$, and the total space $W(\mathcal{L},\alpha,\beta)$ is normal, see \cite[][1.2]{Nakayama_Weierstrass}. In the case that the discriminant locus is a normal crossing divisor, the Weierstraß model $p\colon W(\mathcal{L},\alpha,\beta)\rightarrow Y$ is minimal if and only if $W(\mathcal{L},\alpha,\beta)$ has at worst rational singularities by \cite[][Cor. 2.4]{Nakayama_Weierstrass}. Moreover, there is an isomorphism $R^1p_*\OO_{W(\mathcal{L},\alpha,\beta)}\cong\mathcal{L}$ \cite[][1.2]{Nakayama_Weierstrass}. Suppose that a finite group $G$ acts on $W(\mathcal{L},\alpha,\beta)$ and on $Y$ such that the Weierstraß model $p\colon W(\mathcal{L},\alpha,\beta)\rightarrow Y$ is $G$-equivariant. Then we can apply the flat base change theorem \cite[][III. Prop. 9.3]{Hartshorne_AG} to $R^1p_*\OO_{W(\mathcal{L},\alpha,\beta)}\cong\mathcal{L}$ to get a $G^{op}$-action on $\mathcal{L}$.
\end{remark}
\begin{theorem}\label{bimeromorphic_classification_elliptic_fibration_with_section}\cite[][Thm. 2.5]{Nakayama_Weierstrass}
    Let $f\colon X\rightarrow Y$ be an elliptic fibration over a connected complex manifold $Y$. Suppose that $f$ admits a meromorphic section. Then there exists a unique minimal Weierstraß model $p\colon W(\mathcal{L},\alpha,\beta)\rightarrow Y$ that is bimeromorphic to $f$.
\end{theorem}
\begin{definition}\cite[][3E]{Claudon_Hoering_Lin_fundamental_group}
    Let $Y$ be a connected complex manifold and let $G$ be a finite group acting on $Y$. Let $\Delta\subset Y$ be a $G$-equivariant normal crossing divisor and let $H$ be a $G$-equivariant variation of Hodge structures on $j\colon Y^*:=Y\setminus\Delta\hookrightarrow Y$ associated to some elliptic fiber bundle over $Y^*$. We define
    \begin{equation*}
        \mathcal{E}_G(Y,\Delta,H)
    \end{equation*}
    to be the set of equivalence classes of $G$-equivariant elliptic fibrations $f\colon X\rightarrow Y$ that satisfy the following properties:
    \begin{enumerate}[label=(\roman*)]
        \item $f\colon X\rightarrow Y$ has local meromorphic sections over every point of $Y$,
        \item the elliptic fibration $f^{-1}(Y^*)\rightarrow Y^*$ is bimeromorphic to an elliptic fiber bundle over $Y^*$, and
        \item this elliptic fibration induces $H$ as its $G$-equivariant variation of Hodge structures.
    \end{enumerate}
    Two such elliptic fibrations are equivalent if they are $G$-equivariantly bimeromorphic.
\end{definition}
\begin{lemma}\label{existence_unique_eq_Weierstrass_model}
    Let $Y$ be a connected complex manifold and let $G$ be a finite group acting on $Y$. Let $\Delta\subset Y$ be a $G$-equivariant normal crossing divisor and let $H$ be a $G$-equivariant variation of Hodge structures on $j\colon Y^*:=Y\setminus\Delta\hookrightarrow Y$ associated to some elliptic fiber bundle over $Y^*$. Then there exists a unique minimal Weierstraß model $p\colon W(\mathcal{L},\alpha,\beta)\rightarrow Y$ that is smooth over $Y^*$ and induces $H$ as is $G$-equivariant variation of Hodge structures.
\end{lemma}
\begin{proof}
    Choose a $G$-equivariant $X\rightarrow Y$ compactification of the Jacobian fibration $J_H\rightarrow Y^*$. The zero section of the Jacobian fibration defines a $G$-equivariant meromorphic section of $X\rightarrow Y$. By \Cref{bimeromorphic_classification_elliptic_fibration_with_section} there is a unique minimal Weierstraß model $p\colon W(\mathcal{L},\alpha,\beta)\rightarrow Y$ bimeromorphic to $X\rightarrow Y$. By \cite[][Cor. 2.7]{Nakayama_Weierstrass}, the morphism $p\colon W(\mathcal{L},\alpha,\beta)\rightarrow Y$ is smooth over $Y^*$.\par
    Note that by \Cref{classification_torus_fiber_bundles}, the existence of a section shows that $W(\mathcal{L},\alpha,\beta)\vert_{Y^*}\cong J_H$. We claim that we can extend the $G$-action on $J_H\rightarrow Y^*$ to $p\colon W(\mathcal{L},\alpha,\beta)\rightarrow Y$. Indeed, let $g\in G$ and consider the following fiber product diagram:
    \[\begin{tikzcd}
	   {W(g^*\mathcal{L},g^*\alpha,g^*\beta)} & {W(\mathcal{L},\alpha,\beta)\times_gY} & {W(\mathcal{L},\alpha,\beta)} \\
	   & Y & {Y.}
	   \arrow["\sim", from=1-1, to=1-2]
	   \arrow["g^*p", from=1-1, to=2-2]
	   \arrow[from=1-2, to=1-3]
	   \arrow[from=1-2, to=2-2]
	   \arrow["p", from=1-3, to=2-3]
	   \arrow["g", from=2-2, to=2-3]
    \end{tikzcd}\]
    As $g^*p\colon W(g^*\mathcal{L},g^*\alpha,g^*\beta)\rightarrow Y$ and $p\colon W(\mathcal{L},\alpha,\beta)\rightarrow Y$ are both isomorphic to $J_H\rightarrow Y^*$ over $Y^*$, $g^*p\colon W(g^*\mathcal{L},g^*\alpha,g^*\beta)\rightarrow Y$ and $p\colon W(\mathcal{L},\alpha,\beta)\rightarrow Y$ bimeromorphic Weierstraß models. Note that since $W(\mathcal{L},\alpha,\beta)\rightarrow Y$ is minimal the same holds for $W(g^*\mathcal{L},g^*\alpha,g^*\beta)\rightarrow Y$. Therefore, by \Cref{bimeromorphic_classification_elliptic_fibration_with_section}, the uniqueness of minimal Weierstraß models implies that there is an isomorphism $W(g^*\mathcal{L},g^*\alpha,g^*\beta)\cong W(\mathcal{L},\alpha,\beta)$ over $Y$. Moreover, this isomorphism has to leave the canonical section fixed and restricts over $Y^*$ to the automorphism induced by $g\in G$ of $J_H\rightarrow Y^*$. Thus, we obtain a $G$-action on $W(\mathcal{L},\alpha,\beta)$ by
    \[\begin{tikzcd}
	   {W(\mathcal{L},\alpha,\beta)} & {W(g^*\mathcal{L},g^*\alpha,g^*\beta)} & {W(\mathcal{L},\alpha,\beta)} \\
	   & Y & Y
	   \arrow["\sim", from=1-1, to=1-2]
	   \arrow["p", from=1-1, to=2-2]
	   \arrow["g^*p", from=1-2, to=1-3]
	   \arrow[from=1-2, to=2-2]
	   \arrow["p", from=1-3, to=2-3]
	   \arrow["g", from=2-2, to=2-3]
    \end{tikzcd}\]
    that gives $p\colon W(\mathcal{L},\alpha,\beta)\rightarrow Y$ the structure of a $G$-equivariant minimal Weierstraß model.
\end{proof}
We can construct elements of $\mathcal{E}_G(Y,\Delta,H)$ as follows: Let $G$ be a finite group and let $p\colon W(\mathcal{L},\alpha,\beta)\rightarrow Y$ be a minimal $G$-equivariant Weierstraß model that is smooth over $Y^*$ and has $H$ as its associated variation of Hodge structures, which exists by \Cref{existence_unique_eq_Weierstrass_model}. The subspace $W(\mathcal{L},\alpha,\beta)^\#\subset W(\mathcal{L},\alpha,\beta)$ defines a complex analytic group variety over $Y$ with zero section $\Sigma\subset W(\mathcal{L},\alpha,\beta)^\#$, where the group action is given by the addition of local sections with respect to $\Sigma$. The action of $W(\mathcal{L},\alpha,\beta)^\#$ on $W(\mathcal{L},\alpha,\beta)^\#$ can be extended to an action of $W(\mathcal{L},\alpha,\beta)^\#$ on $W(\mathcal{L},\alpha,\beta)$ \cite[][5.1.1]{Nakayama_global_elliptic_fibration}. By gluing local patches of $p\colon W(\mathcal{L},\alpha,\beta)\rightarrow Y$ along a cocycle of local sections representing the cohomology class $\eta_G\in H_G^1(Y,\mathcal{W}(\mathcal{L},\alpha,\beta)^\#)$, one can thus construct the twisted $G$-equivariant Weierstraß model $p^{\eta_G}\colon W(\mathcal{L},\alpha,\beta)^{\eta_G}\longrightarrow Y$ \cite[][3E]{Claudon_Hoering_Lin_fundamental_group}. This construction defines a map
\begin{equation}\label{twisted_Weierstrass_models}
    H_G^1(Y,\mathcal{W}(\mathcal{L},\alpha,\beta)^\#)\longrightarrow\mathcal{E}_G(Y,\Delta,H).
\end{equation}
This is the analogue of the construction of twisted Jacobian fibration of torus fiber bundles.\par
Recall that for a torus fiber bundle $f\colon X\rightarrow Y$ of fiber dimension $g>0$ with associated variation of Hodge structures $H$ there is the short exact sequence
\begin{equation*}
    0\longrightarrow H_\Z\longrightarrow\mathcal{H}^{g-1,g}\longrightarrow\mathcal{J}_H\longrightarrow 0.
\end{equation*}
\begin{proposition}\label{bimeromorphic_classification_elliptic_fibrations}\cite[][Prop. 2.10]{Nakayama_Weierstrass}\cite[][Prop. 5.5.1]{Nakayama_global_elliptic_fibration}\cite[][3E]{Claudon_Hoering_Lin_fundamental_group}
    Let $p\colon W(\mathcal{L},\alpha,\beta)\rightarrow Y$ be a minimal $G$-equivariant Weierstraß model, such that the discriminant locus $\Delta:=\mathrm{div}(4\alpha^3+27\beta^2)$ is a normal crossing divisor. Denote by $j\colon Y^*\subset Y$ the inclusion and let $H:=R^1p_*\Z\vert_{Y^*}$ be the associated variation of Hodge structures.
    \begin{enumerate}[label=(\roman*)]
        \item There is a $G$-equivariant short exact sequence
        \begin{equation}\label{ses_vhs_line_bundle_W}
            0\longrightarrow j_*H_\Z\longrightarrow\mathcal{L}\longrightarrow\mathcal{W}^\#\longrightarrow 0.
        \end{equation}
        \item There is an injective map
        \begin{equation}\label{injection_into_H1(Wmer)}
            \mathcal{E}_G(Y,\Delta,H)\longhookrightarrow H_G^1(Y,\mathcal{W}^{mer})
        \end{equation}
        such that the composition
        \begin{equation}\label{composition_construction_injection_into_H1(Wmer)}
            H_G^1(Y,\mathcal{W}(\mathcal{L},\alpha,\beta)^\#)\longrightarrow\mathcal{E}_G(Y,\Delta,H)\longhookrightarrow H_G^1(Y,\mathcal{W}^{mer})
        \end{equation}
        equals the map in cohomology induced by the inclusion $\mathcal{W}^\#\subset\mathcal{W}^{mer}$.
    \end{enumerate}
\end{proposition}
Recall that for a torus fiber bundle $f\colon X\rightarrow Y$ over a compact connected K\"ahler manifold $Y$ with associated variation of Hodge structures $H$, the Jacobian fibration $J_H^\eta\rightarrow Y$ twisted by $\eta\in H^1(Y,\mathcal{J}_H)$ is K\"ahler if and only if its Chern class is torsion by \Cref{fundamental_properties_special_torus_bundles}. The following proposition is an analogue of this statement.
\begin{proposition}\label{Weierstrass_model_Kaehler}\cite[][Thm. 3.20, Prop. 3.23]{Claudon_Hoering_Lin_fundamental_group}
    Let $p\colon W\rightarrow Y$ be a minimal Weierstraß model over a compact K\"ahler manifold such that the discriminant locus $\Delta$ is a normal crossing divisor. Denote by $j\colon Y^*:=Y\setminus\Delta\subset Y$ the inclusion and let $H:=R^1p_*\Z\vert_{Y^*}$ be the associated variation of Hodge structures. Let $\eta\in H^1(Y,\mathcal{W}^\#)$ be a cohomology class. Then the following are equivalent:
    \begin{enumerate}[label=(\roman*)]
        \item The total space $W(\mathcal{L},\alpha,\beta)^\eta$ is bimeromorphic to a compact K\"ahler manifold.
        \item The boundary map associated to the short exact sequence $0\rightarrow j_*H_\Z\rightarrow\mathcal{L}\rightarrow\mathcal{W}^\#\rightarrow 0$ sends $\eta$ to a torsion class in $H^2(Y,j_*H_\Z)$.
    \end{enumerate}
\end{proposition}
\subsubsection{Tautological models}
There are classes in $\mathcal{E}_G(Y,\Delta,H)$ that cannot be represented by a twisted Weierstraß model, the reason being that the map $H^1(Y,\mathcal{W}^\#)\rightarrow H^1(Y,\mathcal{W}^{mer})$ may not be surjective. Lin introduced the notion of a tautological model \cite[][3.4]{Lin_algebraic_approximation_codim_1}, which are finite covers of Weierstraß models and showed that every class in $\mathcal{E}_G(Y,\Delta,H)$ that comes from an elliptic fibration $f\colon X\rightarrow Y$, where $X$ is bimeromorphic to a compact K\"ahler manifold, can be represented by a tautological model.
\begin{proposition}\label{tautological_models}\cite[][Lem. 3.19, Lem. 3.24]{Claudon_Hoering_Lin_fundamental_group}\cite[][3.4]{Lin_algebraic_approximation_codim_1}
    Let $f\colon X\rightarrow Y$ be a $G$-equivariant elliptic fibration from a compact complex analytic variety $X$ bimeromorphic to a compact K\"ahler manifold to a connected complex manifold $Y$. Suppose that the discriminant divisor of $f$ is a normal crossing divisor and that $f$ admits local meromorphic sections over every point of $Y$. Denote by $H:=R^1f_*\Z\vert_{Y^*}$ the variation of Hodge structures induced by $f$ and denote by $W(\mathcal{L},\alpha,\beta)\rightarrow Y$ the unique the minimal Weierstraß model associated to $(Y,\Delta,H)$ by \Cref{existence_unique_eq_Weierstrass_model}. Suppose that the zero section $\Sigma\subset W(\mathcal{L},\alpha,\beta)$ is $G$-stable.
    \begin{enumerate}[label=(\roman*)]
        \item Let $\eta_G\in H^1_G(Y,\mathcal{W}^{mer})$ be the image of the class of $f$ under the map \cref{injection_into_H1(Wmer)} $\mathcal{E}_G(Y,\Delta,H)\rightarrow H^1_G(Y,\mathcal{W}^{mer})$. Then there is a positive integer $m$ such that $m\eta_G$ can be lifted to a class in $H^1_G(Y,\mathcal{W}^\#)$.
        \item There is a $G$-equivariant commutative diagram
        \[\begin{tikzcd}
	       {\mathcal{X}} && {W(\mathcal{L},\alpha,\beta)^{m\eta_G}} \\
	       & {Y} && {,}
	       \arrow["m", from=1-1, to=1-3]
	       \arrow["g"', from=1-1, to=2-2]
	       \arrow["p", from=1-3, to=2-2]
        \end{tikzcd}\]
        where $g\colon\mathcal{X}\rightarrow Y$ is an elliptic fibration $G$-equivariantly bimeromorphic to $f\colon X\rightarrow Y$, and $m\colon\mathcal{X}\rightarrow W(\mathcal{L},\alpha,\beta)^{m\eta_G}$ is finite of degree $m$.
        \item The total space of $\mathcal{X}$ is normal and the discriminant locus of $g$ equals the discriminant locus of $p$.
        \item Over the smooth locus of $g$ and $p$, $m$ is the multiplication by $m$-map.
    \end{enumerate}
\end{proposition}
The fibration $g\colon\mathcal{X}\rightarrow Y$ is called the tautological model of $f\colon X\rightarrow Y$. This construction behaves well in families \cite[][3.5]{Lin_algebraic_approximation_codim_1} and under flat pullbacks \cite[][Rem. 3.13]{Lin_algebraic_approximation_codim_1}.
\subsubsection{Equivariant elliptic fiber bundles with trivial monodromy}
The rest of this section is devoted to proving the following proposition.
\begin{proposition}\label{trivial_vhs_constant_j_all_invariant}
    Let $f\colon X\rightarrow Y$ be an elliptic fibration between compact connected complex manifolds, which has local meromorphic sections over every point of $Y$ and is smooth outside a normal crossing divisor $\Delta\subset Y$. Suppose that a finite group $G$ acts on $X$ and $Y$ such that $f$ is $G$-equivariant. Suppose that the local system $\Lambda:=R^1f_*\Z\vert_{Y^*}$ of the variation of Hodge structures induced by $f$ is $G$-equivariantly isomorphic to $\Z^2$, where $\Z^2$ is endowed with the trivial $G$-action, and that the smooth fibers of $f$ are isomorphic to $\C/\Lambda$. Then there is a cohomology class $\eta\in H^1(Y,\OO_Y/\Lambda)^G$ and a group homomorphism $\sigma\colon G\rightarrow E$ such that $f\colon X\rightarrow Y$ is $G$-equivariantly bimeromorphic to $\pi^\eta\colon(Y\times E)^\eta\rightarrow Y$, where $G$ acts diagonally: on $Y$ by the given action and on $E$ by translations with $\sigma(g)$. Moreover, if $X$ and $Y$ are K\"ahler, then $(Y\times E)^\eta$ is also K\"ahler.
\end{proposition}
The following lemma will be useful in the proof.
\begin{lemma}\label{automorphism_smooth_basic_fibration}
    Let $T$ be a connected complex manifold and let $H$ be a variation of Hodge structures on $T$ induced by some elliptic fiber bundle $S\rightarrow T$. Let $\eta\in H^1(T,\mathcal{J}_H)$ be a cohomology class. Then any automorphism of $J_H^\eta$ over $T$ that induces the identity on $R^1\pi^\eta_*\Z\cong H$ is given by translation with a section of $\pi\colon J_H\rightarrow T$.
\end{lemma}
\begin{proof}
    Choose an open cover $\{U_i\}_{i\in I}$ of $T$ and fix isomorphisms $\varphi_i\colon V_i:=(\pi^\eta)^{-1}(U_i)\stackrel{\sim}{\rightarrow}J_{H\vert_{U_i}}$. The composition $\varphi_j\circ\varphi^{-1}_i$ is then equal to translation with a local section $\eta_{ij}$, such that $(\eta_{ij})_{i,j}$ satisfies the cocycle condition and gives rise to the class $\eta$. Let $\psi\colon J_H^\eta\rightarrow J_H^\eta$ be the given automorphism. Then the composites $\psi_i$ given by
    \begin{equation*}
        J_{H\vert_{U_i}}\stackrel{\varphi^{-1}_i}{\longrightarrow} V_i\stackrel{\psi}{\longrightarrow} V_i\stackrel{\varphi_i}{\longrightarrow}J_{H\vert_{U_i}}
    \end{equation*}
    are automorphisms of $J_{H\vert_{U_i}}$ over $U_i$ that induce the identity on the variation of Hodge structures. Thus, by \cite[][Lem. 1.2.1]{Nakayama_elliptic-fibrations} they are given by translations with a section $\psi_i$ of $J_{H\vert_{U_i}}\rightarrow U_i$. Moreover, they fit into the commutative diagram
    \[\begin{tikzcd}
	   {J_{H\vert_{U_{ij}}}} & {J_{H\vert_{U_{ij}}}} \\
	   {J_{H\vert_{U_{ij}}}} & {J_{H\vert_{U_{ij}}},}
	   \arrow["{\psi_i}", from=1-1, to=1-2]
	   \arrow["{\eta_{ij}}"', from=1-1, to=2-1]
	   \arrow["{\eta_{ij}}", from=1-2, to=2-2]
	   \arrow["{\psi_j}"', from=2-1, to=2-2]
    \end{tikzcd}\]
    which shows that $\psi_i$ and $\psi_j$ agree on $U_i\cap U_j$. Therefore, we can glue the local sections $(\psi_i)_{i\in I}$ to obtain a global section of $\pi\colon J_H\rightarrow T$.\par
    Conversely, let $s$ be a global section of $\pi\colon J_H\rightarrow T$. Translation with $s$ induces local automorphisms
    \begin{equation*}
        \psi_i\colon J_{H\vert_{U_i}}\stackrel{s\vert_{U_i}}{\longrightarrow}J_{H\vert_{U_i}}
    \end{equation*}
    over $U_i$. These local automorphisms fit into the commutative diagram
    \[\begin{tikzcd}
	   {J_{H\vert_{U_{ij}}}} & {J_{H\vert_{U_{ij}}}} \\
	   {J_{H\vert_{U_{ij}}}} & {J_{H\vert_{U_{ij}}},}
	   \arrow["{\psi_i}", from=1-1, to=1-2]
	   \arrow["{\eta_{ij}}"', from=1-1, to=2-1]
	   \arrow["{\eta_{ij}}", from=1-2, to=2-2]
	   \arrow["{\psi_j}"', from=2-1, to=2-2]
    \end{tikzcd}\]
    since $s$ is globally defined. Therefore, they glue to an automorphism $\psi$ of $J_H^\eta$ over $T$ that induces the identity on $R^1\pi^\eta_*\Z\cong H$.
\end{proof}
\begin{proof}[Proof of \Cref{trivial_vhs_constant_j_all_invariant}]
    By \Cref{corollary_isotrivial_and_trivial_vhs}, the Jacobian fibration $J_H\rightarrow Y^*$ is isomorphic to $Y^*\times E\rightarrow Y^*$. Let $a$, $b\in\C$ such that
    \begin{equation*}
        E\cong\{y^2z-(x^3+axz^2+bz^3)=0\}\subset\mathbb{P}^2.
    \end{equation*}
    Then $W(\OO_Y,a,b)\cong Y\times E\rightarrow Y$ is the unique minimal Weierstraß model that extends $J_H\rightarrow Y^*$. By \cite[][Lem. 1.3.5]{Nakayama_elliptic-fibrations}, every local meromorphic section of $Y\times E\rightarrow Y$ is holomorphic. Thus, the sheaves $\mathcal{W}^\#$ and $\mathcal{W}^{mer}$ agree, and, by \Cref{corollary_isotrivial_and_trivial_vhs}, they are isomorphic to $\OO_Y/\Lambda$, where $\Lambda\subset\C$ is a rank 2 lattice such that $E\cong\C/\Lambda$. In particular, the composition \cref{composition_construction_injection_into_H1(Wmer)} is of the form
    \begin{equation*}
        H^1_G(Y,\OO_Y/\Lambda)\longrightarrow\mathcal{E}_G(Y,\Delta,H)\longhookrightarrow H^1_G(Y,\OO_Y/\Lambda).
    \end{equation*}
    Therefore, by \Cref{bimeromorphic_classification_elliptic_fibrations}, there is a class $\eta_G\in H^1_G(Y,\OO_Y/\Lambda)$ such that $f\colon X\rightarrow Y$ is $G$-equivariantly bimeromorphic to $(Y\times E)^{\eta_G}\rightarrow Y$. Denote by $\eta\in H^1(Y,\OO_Y/\Lambda)$ the image of $\eta_G$ via the canonical map $H_G^1(Y,\OO_Y/\Lambda)\rightarrow H^1(Y,\OO_Y/\Lambda)$. By combining \Cref{Weierstrass_model_Kaehler} and \Cref{fundamental_properties_special_torus_bundles} we conclude that $(Y\times E)^\eta$ is K\"ahler.\par
    It remains to determine the $G$-action on $(Y\times E)^\eta$ induced by $\eta_G$. The group $G$ acts on $Y\times E$ by acting on the first component. Moreover, as the $G^{op}$-action on $\Lambda$ is trivial, and as the inclusion $\Lambda\subset\OO_Y$ is $G^{op}$-equivariant, the $G^{op}$-action on $\OO_Y$ is trivial as well. By the same reasoning, also the $G^{op}$-action on $\OO_Y/\Lambda$ is trivial. Consider the Hochschild--Serre spectral sequence
    \begin{equation*}
        E^{p,q}_2:=H^2(G^{op},H^q(Y,\OO_Y/\Lambda))\Longrightarrow H^{p+q}_G(Y,\OO_Y/\Lambda).
    \end{equation*}
    For $p+q=1$ we obtain the short exact sequence
    \begin{align*}
        0&\longrightarrow H^1(G^{op},H^0(Y,\OO_Y/\Lambda))\longrightarrow H^1_G(Y,\OO_Y/\Lambda)\\
        &\longrightarrow\ker(H^1(Y,\OO_Y/\Lambda)^G\stackrel{u}{\longrightarrow} H^2(G^{op},H^0(Y,\OO_Y/\Lambda)))\longrightarrow 0.
    \end{align*}
    Note that $u(\eta)=0$. Let us describe the map $u$. We will see that this yields a description of the $G$ action on $(Y\times E)^\eta$. We can choose a good $G$-invariant cover for $Y$, i.e. an open cover $\{U_i\}_{i\in I}$ of $Y$ such that $g(U_i)=U_j$ for some $j\in I$ and such that $g(U_i)$ and $U_i$ are either equal or disjoint. This allows to define a $G^{op}$-action on $I$ by setting $gi$ to be the index of the open subset $g^{-1}(U_i)=U_{gi}$. Let $(\theta_{ij})_{i,j}$ be a cocycle representing a $G$-invariant class $\theta\in H^1(Y,\OO_Y/\Lambda)^G$. Then by definition the two cocycles $\theta_{ij}$ and $g\theta_{(gi)(gj)}$ differ by a coboundary for every $g\in G$: there are tuples $(\lambda_i(g))_{i\in I}$ such that
    \begin{equation*}
        g\theta_{(gi)(gj)}+\lambda_j(g)-\lambda_i(g)=\theta_{ij}
    \end{equation*}
    for all $i$, $j\in I$. For $g=1$ we set $\lambda_i(1):=0$ for all $i\in I$. We define
    \begin{equation*}
        (\nu_i(g,h))_{i\in I}:=(g\lambda_{gi}(h)-\lambda_i(g*h)+\lambda_i(g))_{i\in I},
    \end{equation*}
    where $g*h$ denotes the opposite multiplication. Comparing the equations $h(g\theta_{(gi)(gj)}+\lambda_j(g)-\lambda_i(g))=h(\theta_{ij})$ and $(g*h)\theta_{((g*h)i)((g*h)j)}+\lambda_j(g*h)-\lambda_i(g*h)=\theta_{ij}$ shows
    \begin{equation*}
        \nu_j(g,h)-\nu_i(g,h)=0
    \end{equation*}
    for all $i$, $j\in I$ and all $g$, $h\in G$. Thus, the local sections $\nu_i(g,h)$ glue to a global section $\nu(g,h)\in H^0(Y,\OO_Y/\Lambda)$. The map $G\times G\rightarrow H^0(Y,\OO_Y/\Lambda)$ defined by sending $(g,h)$ to $\nu(g,h)$ satisfies the identity
    \begin{equation*}
        \nu(g,h)+\nu(g*h,k)=g\nu(h,k)+\nu(g,h*k)
    \end{equation*}
    for all $g$, $h$, $k\in G$, i.e $\nu$ is a normalized $G^{op}$ 2-cocycle (the conditions $\nu(g,1)=0=\nu(1,g)$ are satisfies since we chose $\lambda_i(1)=0$). It hence defines a cohomology class $[\nu]\in H^2(G^{op},H^0(Y,\OO_Y/\Lambda))$ \cite[][IV Thm. 3.12]{Brown_group_cohomology}. The map $u\colon H^1(Y,\OO_Y/\Lambda)^G\rightarrow H^2(G^{op},H^0(Y,\OO_Y/\Lambda))$ is then defined by this construction.\par
    We apply this to $\theta:=\eta$. Let $(\eta_{ij})_{i,j}$ be a $G$-invariant cocycle representing the given class $\eta\in H^1(Y,\OO_Y/\Lambda)^G$. Let $\nu'$ be the normalized 2-cocycle constructed from $(\eta_{ij})_{i,j}$ as above. Since $u(\eta)=0$, the 2-cocycle $\nu'$ is trivial. Thus, there is a function $\mu\colon G\rightarrow H^0(Y,\OO_Y/\Lambda)$ such that
    \begin{equation*}
        \nu'(g,h)=g\mu(h)-\mu(g*h)+\mu(g)
    \end{equation*}
    for all $g$, $h\in G$ \cite[][IV Thm. 3.12]{Brown_group_cohomology}. Suppose, for now, that the local morphisms
    \begin{equation}\label{local_morphisms_in_proof}
        (id\times tr(\lambda_i(g)-\mu(g)))\colon U_i\times E\longrightarrow U_i\times E
    \end{equation}
    glue to a morphism $\varphi_g\colon (Y\times E)^\eta\times_gY\rightarrow (Y\times E)^\eta$ over $Y$, where $(Y\times E)^\eta\times_gY$ denotes the fiber product
    \[\begin{tikzcd}
	   {(Y\times E)^\eta\times_gY} & {(Y\times E)^\eta} \\
	   Y & {Y.}
	   \arrow["{\psi_g}", from=1-1, to=1-2]
	   \arrow[from=1-1, to=2-1]
	   \arrow[from=1-2, to=2-2]
	   \arrow["g"', from=2-1, to=2-2]
    \end{tikzcd}\]
    Moreover, we suppose, for now, that the identities
    \begin{align}
        \varphi_h\circ(\varphi_g\times_hY)&=\varphi_{gh}\label{identity1_in_proof}\\
        \psi_h\circ(\psi_g\times_hY)&=\psi_{gh}\label{identity2_in_proof}
    \end{align}
    hold true for all $g$, $h\in G$. Both assertions will be proven in the next lemma. Note that the construction of the morphisms $\varphi_g$ is independent of $\eta_G\in H_G^1(Y,\OO_Y/\Lambda)$.\par
    Denote by $f_{g}$ is the automorphism of $(Y\times E)^\eta$ induced by the $G$-action on $(Y\times E)^\eta$ given by the class $\eta_G\in H_G^1(Y,\OO_Y/\Lambda)$. Consider the commutative diagram
    \[\begin{tikzcd}
	   {(Y\times E)^\eta} & {(Y\times E)^\eta\times_gY} & {(Y\times E)^\eta} & {(Y\times E)^\eta} \\
	   Y & Y & Y & {Y,}
	   \arrow["{\varphi^{-1}_g}", from=1-1, to=1-2]
	   \arrow[from=1-1, to=2-1]
	   \arrow["{\psi_g}", from=1-2, to=1-3]
	   \arrow[from=1-2, to=2-2]
	   \arrow["{f_{g^{-1}}}", from=1-3, to=1-4]
	   \arrow[from=1-3, to=2-3]
	   \arrow[from=1-4, to=2-4]
	   \arrow["\id"', from=2-1, to=2-2]
	   \arrow["g"', from=2-2, to=2-3]
	   \arrow["{g^{-1}}"', from=2-3, to=2-4]
    \end{tikzcd}\]
    Then from \Cref{automorphism_smooth_basic_fibration} it follows that the composition of the upper row equals the translation with a section $\sigma(g)\in H^0(Y,\OO_Y/\Lambda)$. In particular, as the morphisms $\varphi_g$ only depend on the \u{C}ech cocycle for $\eta$ and are independent of $\eta_G$, the morphism $f_g$ is determined by $\sigma(g)$. Therefore, the $G$-action on $(Y\times E)^\eta$ is determined by the assignment $g\mapsto\sigma(g)$. It thus remains to prove that the map $\sigma$ defines a group homomorphism $\sigma\colon G\rightarrow E\subset H^0(Y,\OO_Y/\Lambda)$. To see this, we calculate
    \begin{align*}
        tr(\sigma(g))&=\psi_h\circ(tr(\sigma(g))\times_hY)\circ\psi^{-1}_h\\
        &=\psi_h\circ(f_{g^{-1}}\times_hY)\circ(\psi_g\times_hY)\circ(\varphi^{-1}_g\times_hY)\circ\psi^{-1}_h\\
        &=f_{g^{-1}}\circ\psi_{gh}\circ\varphi^{-1}_{gh}\circ\varphi_h\circ\psi^{-1}_h,
    \end{align*}
    where we used $f_{g^{-1}}\times_hY=\psi^{-1}_h\circ f_{g^{-1}}\circ\psi_h$, \cref{identity1_in_proof} and \cref{identity2_in_proof} in the last equality. Thus,
    \begin{align*}
        tr(\sigma(g))&=f_{g^{-1}}\circ\psi_{gh}\circ\varphi^{-1}_{gh}\circ\varphi_h\circ\psi^{-1}_h\\
        &=f_{g^{-1}}\circ f_{gh}\circ f_{(gh)^{-1}}\circ\psi_{gh}\circ\varphi^{-1}_{gh}\circ\varphi_h\circ\psi^{-1}_h\\
        &=f_{h}\circ tr(\sigma(gh))\circ\varphi_h\circ\psi^{-1}_h\circ f_h\circ f_{h^{-1}}\\
        &=f_h\circ tr(\sigma(gh))\circ tr(-\sigma(h))\circ f_{h^{-1}}\\
        &=tr(\sigma(gh)-\sigma(h)).
    \end{align*}
    To justify the last step, we claim that the $G$-action on $(Y\times E)^\eta$ is locally given by translations with a local section. In this case, as $E$ is abelian and the action on $\OO_Y/\Lambda$ is trivial, $f_h$ commutes with $tr(\sigma(gh)-\sigma(h))$, which justifies the last step. To see that the $G$-action is locally given by translations with local sections, consider the commutative diagram
    \[\begin{tikzcd}
	   {U_i\times E} & {g(U_i)\times E} & {U_i\times E} \\
	   {U_i} & {g(U_i)} & {U_i.}
	   \arrow["{f_g\vert_{U_i}}", from=1-1, to=1-2]
	   \arrow[from=1-1, to=2-1]
	   \arrow["{g^{-1}\times\id}", from=1-2, to=1-3]
	   \arrow[from=1-2, to=2-2]
	   \arrow[from=1-3, to=2-3]
	   \arrow["g"', from=2-1, to=2-2]
	   \arrow["{g^{-1}}"', from=2-2, to=2-3]
    \end{tikzcd}\]
    The upper row is a morphism over $U_i$ that leaves the variation of Hodge structures fixed. By \Cref{automorphism_smooth_basic_fibration}, it is given by translation with a section. To see that $\sigma\colon G\rightarrow H^0(Y,\OO_Y/\Lambda)$ has image in $E\subset H^0(Y,\OO_Y/\Lambda)$, note that since $G$ is finite, the group homomorphism $\sigma$ has image in the subset of torsion sections. Moreover, as the torsion points of a given order are disjoint, $\sigma(g)$ must be a constant section, i.e., $\sigma(g)\in E$.
\end{proof}
\begin{lemma}
    The local morphisms \cref{local_morphisms_in_proof}
    \begin{equation*}
        (id\times tr(\lambda_i(g)-\mu(g)))\colon U_i\times E\longrightarrow U_i\times E
    \end{equation*}
    glue to a morphism $\varphi_g\colon (Y\times E)^\eta\times_gY\rightarrow (Y\times E)^\eta$ over $Y$, where $(Y\times E)^\eta\times_gY$ denotes the fiber product
    \[\begin{tikzcd}
	   {(Y\times E)^\eta\times_gY} & {(Y\times E)^\eta} \\
	   Y & {Y.}
	   \arrow["{\psi_g}", from=1-1, to=1-2]
	   \arrow[from=1-1, to=2-1]
	   \arrow[from=1-2, to=2-2]
	   \arrow["g"', from=2-1, to=2-2]
    \end{tikzcd}\]
    Moreover, the identities \cref{identity1_in_proof}, \cref{identity2_in_proof}
    \begin{align*}
        \varphi_h\circ(\varphi_g\times_hY)=\varphi_{gh}\\
        \psi_h\circ(\psi_g\times_hY)=\psi_{gh}
    \end{align*}
    hold true for all $g$, $h\in G$.
\end{lemma}
\begin{proof}
    The local morphisms glue if and only if the diagram
    \[\begin{tikzcd}
	   {U_{ij}\times E} && {U_{ij}\times E} \\
	   {U_{ij}\times E} && {U_{ij}\times E}
	   \arrow["{\lambda_i(g)-\mu(g)}", from=1-1, to=1-3]
	   \arrow["{g\eta_{(gi)(gj)}}"', from=1-1, to=2-1]
	   \arrow["{\eta_{ij}}", from=1-3, to=2-3]
	   \arrow["{\lambda_j(g)-\mu(g)}"', from=2-1, to=2-3]
    \end{tikzcd}\]
    commutes. This is the case since $g\eta_{(gi)(gj)}+\lambda_j(g)-\lambda_i(g)=\eta_{ij}$ by construction.\par
    Now, note that the first identity \cref{identity1_in_proof} holds if and only if the diagram
    \[\begin{tikzcd}
	   {U_{i}\times E} && {U_{i}\times E} \\
	   {U_i\times E} && {U_{i}\times E}
	   \arrow["{h\lambda_{hi}(g)-h\mu(g)}", from=1-1, to=1-3]
	   \arrow["\id"', from=1-1, to=2-1]
	   \arrow["{\lambda_i(h)-\mu(h)}", from=1-3, to=2-3]
	   \arrow["{\lambda_i(gh)-\mu(gh)}"', from=2-1, to=2-3]
    \end{tikzcd}\]
    commutes. The composite through the upper right-hand corner is given by
    \begin{align*}
        &h\lambda_{hi}(g)+\lambda_i(h)-(h\mu(g)+\mu(h))\\
        =&h\lambda_{hi}(g)-\lambda_i(h*g)+\lambda_i(h)+\lambda_i(h*g)-(h\mu(g)-\mu(h*g)+\mu(h)+\mu(h*g))\\
        =&\nu(h,g)+\lambda(h*g)-(\nu(h,g)+\mu(h*g))\\
        =&\lambda(gh)-\mu(gh).
    \end{align*}
    The second identity \cref{identity2_in_proof} holds because of the universal property of the fiber product.
\end{proof}
\section{Bimeromorphic Geometric Aspects}
Recall that a compact K\"ahler manifold $X$ satisfies condition \ref{conditionC} from \Cref{conjecture_Kotschick} if there is a holomorphic 1-form $\omega\in H^0(X,\Omega^1_X)$ such that for any finite \'etale cover $\tau\colon Y\rightarrow X$, the sequence
\begin{equation*}
    H^{i-1}(Y,\C)\stackrel{\wedge\tau^*\omega}{\longrightarrow}H^i(Y,\C)\stackrel{\wedge\tau^*\omega}{\longrightarrow}H^{i+1}(Y,\C)
\end{equation*}
given by the cup product is exact for all $i$.
\begin{proposition}\label{blow-downs_to_elliptic_center}\cite[][Prop. 3.1]{Hao_Schreieder_3-folds}
    Let $X$ be a smooth compact K\"ahler threefold that admits a 1-form $\omega\in H^0(X,\Omega^1_X)$ such that the complex $(H^\bullet(X,\C),\wedge\omega)$, i.e.
    \begin{equation*}
        \cdots\stackrel{\wedge\omega}{\longrightarrow} H^{i-1}(X,\C)\stackrel{\wedge\omega}{\longrightarrow}H^i(X,\C)\stackrel{\wedge\omega}{\longrightarrow}H^{i+1}(X,\C)\stackrel{\wedge\omega}{\longrightarrow}\cdots
    \end{equation*}
    is exact. If $K_X$ is not nef and $X$ does not carry the structure of a Mori fiber space, then there is a smooth compact K\"ahler threefold $Y$ such that $X$ is the blow-up of $Y$ along an elliptic curve $E\subset Y$. Moreover, if $\omega'\in H^0(Y,\Omega^1_Y)$ denotes the 1-form induced by $\omega$, then $(H^\bullet(Y,\C),\wedge\omega')$ is exact and $\omega'\vert_E$ is non-trivial.
\end{proposition}
\begin{proof}
    By \cite[][Thm. 1.3]{Hoering_Peternell_Kaehler-MMP} and \cite[][Thm. 1.1]{Hoering_Peternell_Kaehler-MFS} extremal contractions exist in the K\"ahler category. As contractions are projective morphisms, they are described by \cite[][Thm. 3.3]{Mori_3-folds_not_nef}. We can hence apply the argument in \cite[][Prop. 3.1]{Hao_Schreieder_3-folds}.
\end{proof}
\begin{corollary}\label{reduction_minimal_model_Mori_fiber_space}\cite[][Cor. 3.2]{Hao_Schreieder_3-folds}
    Let $X$ be a smooth compact K\"ahler threefold which satisfies condition \ref{conditionC} from \Cref{conjecture_Kotschick}. Then there is a sequence of blow-downs
    \begin{equation*}
        X=:X_0\longrightarrow\cdots\longrightarrow X_n
    \end{equation*}
    of smooth compact K\"ahler threefolds $X_i$ along elliptic curves $E_i\subset X_i$ such that $X_n$ is either a minimal model or a Mori fiber space. Moreover, if $\omega\in H^0(X,\Omega^1_X)$ is a holomorphic 1-form such that $(X,\omega)$ satisfies condition \ref{conditionC}, then the induced holomorphic 1-forms $\omega_i\in H^0(X_i,\Omega^1_{X_i})$ restrict non-trivially on the center of blow-ups $X_{i-1}\rightarrow X_i$, and each pair $(X_i,\omega_i)$ satisfies condition \ref{conditionC}.
\end{corollary}
\begin{proposition}\label{reduction_to_bimeromorphic_problem}
    Let $X$ be a smooth minimal K\"ahler threefold which satisfies condition \ref{conditionC} from \Cref{conjecture_Kotschick}. Suppose that there is a smooth compact K\"ahler threefold $Y$ that is bimeromorphic to $X$, a smooth compact K\"ahler surface $S$, an elliptic curve $E=\C/\Lambda$, a cohomology class $\eta\in H^1(S,\OO_S/\Lambda)$, and a finite \'etale cover $\varphi\colon (S\times E)^\eta\rightarrow Y$. Then there is a smooth minimal K\"ahler surface $T$ bimeromorphic to $S$ and a finite \'etale cover
    \begin{equation*}
        \psi\colon (T\times E)^\theta\longrightarrow X,
    \end{equation*}
    where $\theta\in H^1(T,\OO_T/\Lambda)$. Moreover, if $\varphi$ is an isomorphism, then $\psi$ is an isomorphism, and if $\eta=0$, then $\theta=0$.
\end{proposition}
To prove this proposition, we need some preparation.
\begin{lemma}\label{transport_of_twist_along_blow-down}
    Let $S$ be a smooth compact K\"ahler surface and let $\tilde{S}\rightarrow S$ be a blow-up of a point $s\in S$. Let $E=\C/\Lambda$ an elliptic curve and let $\tilde{\eta}\in H^1(\tilde{S},\OO_{\tilde{S}}/\Lambda)$ be a cohomology class such that the twisted Jacobian fibration $\smash{(\tilde{S}\times E)^{\tilde{\eta}}\rightarrow\tilde{S}}$ is K\"ahler. Then there exists a unique class $\eta\in H^1(S,\OO_S/\Lambda)$ such that $(S\times E)^{\eta}$ is K\"ahler and such that 
    \begin{equation*}
        (\tilde{S}\times E)^{\tilde{\eta}}\cong Bl_{\pi^{-1}(s)}(S\times E)^{\eta}.
    \end{equation*}
    In particular, if $\tilde{\eta}=0$, then $\eta=0$.
\end{lemma}
\begin{proof}
    Consider the commutative diagram
    \[\begin{tikzcd}
    	{H^1(S,\Lambda)} & {H^1(S,\mathcal{O}_S)} & {H^1(S,\mathcal{O}_S/\Lambda)} & {H^2(S,\Lambda)} & {H^2(S,\mathcal{O}_S)} \\
    	{H^1(\tilde{S},\Lambda)} & {H^1(\tilde{S},\mathcal{O}_{\tilde{S}})} & {H^1(\tilde{S},\mathcal{O}_{\tilde{S}}/\Lambda)} & {H^2(\tilde{S},\Lambda)} & {H^2(\tilde{S},\mathcal{O}_{\tilde{S}}).}
    	\arrow[from=1-1, to=1-2]
    	\arrow["\cong", from=1-1, to=2-1]
    	\arrow[from=1-2, to=1-3]
    	\arrow["\cong", from=1-2, to=2-2]
    	\arrow[from=1-3, to=1-4]
    	\arrow[from=1-3, to=2-3]
    	\arrow[from=1-4, to=1-5]
    	\arrow[from=1-4, to=2-4]
    	\arrow["\cong", from=1-5, to=2-5]
    	\arrow[from=2-1, to=2-2]
    	\arrow[from=2-2, to=2-3]
    	\arrow[from=2-3, to=2-4]
    	\arrow[from=2-4, to=2-5]
    \end{tikzcd}\]
    From $H^2(\tilde{S},\Z)\cong H^1(S,\Z)\oplus\Z$ \cite[][Thm. 7.31]{Voisin_HT_1} it follows that $H^2(\tilde{S},\Lambda)\cong H^1(S,\Lambda)\oplus\Lambda$. Thus, a diagram chase shows that $H^1(S,\OO_S/\Lambda)\rightarrow H^1(\tilde{S},\OO_{\tilde{S}}/\Lambda)$ is injective. In particular, if $\eta$ exists, then it is unique. Moreover, this also shows that $H^2(S,\Lambda)\rightarrow H^2(\tilde{S},\Lambda)$ is an isomorphism on the torsion subgroups. As $\smash{(\tilde{S}\times E)^{\tilde{\eta}}}$ is Kähler, $c(\tilde{\eta})$ is torsion by \Cref{fundamental_properties_special_torus_bundles}. Another diagram chase then shows that there is a cohomology class $\eta\in H^1(S,\OO_S/\Lambda)$ that is mapped to $\tilde{\eta}$ by the pullback map $H^1(S,\OO_S/\Lambda)\rightarrow H^1(\tilde{S},\OO_{\tilde{S}}/\Lambda)$. Since $H^2(S,\Lambda)\rightarrow H^2(\tilde{S},\Lambda)$ is an isomorphism on the torsion subgroups, the commutative diagram
    \[\begin{tikzcd}
	   {H^1(S,\mathcal{O}_{S}/\Lambda)} & {H^2(S,\Lambda)} \\
	   {H^1(\tilde{S},\mathcal{O}_{\tilde{S}}/\Lambda)} & {H^2(\tilde{S},\Lambda)}
	   \arrow[from=1-1, to=1-2]
	   \arrow[from=1-1, to=2-1]
	   \arrow[from=1-2, to=2-2]
	   \arrow[from=2-1, to=2-2]
    \end{tikzcd}\]
    shows that $c(\eta)$ is torsion. Hence, by \Cref{fundamental_properties_special_torus_bundles}, $(S\times E)^\eta$ is K\"ahler. Consider the natural morphism
    \begin{equation*}
        (\tilde{S}\times E)^{\tilde{\eta}}\longrightarrow (S\times E)^\eta.
    \end{equation*}
    Note that this factors through $Bl_{\pi^{-1}(s)}(S\times E)^{\eta}$ by the universal property of blowing-up. It is easy to see that this defines the desired isomorphism.
\end{proof}
\begin{lemma}\label{etale_endomorphisms}\cite[][Lem. 3.9]{Nakayama_Zhang_etale_endomorphism}
    Let $X$, $Y$ be complex manifolds, let $V$ and $W$ be normal complex analytic varieties, and suppose there is a commutative diagram
    \[\begin{tikzcd}
	   X & Y \\
	   V & W
	   \arrow["f", from=1-1, to=1-2]
	   \arrow["\varphi", from=1-1, to=2-1]
	   \arrow["\psi", from=1-2, to=2-2]
	   \arrow["h", from=2-1, to=2-2]
    \end{tikzcd}\]
    of proper surjective morphisms. If $\varphi$ and $\psi$ have connected fibers, the induced morphism $X\rightarrow V\times_WY$ is an isomorphism on a dense open subset of $V$, $f$ is finite \'etale, $h$ is finite, $R^i\varphi_*\OO_X=0$, and $R^i\psi_*\OO_Y=0$ for all $i>0$, then also $h$ is \'etale.
\end{lemma}
\begin{proof}[Proof of \Cref{reduction_to_bimeromorphic_problem}]
    We first claim that we can assume that $Y$ is minimal. For this, we use the notation $Y_0:=Y$, $S_0:=S$, $\eta_0:=\eta$, $\varphi_0:=\varphi\colon (S_0\times E)^{\eta_0}\rightarrow Y_0$, and $\pi_0\colon (S_0\times E)^{\eta_0}\rightarrow S_0$. Then, by \Cref{blow-downs_to_elliptic_center}, there is a blow-down $b\colon Y_0\rightarrow Y_1$ with an elliptic curve as the center, and $Y_1$ is smooth. The Stein factorization applied to $b\circ\varphi_0$ induces the commutative diagram
    \[\begin{tikzcd}
	   {(S_0\times E)^{\eta_0}} & {X_1} \\
	   {Y_0} & {Y_1,}
	   \arrow["a", from=1-1, to=1-2]
	   \arrow["{\varphi_0}"', from=1-1, to=2-1]
	   \arrow["{\varphi_1}", from=1-2, to=2-2]
	   \arrow["b"', from=2-1, to=2-2]
    \end{tikzcd}\]
    where $\varphi_1$ is finite and $a$ is proper bimeromorphic. Note that $R^ia_*\OO_{(S_0\times C)^{\eta_0}}=0$ and $R^ib_*\OO_{Y_0}=0$ for all $i>0$. This is enough to conclude that $\varphi_1\colon X_1\rightarrow Y_1$ is finite \'etale due to \Cref{etale_endomorphisms}. In particular, $X_1$ is smooth. Let $D\subset Y_1$ be the elliptic curve along which is blown up, and denote by $\tilde{D}\subset X_1$ its pullback to $X_1$. This is a disjoint union of finitely many elliptic curves. Let $\tilde{X}\rightarrow X_1$ be the blow-up along $\tilde{D}$. Then the universal property of blowing-up shows that $a$ factors through $(S_0\times E)^{\eta_0}\stackrel{\sim}{\rightarrow}\tilde{X}$, which must be an isomorphism. Thus, $a\colon (S_0\times E)^{\eta_0}\rightarrow X_1$ is a sequence of blow-ups along elliptic curves. Let $E_i\subset (S_0\times E)^{\eta_0}$ be the exceptional divisor over the elliptic curve $D_i\subset\tilde{D}$. Then $E_i\rightarrow D_i$ is a projective bundle, hence of Kodaira dimension $-\infty$. In particular, $E_i$ cannot dominate $S_0$ via $\pi_0\colon (S_0\times E)^{\eta_0}\rightarrow S_0$. Since $\pi_0$ is equi-dimensional, $E_i$ is mapped to an irreducible curve in $F_i\subset S_0$. The pullback $\pi_0^*F_i\in CH^1((S_0\times E)^{\eta_0})$ equals $E_i$, as $\pi_0$ is an elliptic fiber bundle. Let $\mathcal{N}_i$ be the normal bundle of $D_i\subset X_1$. Then there is a canonical identification $E_i\cong\mathbb{P}(\mathcal{N}_i)$ over $D_i$. Denote by $\zeta_i\in CH^2((S_0\times E)^{\eta_0})$ the push forward of the first Chern class of $\OO_{\mathbb{P}(\mathcal{N}_i)}(1)$. Using the blow-up formula \cite[][Prop. 13.12]{Eisenbud_Harris_3264}, we calculate:
    \begin{align*}
        \pi_0^*(F^2_i)=(\pi_0^*F_i)^2=E^2_i=-\zeta_i.
    \end{align*}
    Note that for a point $x\in F_i$, the pullback $\pi_0^*x\in CH^2((S_0\times E)^{\eta_0})$ is equal to $\zeta_i$, and thus $F^2_i=-1$. Similarly, since $\pi_0^*K_{S_0}=K_{(S_0\times E)^{\eta_0}}$, we have
    \begin{equation*}
        \pi_0^*(K_{S_0}\cdot F_i)=K_{(S_0\times E)^{\eta_0}}\cdot E_i=-\zeta_i,
    \end{equation*}
    and hence $K_{S_0}\cdot F_i=-1$. By \cite[][Lem. 1.1.4]{Matsuki_introduction_mmp}, the curves $F_i$ must be a $(-1)$-curves. Castelnuovo's contractibility criterion then implies that we can blow down the curves $F_i$ to get a contraction $S_0\rightarrow S_1$ to a smooth compact K\"ahler surface. Applying \Cref{transport_of_twist_along_blow-down} to $S_0\rightarrow S_1$ shows that there is a cohomology class $\eta_1\in H^1(S_1,\OO_{S_1}/\Lambda)$ that pulls back to $\eta_1$ such that $(S_1\times E)^{\eta_1}$ is K\"ahler and such that the morphisms $(S_0\times E)^{\eta_0}\rightarrow X_1$ and $(S_0\times E)^{\eta_0}\rightarrow (S_1\times E)^{\eta_1}$ are the same blow-ups. Therefore, there is an isomorphism $X_1\cong (S_1\times E)^{\eta_1}$. Continuing this process, we obtain a sequence of blow-downs $Y_0\rightarrow\cdots\rightarrow Y_n$ such that $Y_n$ is a smooth minimal threefold, a sequence of blow-downs $S_0\rightarrow\cdots\rightarrow S_n$ of smooth compact K\"ahler surfaces, and cohomology classes $\eta_i\in H^1(S_i,\OO_{S_i}/\Lambda)$ for $i=1,\ldots,n$ such that each $(S_i\times E)^{\eta_i}$ is K\"ahler and such that the diagram
    \[\begin{tikzcd}
	   {(S_0\times E)^{\eta_0}} & \cdots & {(S_n\times E)^{\eta_n}} \\
	   {Y_0} & \cdots & {Y_n}
	   \arrow[from=1-1, to=1-2]
	   \arrow["{\varphi_0}"', from=1-1, to=2-1]
	   \arrow[from=1-2, to=1-3]
	   \arrow["{\varphi_n}", from=1-3, to=2-3]
	   \arrow[from=2-1, to=2-2]
	   \arrow[from=2-2, to=2-3]
    \end{tikzcd}\]
    commutes, where the vertical morphisms are finite \'etale. Since the vertical morphisms are of the same degree as $\varphi_0$, we see that if $\varphi_0$ is an isomorphism, then also $\varphi_n$ is an isomorphism. Furthermore, if $\eta=0$, then by \Cref{transport_of_twist_along_blow-down}, $\eta_n=0$. Thus, by replacing $\varphi\colon (S\times E)^\eta\rightarrow Y$ with $\varphi_n\colon (S_n\times E)^{\eta_n}\rightarrow Y_n$ we conclude that it suffices to prove \Cref{reduction_to_bimeromorphic_problem} in the case where $Y$ is minimal.\par
    Suppose that $Y$ is minimal. Let $\tau\colon X'\rightarrow X$ be the unique finite \'etale cover bimeromorphic to the finite \'etale cover $\varphi\colon (S\times E)^\eta\rightarrow Y$. We claim that there is an isomorphism $(S\times E)^\eta\cong X'$. In this case, the composition $\psi\colon (S\times E)^\eta\cong X'\rightarrow X$ is the desired finite \'etale cover. Moreover, if $\varphi$ is an isomorphism, then $\psi$ is also an isomorphism.\par
    To prove the claimed isomorphism, note that as $\varphi\colon (S\times E)^{\eta}\rightarrow Y$, resp.~ $\tau\colon X'\rightarrow X$ is finite \'etale and as $Y$, resp.~ $X$ is minimal, $(S\times E)^\eta$, resp $X'$ is minimal as well. Moreover, since the canonical divisor on $(S\times E)^\eta$ is isomorphic to the pullback of the canonical divisor on $S$, $S$ is also minimal. (Here, the term minimal means that the respective canonical divisor is nef. This should not be confused with the term minimal Weierstraß model.) Therefore, by \cite[][Thm. 4.9]{Kollar_flops}, $(S\times E)^\eta$ and $X'$ are connected by a sequence of flops. It is thus enough to show that $(S\times E)^\eta$ does not admit any non-trivial flops.\par
    Suppose there is a flop $(S\times E)^\eta\rightarrow Z\leftarrow W$. By definition \cite[][Def. 2.1]{Kollar_flops}, the flopping contraction $(S\times E)^\eta\rightarrow Z$ is a projective morphism. We can hence argue similarly as in the proof of \cite[][Thm. 3.4]{Hao_Nowhere_Vanishing_Holomorphic_One-Forms_on_Varieties_of_Kodaira_Codimension_One}. First, note that if the flopping contraction $(S\times E)^\eta\rightarrow Z$ contracts a fiber of $\pi\colon (S\times E)^\eta\rightarrow S$, then it must factor through $S$. In particular, the flopping contraction could not be a bimeromorphic morphism, which is absurd. Therefore, the exceptional locus of the flopping contraction $(S\times E)^\eta\rightarrow Z$ is not contained in the fibers of $\pi\colon (S\times E)^\eta\rightarrow S$. Thus, if we apply the action of $E$ on $(S\times E)^\eta$ given by translation to an irreducible component of the exceptional locus, we sweep up a divisor in $(S\times E)^\eta$. This divisor must be contracted by the flopping contraction $(S\times E)^\eta\rightarrow Z$. It hence is not small, a contradiction.
\end{proof}
\section{Surfaces}
\begin{proposition}\label{structure_thm_surfaces}
    Let $S$ be a smooth compact K\"ahler surface that satisfies condition \ref{conditionC} from \Cref{conjecture_Kotschick}. Then there is a finite \'etale cover $S'\rightarrow S$, a compact connected K\"ahler manifold $Y$, a torus $A=\C^g/\Lambda$, a cohomology class $\eta\in H^1(Y,\OO^{\oplus g}_Y/\Lambda)$ and a smooth morphism
    \begin{equation*}
        S'\longrightarrow (Y\times A)^\eta,
    \end{equation*}
    that is either an isomorphism or a $\mathbb{P}^1$-bundle. Moreover, if $\omega\in H^0(S,\Omega^1_S)$ is a holomorphic 1-form such that $(S,\omega)$ satisfies condition \ref{conditionC}, then the induced holomorphic 1-form on $(Y\times A)^\eta$ restricts non-trivially on the fibers of $(Y\times A)^\eta\rightarrow Y$.
\end{proposition}
\begin{proof}
    By \cite[][Cor. 3.1]{Schreieder_topology}, $S$ is isomorphic to one of the following:
    \begin{enumerate}[label=(\roman*)]
        \item\label{cases_in_proof_minimal_ruled_over_elliptic} a minimal ruled surface over an elliptic curve,
        \item\label{cases_in_proof_2-torus} a 2-torus,
        \item\label{cases_in_proof_smooth_elliptic_fibration} a minimal elliptic surface $f\colon S\rightarrow E$ such that $f$ is smooth and $E$ is elliptic, or
        \item\label{cases_in_proof_singular_elliptic_fibration} a minimal elliptic surface $f\colon S\rightarrow C$ such that all singular fibers are multiples of smooth elliptic curves.
    \end{enumerate}
    In \cref{cases_in_proof_smooth_elliptic_fibration}, $S$ must be bielliptic or a 2-torus, and there is a finite \'etale cover that is a 2-torus.\par
    In \cref{cases_in_proof_singular_elliptic_fibration}, we can exclude the case in which $C$ is elliptic and $f$ is smooth (this case is dealt with in \cref{cases_in_proof_smooth_elliptic_fibration}). In particular, the holomorphic 1-form on $S$ that satisfies condition \ref{conditionC} does not come from $C$. Thus, the Albanese morphism $S\rightarrow Alb(S)$ does not contract the fibers. Therefore, the fibers are mapped to a translate of a fixed elliptic curve in $Alb(S)$, and hence the smooth fibers of $f$ have constant $j$-invariant. As all singular fibers are multiples of smooth elliptic curves, topological base change \cite[\href{https://stacks.math.columbia.edu/tag/09V6}{Tag 09V6}]{stacks-project} shows that $R^1f_*\Z$ is a local system. Let $E$ be the general fiber of $f$. Since the identity component $Aut^0(E)\subset Aut(E)$ acts trivially on $H^1(E,\Z)$, the monodromy of $R^1f_*\Z$ is finite. Hence, there exists a finite \'etale cover $C_1\rightarrow C$ such that $S_1:=S\times_CC_1\rightarrow C_1$ is an elliptic fibration with trivial monodromy and constant j-invariant. By \cite[][Prop. 3.11]{Claudon_Hoering_Lin_fundamental_group}, there is a finite Galois cover $C_2\rightarrow C_1$ with Galois group $G$ such that the induced elliptic fibration $S_2\rightarrow C_2$ from the normalization $S_2$ of $S_1\times_{C_1}C_2$ has local meromorphic sections over every point of $C_2$. Thus, $S_2\rightarrow C_2$ is a smooth elliptic fibration. By \Cref{trivial_vhs_constant_j_all_invariant}, there is an elliptic curve $E'=\C/\Lambda$ and a cohomology class $\eta_G\in H^1_G(C_2,\OO_{C_2}/\Lambda)$ such that $S_2$ is $G$-equivariantly bimeromorphic to $(C_2\times E')^\eta$, where $\eta\in H^1(C_2,\OO_{C_2}/\Lambda)$ denotes the image of $\eta_G$ along $H_G^1(C_2,\OO_{C_2}/\Lambda)\rightarrow H^1(C_2,\OO_{C_2}/\Lambda)$ and $(C_2\times E')^\eta$ is K\"ahler. Moreover, the group $G$ acts diagonally on $(C_2\times E')^\eta$: on $C_2$ the action is already defined and on $E'$ it is given by a group homomorphism $\sigma\colon G\rightarrow E'$. Define $H$ to be the kernel of $\sigma\colon G\rightarrow E'$, and let $C_3:=C_2/H$. Then
    \begin{equation*}
        (C_2\times E')^\eta/H\rightarrow C_3
    \end{equation*}
    is a locally trivial elliptic fiber bundle between compact connected complex manifolds with trivial monodromy. Thus, by \Cref{classification_torus_fiber_bundles}, there is a cohomology class $\tilde{\eta}\in H^1(C_3,\OO_{C_3}/\Lambda)$ such that $(C_2\times E')^\eta/H\cong (C_3\times E')^{\tilde{\eta}}$. Since $(C_2\times E')^{\eta}$ is K\"ahler and $(C_2\times E')^{\eta}\rightarrow (C_3\times E')^{\tilde{\eta}}$ is finite, also $(C_3\times E')^{\tilde{\eta}}$ is K\"ahler. Note that $G/H$ acts fixed-point freely on $(C_3\times E')^{\tilde{\eta}}$. Thus, $(C_3\times E')^{\tilde{\eta}}\rightarrow (C_3\times E')^{\tilde{\eta}}/(G/H)$ is finite \'etale. In particular, $(C_3\times E')^{\tilde{\eta}}/(G/H)\cong (C_2\times E')^\eta/G$ is a smooth minimal K\"ahler surface, which is bimeromorphic to $S_1$ by construction. Therefore, as $(C_3\times E')^{\tilde{\eta}}/(G/H)$ and $S_1$ are minimal of non-negative Kodaira dimension, they are isomorphic. Hence, we obtain the desired finite \'etale cover
    \begin{equation*}
        (C_3\times E')^{\tilde{\eta}}\longrightarrow (C_3\times E')^{\tilde{\eta}}/(G/H)\cong S_1\longrightarrow S.
    \end{equation*}
\end{proof}
\section{Kodaira dimension 2}
\begin{theorem}\label{Kodaira_dim_2_Weierstrass}
    Let $X$ be a smooth minimal K\"ahler threefold of Kodaira dimension 2 that satisfies condition \ref{conditionC} from \Cref{conjecture_Kotschick}. Then there exists a smooth compact K\"ahler surface $S$ of general type, an elliptic curve $E=\C/\Lambda$, a cohomology class $\eta\in H^1(S,\OO_S/\Lambda)$ and a finite \'etale cover
    \begin{equation*}
        (S\times E)^\eta\longrightarrow X.
    \end{equation*}
    Moreover, if $\omega\in H^0(X,\Omega^1_X)$ is a holomorphic 1-form such that $(X,\omega)$ satisfies condition \ref{conditionC}, then the induced holomorphic 1-form on $(S\times E)^\eta$ restricts non-trivially on the fibers of $(S\times E)^\eta\rightarrow S$.
\end{theorem}
\begin{proof}
    We first follow \cite[][Sec. 5.2]{Hao_Schreieder_3-folds}. Let $X$ be a smooth minimal  K\"ahler threefold of Kodaira dimension 2 that satisfies condition \ref{conditionC}. The abundance theorem \cite[][Thm. 1.1]{Campana_Hoering_Peternell_Kaehler-abundance}, \cite{Erratum_Addendum_Kaehler-abundance} allows to construct the Iitaka fibration $f\colon X\rightarrow S$, where $S$ is a normal projective surface. Since \cite[][Lem. 2.5, Prop. 5.3]{Hao_Schreieder_3-folds} also hold in the K\"ahler case (their proofs do not require projectivity of $X$), $f$ has only finitely many singular fibers that are not multiples of an elliptic curve, and the $j$-invariants of the smooth fibers are constant. By \cite[][Prop. 5.8]{Hao_Schreieder_3-folds} we can find a finite \'etale cover $\tau\colon X'\rightarrow X$, such that the elliptic fibration $f'\colon X'\rightarrow S'$ that is induced by $f$ via Stein factorization has trivial monodromy, i.e $R^1f'_*\Z\vert_{S'\setminus \Delta'}\cong\Z^{2}$, where $\Delta'\subset S'$ is the discriminant locus (as before, the proof of \cite[][Prop. 5.8]{Hao_Schreieder_3-folds} does not require projectivity and remains true for cohomology with integral coefficients).\par
    Choose a log-desingularization $S''\rightarrow S'$ of $(S',\Delta')$ and a desingularization $X''\rightarrow X\times_{S'}S''$. Then, by \cite[][Thm. 3.3.3]{Nakayama_elliptic-fibrations}, the elliptic fibration $f''\colon X''\rightarrow S''$ is locally projective. We can therefore, by \cite[][Prop. 3.11]{Claudon_Hoering_Lin_fundamental_group}, find a finite Galois cover $\overline{X}\rightarrow X''$ such that $\overline{X}$ is smooth and the elliptic fibration $\overline{f}\colon\overline{X}\rightarrow\overline{S}$ induced by the Stein factorization has local meromorphic sections over every point in $\overline{S}$. Since the monodromy of $f'$ is trivial and the $j$-invariants of the smooth fibers are constant, the same holds for $\overline{f}$, and $G$ must act trivially on the variation of Hodge structures associated to $\overline{f}$. We can thus apply \Cref{trivial_vhs_constant_j_all_invariant} to conclude that there is an elliptic curve $E=\C/\Lambda$ and a cohomology class $\eta_G\in H_G^1(\overline{S},\OO_{\overline{S}}/\Lambda)$ such that $\overline{f}$ is $G$-equivariantly bimeromorphic to 
    \begin{equation*}
        (\overline{S}\times E)^\eta\longrightarrow\overline{S},
    \end{equation*}
    where $\eta$ denotes the image of $\eta_G$ along $H_G^1(\overline{S},\OO_{\overline{S}}/\Lambda)\rightarrow H^1(\overline{S},\OO_{\overline{S}}/\Lambda)$ and $(\overline{S}\times E)^\eta$ is K\"ahler. Moreover, the group $G$ acts diagonally: on $\overline{S}$ the action is already defined and on $E$ it is given by a group homomorphism $\sigma\colon G\rightarrow E$. Define $H:=\ker(\sigma)$ and $T:=\overline{S}/H$. Note that $T$ is smooth. Then $(\overline{S}\times E)^\eta/H\rightarrow T$ is a locally trivial elliptic fiber bundle between complex manifolds with trivial monodromy. Thus, by \Cref{classification_torus_fiber_bundles}, there is a cohomology class $\eta'\in H^1(T,\OO_{T}/\Lambda)$ such that $(\overline{S}\times E)^\eta/H\cong (T\times E)^{\eta'}$. Since $(\overline{S}\times E)^{\eta}$ is K\"ahler and $(\overline{S}\times E)^{\eta}\rightarrow (T\times E')^{\eta'}$ is finite, also $(T\times E')^{\eta'}$ is K\"ahler. Define $Y:=(T\times E)^{\eta'}/(G/H)\cong (\overline{S}\times E)^\eta/G$. Note that the quotient group $G/H$ acts fixed-point freely on $(T\times E)^{\eta'}$. Thus, the quotient map $\varphi\colon (T\times E)^{\eta'}\rightarrow Y$ is finite \'etale and hence $Y$ is a smooth compact K\"ahler threefold. Moreover, by construction $Y$ is bimeromorphic to $X$. Thus, by \Cref{reduction_to_bimeromorphic_problem} there exists a smooth minimal K\"ahler surface $T'$ bimeromorphic to $T$, a cohomology class $\eta''\in H^1(T',\OO_{T'}/\Lambda)$ and a finite \'etale cover $(T'\times E)^{\eta''}\rightarrow X'$. Therefore,
    \begin{equation*}
        (T'\times E)^{\eta''}\longrightarrow X'\longrightarrow X
    \end{equation*}
    is the desired finite \'etale cover.
\end{proof}
\section{Kodaira dimension 1}
\begin{theorem}\label{Kodaira_dim_1_Weierstrass}
    Let $X$ be a smooth minimal K\"ahler threefold of Kodaira dimension 1 that satisfies condition \ref{conditionC} from \Cref{conjecture_Kotschick}. Then one of the two possibilities holds true.
    \begin{enumerate}[label=(\roman*)]
        \item There is a smooth compact curve $C$ of genus at least 2, a 2-torus $A=\C^2/\Lambda$, a cohomology class $\eta\in H^1(C,\OO^{\oplus 2}_C/\Lambda)$ and a finite \'etale cover
        \begin{equation*}
            (C\times A)^\eta\longrightarrow X.
        \end{equation*}
        \item There exists a smooth compact K\"ahler surface $S$ of Kodaira dimension 1 that does not satisfy condition \ref{conditionC}, an elliptic curve $E=\C/\Lambda$, a cohomology class $\eta\in H^1(S,\OO_S/\Lambda)$ and a finite \'etale cover
        \begin{equation*}
            (S\times E)^\eta\longrightarrow X.
        \end{equation*}
    \end{enumerate}
    Moreover, if $\omega\in H^0(X,\Omega^1_X)$ is a holomorphic 1-form such that $(X,\omega)$ satisfies condition \ref{conditionC}, then the induced holomorphic 1-form on $(C\times A)^\eta$, resp.~ $(S\times E)^\eta$ restricts non-trivially on the fibers of $(C\times A)^\eta\rightarrow C$, resp.~ $(S\times E)^\eta\rightarrow S$.
\end{theorem}
Let $X$ be a smooth minimal K\"ahler threefold of Kodaira dimension 1 that satisfies condition \ref{conditionC} from \Cref{conjecture_Kotschick}. The abundance theorem \cite[][Thm. 1.1]{Campana_Hoering_Peternell_Kaehler-abundance}, \cite{Erratum_Addendum_Kaehler-abundance} allows to construct the Iitaka fibration $f\colon X\rightarrow C$, where $C$ is a smooth compact curve.
\begin{lemma}\label{smooth_fibers_2-tori_bielliptic}\cite[][Lem. 6.1]{Hao_Schreieder_3-folds}
    Let $X$ be a smooth minimal K\"ahler threefold of Kodaira dimension 1 that satisfies condition \ref{conditionC} from \Cref{conjecture_Kotschick}. Let $f\colon X\rightarrow C$ be the Iitaka fibration, where $C$ is a smooth compact curve. Then the smooth fibers of $f$ are 2-tori or bielliptic surfaces.
\end{lemma}
\begin{proof}
    Note that the proofs of \cite[][Lem. 2.5, Lem. 6.1]{Hao_Schreieder_3-folds} also hold in the K\"ahler case as they do not require projectivity of $X$.
\end{proof}
The multiple fibers of the Iitaka fibration $f\colon X\rightarrow C$ define an effective divisor $D$ on $C$, i.e., an orbifold structure on $C$. We call the orbifold structure good if there exists a finite orbifold cover with trivial orbifold structure. The classification of orbifolds is the following:
\begin{theorem}\label{classification_orbifolds_curves}\cite[][Cor. 2.29]{Cooper_Hodgson_Kerckhoff_orbifolds}
    Let $C$ be a smooth compact curve, and let $D$ be an effective divisor on $C$. The orbifold structure induced by $D$ on $C$ is good except for the cases $C\cong\mathbb{P}^1$ and either $D$ consists of one point with some multiplicity or $D$ consists of two points with different multiplicities. In particular, every orbifold structure on a curve with positive genus is good.
\end{theorem}
\subsection{The general fiber is not contracted via the Albanese morphism}\label{Iitaka_fiber_not_contracted}
Let $X$ be a smooth minimal K\"ahler threefold of Kodaira dimension 1 and let $f\colon X\rightarrow C$ be the Iitaka fibration. The goal of this section is to prove \Cref{Kodaira_dim_1_Weierstrass} under the following assumption:
\begin{assumption}\label{assumption_no_contraction}
    The general fiber of the Iitaka fibration $f\colon X\rightarrow C$ is not contracted by the Albanese morphism $X\rightarrow Alb(X)$.
\end{assumption}
\begin{definition}\cite[][Def. 1.9]{Campana_Peternell_2-forms}
    Let $X$ be a smooth compact K\"ahler threefold with a surjective morphism $f\colon X\rightarrow C$ to a smooth curve $C$. A holomorphic 2-form $\xi$ on $X$ is called vertical with respect to $f$ if it is in the kernel of the natural map $\Omega^2_X\rightarrow\Omega^2_{X/C}$.
\end{definition}
\begin{theorem}\label{smooth_fibers_after_cover}\cite[][Thm. 4.2]{Campana_Peternell_2-forms}
    Let $X$ be a smooth minimal K\"ahler threefold of Kodaira dimension 1 and let $f\colon X\rightarrow C$ be its Iitaka fibration. If there is a holomorphic 2-form $\xi$ on $X$ which is not vertical with respect to $f$, then there is a finite morphism $C'\rightarrow C$ such that the induced morphism $f'\colon X'\rightarrow C'$ is smooth, where $X'$ is the normalization of $X\times_CC'$. In particular, $f$ is quasi-smooth. Moreover, the fibers of $f'$ are either K3 surfaces of abelian surfaces.
\end{theorem}
\begin{proof}[Proof of \Cref{Kodaira_dim_1_Weierstrass} under \Cref{assumption_no_contraction}]
    Let $X$ be a smooth minimal K\"ahler threefold of Kodaira dimension 1 that satisfies condition \ref{conditionC} from \Cref{conjecture_Kotschick}. Let $f\colon X\rightarrow C$ be the Iitaka fibration, where $C$ is a smooth compact curve. We will prove \Cref{Kodaira_dim_1_Weierstrass} in the case where the general fiber of $f$ is not contracted by the Albanese morphism.\par
    We first follow the argument in \cite[][Sec. 6.2]{Hao_Schreieder_3-folds}: Note that since the Albanese morphism does not contract a general fiber of $f$, the smooth fibers of $f$ must be 2-tori by \Cref{smooth_fibers_2-tori_bielliptic}. Hence, the Albanese morphism maps a general fiber of $f$ to a translate of a fixed 2-dimensional subtorus of $Alb(X)$. Thus, the pullback of a general 2-form on $Alb(X)$ to $X$ is not vertical with respect to $f$. By \Cref{smooth_fibers_after_cover} the Iitaka fibration $f\colon X\rightarrow C$ is quasi-smooth. Let $D\subset C$ be the divisor induced by the multiple fibers. By \cite[][Sec. 6.2]{Hao_Schreieder_3-folds}, the orbifold structure of $D\subset C$ is good. Let $C'\rightarrow C$ be a finite orbifold cover with trivial orbifold structure. Define $X'$ as the normalization of $X\times_CC'$ and note that the natural map $f'\colon X'\rightarrow C'$ is smooth. Moreover, a local analysis shows that $X'\rightarrow X$ is finite \'etale.\par
    We then follow the argument in \cite[][Prop. 6.4]{Hao_Schreieder_3-folds}: Note that $R^1f'_*\Z$ is a local system. Let $F$ be a fiber of $f'$ and note that $Aut(F)$ acts with finite quotient on $H^1(F,\Z)$. Hence, the local system $R^1f'_*\Z$ has finite monodromy. There is therefore a finite \'etale cover $C''\rightarrow C'$ such that the monodromy of $R^1f'_*\Z$ becomes trivial after pulling back to $C''$. The induced finite \'etale cover $X'':=X'\times_{C'}X'\rightarrow X'$ comes thus with a smooth fibration $f''\colon X''\rightarrow C''$ such that the local system $R^1f''_*\Z$ has trivial monodromy.\par
    By relative Poincar\'e duality, the relative cup-product
    \begin{equation*}
        R^1f''_*\Z\otimes R^3f''_*\Z\stackrel{\smile}{\longrightarrow}R^4f''_*\Z\cong\Z
    \end{equation*}
    is an isomorphism. Therefore, the monodromy of $R^3f''_*\Z$ is trivial as well. Recall that the Albanese morphism $X\rightarrow Alb(X)$ maps the fibers of $f$ to translates of a fixed subtorus in $Alb(X)$. Thus, also the fibers of $f''$ are mapped to translates of a fixed subtorus $A\subset Alb(X'')$. Therefore, the fibers of $f''$ are all isomorphic to a fixed \'etale cover $F\rightarrow A$. In particular, $f''$ is isotrivial. Hence, by writing $F=\C^2/\Lambda$ \Cref{corollary_isotrivial_and_trivial_vhs} implies that there is a cohomology class $\eta\in H^1(C'',\OO^{\oplus 2}_{C''}/\Lambda)$ such that
    \begin{equation*}
        X''\cong (C''\times F)^\eta\longrightarrow C''.
    \end{equation*}
\end{proof}
\subsection{The general fiber is contracted to a curve via the Albanese morphism}\label{Iitaka_fiber_contracted_to_curve}
Let $X$ be a smooth minimal K\"ahler threefold of Kodaira dimension 1 and let $f\colon X\rightarrow C$ be the Iitaka fibration. The goal of this section is to prove \Cref{Kodaira_dim_1_Weierstrass} under the following assumption:
\begin{assumption}\label{assumption}
    For every finite \'etale cover $\tau\colon\tilde{X}\rightarrow X$ the general fiber of the induced fibration $\tilde{f}\colon\tilde{X}\rightarrow\tilde{C}$ by the Stein factorization is contracted to a curve by the Albanese morphism, i.e. to a translate of a fixed elliptic curve $E\subset Alb(X)$.
\end{assumption}
The first step is to prove the following proposition:
\begin{proposition}\label{key_proposition_kodaira_1_G-equivariant}
    Let $X$ be a smooth minimal K\"ahler threefold of Kodaira dimension 1 and let $f\colon X\rightarrow C$ be the Iitaka fibration. Suppose that \Cref{assumption} holds true. Then there is a finite group $G$ and a $G$-equivariant commutative diagram of compact connected K\"ahler manifolds
    \[\begin{tikzcd}
	   {X'} && {(C'\times E')^{\eta'}} \\
	   & {C'} && {,}
	   \arrow["{g'}", from=1-1, to=1-3]
	   \arrow["{f'}"', from=1-1, to=2-2]
	   \arrow["{\pi'}", from=1-3, to=2-2]
    \end{tikzcd}\]
    such that the following holds:
    \begin{enumerate}[label=(\roman*)]
        \item\label{item1_key_proposition_kodaira_1_G-equivariant} There is a finite \'etale cover $X'/G\rightarrow X$ and a finite cover $C'/G\rightarrow C$.
        \item\label{item2_key_proposition_kodaira_1_G-equivariant} $E'=\C/\Lambda'$ is an elliptic curve and $\eta'\in H^1_G(C',\OO_{C'}/\Lambda')$ is a cohomology class.
        \item\label{item3_key_proposition_kodaira_1_G-equivariant} The morphism $g'$ is a locally projective elliptic fibration with local meromorphic sections over every point of $(C'\times E')^{\eta'}$, and its discriminant divisor is a finite union of fibers of $\pi'$.
        \item\label{item4_key_proposition_kodaira_1_G-equivariant} The general fiber of $f'$ is a 2-torus.
        \item\label{item5_key_proposition_kodaira_1_G-equivariant} $G$ is either trivial or $C'/G\cong\mathbb{P}^1$.
        \item\label{item6_key_proposition_kodaira_1_G-equivariant} If $C\cong\mathbb{P}^1$ (for example if $G$ is non-trivial) then $\eta'=0$.
    \end{enumerate}
\end{proposition}
The proof of \Cref{key_proposition_kodaira_1_G-equivariant} in the case where the genus of $C$ is positive requires some preparation.
\begin{construction}\label{construction_for_key_proposition_kodaira_1_G-equivariant}
    Let $X$ be a smooth minimal K\"ahler threefold of Kodaira dimension 1 and let $f\colon X\rightarrow C$ be the Iitaka fibration. Suppose that \Cref{assumption} holds true. Suppose furthermore that the genus of $C$ is positive. Let $Y\rightarrow Alb(X)$ be the normalization of the image of the Albanese morphism $X\rightarrow Alb(X)$. Note that this comes with a morphism to the Jacobian of $C$, which has image equal to the image of the Jacobian embedding. Therefore, we obtain a morphism $Y\rightarrow C$. Since the general fiber is isomorphic to $E$, the morphism $Y\rightarrow C$ is an elliptic fibration. The universal property of the normalization induces the commutative diagram
    \[\begin{tikzcd}
	   X && Y \\
	   & C && {.}
	   \arrow[from=1-1, to=1-3]
	   \arrow["f"', from=1-1, to=2-2]
	   \arrow[from=1-3, to=2-2]
    \end{tikzcd}\]
\end{construction}
Recall that a finite morphism $h\colon Z'\rightarrow Z$ between normal complex analytic varieties is called quasi-\'etale if it is \'etale in codimension 1 \cite[][Def. 1.1]{Catanese_quasi_etale}. If $Z$ is smooth, then this implies that $f$ is \'etale \cite[][Rem. 3.1]{Catanese_quasi_etale}.
\begin{lemma}\label{step0_key_proposition_kodaira_1_G_equivariant}
    In the notation of \Cref{construction_for_key_proposition_kodaira_1_G-equivariant}, there is a finite cover $C_0\rightarrow C$, such that $Y_0\rightarrow C_0$ has no multiple fibers, where $Y_0$ is defined as the normalization of $Y\times_CC_0$. Moreover, $Y_0\rightarrow Y$ is finite quasi-\'etale and $X_0:=X\times_YY_0\rightarrow X$ is finite \'etale. 
\end{lemma}
\begin{proof}
    The multiple fibers of $Y\rightarrow C$ induce an orbifold structure on $C$. Since the genus of $C$ is positive, there is a finite orbifold cover $C_0\rightarrow C$ with trivial orbifold structure by \Cref{classification_orbifolds_curves}. Let $Y_0$ be the normalization of $Y\times_CC_0$. A local analysis shows that $Y_0\rightarrow C_0$ has no multiple fibers. The finite morphism $Y_0\rightarrow Y$ is \'etale away from the singular locus of $Y$. Since $Y$ is normal, the singular locus has codimension at least 2. Thus, as $X_0$ is defined as the fiber product $X_0:=X\times_YY_0$, also $X_0\rightarrow X$ is finite quasi-\'etale. Since $X$ is smooth, this implies that $X_0\rightarrow X$ is in fact \'etale.
\end{proof}
\begin{lemma}\label{step1_key_proposition_kodaira_1_G_equivariant}
    In the notation of \Cref{step0_key_proposition_kodaira_1_G_equivariant}, there is a finite \'etale cover $C_1\rightarrow C_0$, such that $Y_1:=Y_0\times_{C_0}C_1\rightarrow C_1$ is an elliptic fibration with trivial monodromy. Moreover, $Y_1\rightarrow Y_0$ and $X_1:=X_0\times_{Y_0}Y_1\rightarrow X_0$ are both finite \'etale.
\end{lemma}
\begin{proof}
    We argue as in \cite[][Prop. 5.8]{Hao_Schreieder_3-folds}. The fibers of $f$ are connected. Thus, so are the fibers of $\pi_0\colon Y_0\rightarrow C_0$, and hence it is an elliptic fibration. Note that the $j$-invariant of the general fiber of $Y\rightarrow C$ is constant. Therefore, the same holds for $\pi_0\colon Y_0\rightarrow C_0$. Let $E_0$ be a general fiber of $\pi_0$, let $\Delta\hookrightarrow C_0$ be the discriminant divisor, and let $j\colon U:=C_0\setminus\Delta\subset C_0$ be the inclusion. Define $H:=R^1(\pi_0)_*\Z\vert_U$. The identity component $Aut^0(E_0)\subset Aut(E_0)$ acts trivially on $H^1(E_0,\Z)$. Therefore, $j_*H$ has finite monodromy. Thus, there is a finite \'etale cover $C_1\rightarrow C_0$ such that $Y_1:=Y_0\times_{C_0}C_1\rightarrow C_1$ has trivial monodromy. Since \'etale morphisms are stable under fiber products, also $Y_1\rightarrow Y_0$ and $X_1\rightarrow X_0$ are finite \'etale.
\end{proof}
\begin{lemma}\label{bimeromorphic_to_elliptic_torsor_over_curve_implies_smooth}
    In the notation of \Cref{step0_key_proposition_kodaira_1_G_equivariant} and \Cref{step1_key_proposition_kodaira_1_G_equivariant}, there is an isomorphism $Y_1\cong(C_1\times E_1)^{\eta_1}$, where $E_1=\C/\Lambda_1$ is an elliptic curve and $\eta_1\in H^1(C_1,\OO_{C_1}/\Lambda_1)$ is a cohomology class such that $(C_1\times E_1)^{\eta_1}$ is K\"ahler.
\end{lemma}
\begin{proof}
    Consider the following commutative diagram
    \[\begin{tikzcd}
	   & {A_1} && {A_0} && {Alb(X)} \\
	   {Y_1} && {Y_0} && Y \\
	   & {Alb(C_1)} && {Alb(C_0)} && {Alb(C)} \\
	   {C_1} && {C_0} && {C,}
	   \arrow[from=1-2, to=1-4]
	   \arrow[from=1-2, to=3-2]
	   \arrow[from=1-4, to=1-6]
	   \arrow[from=1-4, to=3-4]
	   \arrow[from=1-6, to=3-6]
	   \arrow[from=2-1, to=1-2]
	   \arrow[from=2-1, to=2-3]
	   \arrow[from=2-1, to=4-1]
	   \arrow[from=2-3, to=1-4]
	   \arrow[from=2-3, to=2-5]
	   \arrow[from=2-3, to=4-3]
	   \arrow[from=2-5, to=1-6]
	   \arrow[from=2-5, to=4-5]
	   \arrow[from=3-2, to=3-4]
	   \arrow[from=3-4, to=3-6]
	   \arrow[from=4-1, to=3-2]
	   \arrow[from=4-1, to=4-3]
	   \arrow[from=4-3, to=3-4]
	   \arrow[from=4-3, to=4-5]
	   \arrow[from=4-5, to=3-6]
    \end{tikzcd}\]
    where $A_0:=Alb(C_0)\times_{Alb(C)}Alb(X)$ and $A_1:=Alb(C_1)\times_{Alb(C_0)}A_0$. Note that $A_0$ and $A_1$ are tori. By construction of $Y$, the fibers of $Y\rightarrow Alb(X)$ are finite. Since $Y_0\rightarrow Y$ is a finite morphism, the fibers of $Y_0\rightarrow A_0$ are also finite. By the same reasoning, also the fibers of $Y_1\rightarrow A_1$ are finite. Therefore, $Y_1$ cannot contain any rational curve.\par
    Let $\mu\colon\tilde{Y}\rightarrow Y_1$ be a desingularization, obtained by blowing up smooth points in the singular locus, and such that $\mu$ does not contract any $(-1)$-curves. Note that by construction, $Y$ is K\"ahler and hence so is $\tilde{Y}$. By \Cref{trivial_vhs_constant_j_all_invariant}, the elliptic fibration $\tilde{Y}\rightarrow C_1$ is bimeromorphic to $\pi_1\colon (C_1\times E_1)^{\eta_1}\rightarrow C_1$ for some $\eta_1\in H^1(C_1,\OO_{C_1}/\Lambda_1)$ such that $(C_1\times E_1)^{\eta_1}$ is K\"ahler, where $E_1\cong\C/\Lambda_1$. The canonical bundle of $(C_1\times E_1)^{\eta_1}$ is given by the pullback of the canonical bundle of $C_1$. Since the genus of $C_1$ is positive, the canonical bundle of $(C_1\times E_1)^{\eta_1}$ is nef. Therefore, $(C_1\times E_1)^{\eta_1}$ is the unique minimal model of $\tilde{Y}$. We obtain a commutative diagram of the form
    \[\begin{tikzcd}
	   {\tilde{Y}} & {(C_1\times E_1)^{\eta_1}} \\
	   {Y_1} & {C_1,}
	   \arrow["b", from=1-1, to=1-2]
	   \arrow["\mu"', from=1-1, to=2-1]
	   \arrow["{\pi_1}", from=1-2, to=2-2]
	   \arrow["f_1", from=2-1, to=2-2]
    \end{tikzcd}\]
    where $b$ is a finite sequence of blow-ups of points.\par
    To show that $Y_1$ is isomorphic to $(C_1\times E_1)^{\eta_1}$, let $y\in Y$ be a singular point and denote its image in $C_1$ by $c$. We calculate the fiber $\tilde{Y}_c$ of $\tilde{Y}\rightarrow C_1$ in two different ways.\par
    On the one hand, let $\tilde{F}$ be the strict transform of the fiber $(Y_1)_c$ over $c$ in $\tilde{Y}$. Then $\tilde{Y}_c$ is given by the union of $\tilde{F}$ and the irreducible components of the exceptional divisor of $\mu$ lying over $\tilde{F}$. Since $y$ is a singular point, the fiber of $\mu$ over $y$ contains at least one component of the exceptional divisor of $\mu$. Hence, $\tilde{Y}_c$ is not irreducible. On the other hand, by the commutativity of the diagram, $\tilde{Y}_c$ can be explicitly calculated as follows: let $\tilde{E}\subset\tilde{Y}$ be the strict transform of $\pi^{-1}_1(c)\cong E_1$ in $\tilde{Y}$. Then $\tilde{Y}_c$ is given by a tree of curves, where one curve is $\tilde{E}$, all other curves are isomorphic to $\mathbb{P}^1$ and the leaves of this tree are either $(-1)$-curves or $\tilde{E}$. As $\tilde{Y}_c$ is not irreducible, it contains at least one $(-1)$-curve. However, as we saw above, $Y_1$ does not contain any rational curves. Thus, $\mu$ must contract all rational curves in $\tilde{Y}_c$. In particular, $\mu$ contracts at least one $(-1)$-curve. This is a contradiction. Therefore, $Y_1$ is smooth and $\mu$ is an isomorphism. Moreover, $b$ is an isomorphism as $Y_1$ does not contain any rational curves.
\end{proof}
\begin{lemma}\label{step2_key_proposition_kodaira_1_G_equivariant}
    In the notation of \Cref{step1_key_proposition_kodaira_1_G_equivariant} and \Cref{bimeromorphic_to_elliptic_torsor_over_curve_implies_smooth}, let $Z\rightarrow(C_1\times E_1)^{\eta_1}$ be the finite morphism in the Stein factorization of $X_1\rightarrow Y_1\cong (C_1\times E_1)^{\eta_1}$. Then there exists a finite cover $h\colon C_2\rightarrow C_1$ and an elliptic curve $E_2=\C/\Lambda_2$ that admits a finite \'etale cover $E_2\rightarrow E_1$ such that the normalization of
    \begin{equation*}
        Z\times_{(C_1\times E_1)^{\eta_1}}(C_2\times E_1)^{h^*\eta_1}
    \end{equation*}
    is isomorphic to $(C_2\times E_2)^{\eta_2}$ for some cohomology class $\eta_2\in H^1(C_2,\OO_{C_2}/\Lambda_2)$ such that $(C_2\times E_2)^{\eta_2}$ is K\"ahler. Moreover, the natural morphism
    \begin{equation*}
        X_2:=X_1\times_Z(C_2\times E_2)^{\eta_2}\longrightarrow X_1
    \end{equation*}
    is finite \'etale and the morphism $g_2\colon X_2\rightarrow (C_2\times E_2)^{\eta_2}$ is an elliptic fibration.
\end{lemma}
\begin{proof}
    To construct the finite cover $h\colon C_2\rightarrow C_1$, consider the commutative diagram
    \[\begin{tikzcd}
	   {X_1} & Z & {(C_1\times E_1)^{\eta_1}} \\
	   & {C_1} & {.}
	   \arrow[from=1-1, to=1-2]
	   \arrow["{f_1}"', from=1-1, to=2-2]
	   \arrow["r", from=1-2, to=1-3]
	   \arrow["t", from=1-2, to=2-2]
	   \arrow["{\pi_1}", from=1-3, to=2-2]
    \end{tikzcd}\]
    Choose a point $c\in C_1$ such that $X_c:=f_1^{-1}(c)$ is smooth. Define $Z_c:=t^{-1}(c)$. Since $f_1$ has connected fibers, the same holds for $t$. In particular, $Z_c$ is a smooth compact curve. Since $X_c$ is a 2-torus or bielliptic, $Z_c$ has genus at most 1. On the other hand, as $Z_c$ admits a finite morphism to an elliptic curve, it must have genus equal to 1, and $Z_c\rightarrow\pi^{-1}_1(c)\cong E_1$ is finite \'etale. Therefore, the branch locus of $r$ consists of a finite union of multiples of fibers of $\pi_1$. We hence obtain an orbifold structure on $C_1$. By \Cref{classification_orbifolds_curves}, we can define $h$ to be an orbifold cover $h\colon C_2\rightarrow C_1$ with trivial orbifold structure.\par
    Denote by $Z'$ the normalization of $Z\times_{(C_1\times E_1)^{\eta_1}}(C_2\times E_1)^{h^*\eta_1}$. Note that $Z'\rightarrow C_2$ is an elliptic fibration with trivial monodromy and constant $j$-invariant. Thus, by \Cref{corollary_isotrivial_and_trivial_vhs}, the minimal model of $Z'$ is given by $(C_2\times E_2)^{\eta_2}$ for some elliptic curve $E_2\cong\C/\Lambda_2$ and some cohomology class $\eta_2\in H^1(C_2,\OO_{C_2}/\Lambda_2)$ such that $(C_2\times E_2)^{\eta_2}$ is K\"ahler. We claim that $Z'$ is in fact isomorphic to $(C_2\times E_2)^{\eta_2}$. To prove this claim, note that $Z'$ does not contain any rational curves. Indeed, the composition
    \begin{equation*}
        Z'\longrightarrow Z\longrightarrow (C_1\times E_1)^{\eta_1}
    \end{equation*}
    is a composition of finite morphisms. Thus, the image of a rational curve in $Z'$ is a rational curve in $(C_1\times E_1)^{\eta_1}$. It hence suffices to show that $(C_1\times E_1)^{\eta_1}$ does not contain any rational curves. As the genus of $C_1$ is positive, the projection $\pi_1\colon (C_1\times E_1)^{\eta_1}\rightarrow C_1$ contracts any rational curve in $(C_1\times E_1)^{\eta_1}$. Thus, any rational curve in $(C_1\times E_1)^{\eta_1}$ lies in a fiber of $\pi_1$. However, as the fibers of $\pi_1$ are isomorphic to the elliptic curve $E_1$, $(C_1\times E_1)^{\eta_1}$ cannot contain any rational curve. To finish the proof of the claim, let $\mu\colon\tilde{Z}\rightarrow Z'$ be a desingularization, obtained by blowing up smooth points in the singular locus, and such that $\mu$ does not contract any $(-1)$-curves. Then, there exists a diagram as in the proof of \Cref{bimeromorphic_to_elliptic_torsor_over_curve_implies_smooth}:
    \[\begin{tikzcd}
	   {\tilde{Z}} & {(C_2\times E_2)^{\eta_2}} \\
	   {Z'} & {C_2,}
	   \arrow["b", from=1-1, to=1-2]
	   \arrow["\mu"', from=1-1, to=2-1]
	   \arrow["{\pi_2}", from=1-2, to=2-2]
	   \arrow[from=2-1, to=2-2]
    \end{tikzcd}\]
    Applying the same arguments as in the proof of \Cref{bimeromorphic_to_elliptic_torsor_over_curve_implies_smooth}, we see that $\mu\colon\tilde{Z}\rightarrow Z'$ and $b$ are isomorphisms. Thus, $Z'\cong (C_2\times E_2)^{\eta_2}$.\par
    By construction, there is a commutative diagram
    \[\begin{tikzcd}
	   {(C_2\times E_2)^{\eta_2}\cong Z'} & {Z\times_{(C_1\times E_1)^{\eta_1}}(C_2\times E_1)^{h^*\eta_1}} & {(C_2\times E_1)^{h^*\eta_1}} \\
	   & {C_2} & {.}
	   \arrow[from=1-1, to=1-2]
	   \arrow[from=1-1, to=2-2]
	   \arrow[from=1-2, to=1-3]
	   \arrow[from=1-2, to=2-2]
	   \arrow[from=1-3, to=2-2]
    \end{tikzcd}\]
    We hence obtain fiber-wise a morphism $E_2\rightarrow E_1$. As the upper row is a composition of finite morphisms, the induced morphisms $E_2\rightarrow E_1$ are also finite. Since $E_1$ and $E_2$ are elliptic curves, the Riemann--Hurwitz formula shows that the finite morphisms $E_2\rightarrow E_1$ are \'etale.\par
    Next, we show that $X_2\rightarrow X_1$ is finite \'etale. Due to the choice of $h\colon C_2\rightarrow C_1$, a local analysis shows that
    \begin{equation*}
        Z'\longrightarrow Z
    \end{equation*}
    is finite quasi-\'etale. Thus, also $X_1\times_ZZ'\rightarrow X_1$ is finite quasi-\'etale. Since $Z'\cong (C_2\times E_2)^{\eta_2}$, $X_2\rightarrow X_1$ is finite quasi-\'etale. As $X_1$ is smooth, this morphism is finite \'etale.\par
    It remains to prove that the induced morphism $g_2\colon X_2\rightarrow (C_2\times E_2)^{\eta_2}$ is an elliptic fibration. Note that by construction, the finite morphism in the Stein factorization of $g_2\colon X_2\rightarrow (C_2\times E_2)^{\eta_2}$ is given by $\id\colon (C_2\times E_2)^{\eta_2}\rightarrow (C_2\times E_2)^{\eta_2}$. Therefore, the morphism has connected fibers. Denote by $f_2$ the natural morphism $f_2\colon X_2\rightarrow C_2$. For a general $c\in C_2$, we obtain the commutative diagram
    \[\begin{tikzcd}
	   {f^{-1}_2(c)} && {E_2} \\
	   & {Alb(f^{-1}_2(c))} && {,}
	   \arrow[from=1-1, to=1-3]
	   \arrow[from=1-1, to=2-2]
	   \arrow[from=2-2, to=1-3]
    \end{tikzcd}\]
    where $f^{-1}_2(c)$ is bielliptic or a 2-torus, the horizontal morphism is induced by $g_2$, and the lower half of the diagram comes from the universal property of the Albanese. If $f^{-1}_2(c)$ is a 2-torus, the horizontal map must be an elliptic fibration since it has connected fibers. If $f^{-1}_2(c)$ is bielliptic, the morphism to $Alb(f^{-1}_2(c))$ is an elliptic fibration onto an elliptic curve. In particular, the morphism $Alb(f^{-1}_2(c))\rightarrow E_2$ is finite. Since the horizontal morphism has connected fibers, $Alb(f^{-1}_2(c))\rightarrow E_2$ is an isomorphism. Hence, the horizontal morphism is isomorphic to the elliptic fibration $f_2^{-1}(c)\rightarrow Alb(f_2^{-1}(c)$. This shows that $g_2$ is an elliptic fibration, as we want.
\end{proof}
Recall the definition of an extension of \'etale covers of tori to \'etale covers of torus fiber bundles, \Cref{definition_extension_of_etale_cover_of_elliptic_curves}: Let $Y$ be a compact connected K\"ahler manifold and let $\varphi\colon A'=\C^g/\Lambda'\rightarrow A=\C^g/\Lambda$ be a finite \'etale cover of tori. Suppose $\eta\in H^1(Y,\OO_Y^{\oplus g}/\Lambda)$ is a cohomology class with trivial Chern class. Then there exists a non-unique cohomology class $\eta'\in H^1(Y,\OO_Y^{\oplus g}/\Lambda')$ and a finite \'etale cover
\begin{equation*}
    (Y\times A')^{\eta'}\longrightarrow(Y\times A)^{\eta}.
\end{equation*}
We call $(Y\times A')^{\eta'}\rightarrow(Y\times A)^{\eta}$ an extension of $\varphi\colon A'\rightarrow A$.
\begin{lemma}\label{step3_key_proposition_kodaira_1_G_equivariant}
    In the notation of \Cref{step2_key_proposition_kodaira_1_G_equivariant}, there is a finite \'etale cover $E_3\rightarrow E_2$, such that for an extension $(C_2\times E_3)^{\eta_3}\rightarrow(C_2\times E_2)^{\eta_2}$ of $E_3\rightarrow E_2$, the general fiber of the morphism
    \begin{equation*}
        f_3\colon X_3:=X_2\times_{(C_2\times E_2)^{\eta_2}}(C_2\times E_3)^{\eta_3}\longrightarrow C_2
    \end{equation*}
    is a 2-torus, the natural morphism $X_3\rightarrow X_2$ is finite \'etale and the morphism $g_3\colon X_3\rightarrow(C_2\times E_3)^{\eta_3}$ is still an elliptic fibration.
\end{lemma}
\begin{proof}
    If the general fiber of $f_2$ was a 2-torus, we can take $E_3=E_2$. Suppose the general fiber of $f_2$ is bielliptic. Then there is a finite \'etale cover $E_3\rightarrow E_2$ such that $f^{-1}_2(c)\times_{E_2}E_3$ is a 2-torus for a general $c\in C_2$. By \Cref{fundamental_properties_special_torus_bundles}, as $(C_2\times E_2)^{\eta_2}$ is K\"ahler, the Chern class $c(\eta_2)\in H^2(C_2,\Lambda_2)$ is a torsion class, where $\Lambda_2\subset\C$ is a lattice such that $E_2\cong\C/\Lambda_2$. However, since $H^2(C_2,\Lambda_2)\cong\Lambda_2$ is torsion free, $c(\eta_2)=0$. Thus, choosing an extension $(C_2\times E_3)^{\eta_3}\rightarrow(C_2\times E_2)^{\eta_2}$ of $E_3\rightarrow E_2$ by \Cref{lemma_extension_of_etale_cover_of_elliptic_curves} shows that $f_3\colon X_3\rightarrow C_2$ has as general fiber a 2-torus.\par
    Note that by construction, $(C_2\times E_3)^{\eta_3}\rightarrow (C_2\times E_2)^{\eta_2}$ is finite \'etale. Since $X_3$ is defined as a fiber product, also $X_3\rightarrow X_2$ is finite \'etale.\par
    To see that $g_3$ is still an elliptic fibration, note that the fibers are not affected by the extension of $E_3\rightarrow E_2$.
\end{proof}
We are now in the position to prove \Cref{key_proposition_kodaira_1_G-equivariant} stated above.
\begin{proof}[Proof of \Cref{key_proposition_kodaira_1_G-equivariant}]
    \begin{case}\label{case1_proof_key_proposition_kodaira_1_G-equivariant}
        Suppose first that the genus of $C$ is positive. We can then perform the steps above to arrive at the situation of \Cref{step3_key_proposition_kodaira_1_G_equivariant}. In the notation of \Cref{step0_key_proposition_kodaira_1_G_equivariant} -- \Cref{step3_key_proposition_kodaira_1_G_equivariant}, denote by $U\subset C_2$ the non-empty Zariski-open subset over which $f_3\colon X_3\rightarrow C_2$ is smooth. Then $g_3\colon X_3\rightarrow (C_2\times E_3)^{\eta_3}$ is smooth over $\pi^{-1}_3(U)$. In particular, the discriminant divisor of $g_3$ consists of finitely many fibers of $\pi_3\colon(C_2\times E_3)^{\eta_3}\rightarrow C_2$. The multiplicity of the singular fibers induces an orbifold structure on $C_2$. Let $e\colon C_3\rightarrow C_2$ be a finite orbifold cover with trivial orbifold structure as in \Cref{classification_orbifolds_curves}. Define $X_4$ to be the normalization of
        \begin{equation*}
            X_3\times_{(C_2\times E_3)^{\eta_3}}(C_3\times E_3)^{e^*\eta_3}.
        \end{equation*}
        We set $X':=X_4$, $C':=C_3$, $E':=E_3$, $\eta':=e^*\eta_3$, $G:=0$ and claim that these satisfy all conditions in \Cref{key_proposition_kodaira_1_G-equivariant}.\par
        As $e\colon C_3\rightarrow C_2$ is a finite orbifold cover with trivial orbifold structure, a local analysis shows that the natural map $X'/G=X_4\rightarrow X_3$ is finite \'etale. Moreover, the composition
        \begin{equation*}
            X'/G=X_4\longrightarrow X_3\longrightarrow X_2\longrightarrow X_1\longrightarrow X_0\longrightarrow X
        \end{equation*}
        is a composition of finite \'etale covers by \Cref{step0_key_proposition_kodaira_1_G_equivariant}, \Cref{step1_key_proposition_kodaira_1_G_equivariant}, \Cref{step2_key_proposition_kodaira_1_G_equivariant}, \Cref{step3_key_proposition_kodaira_1_G_equivariant}. Therefore, \cref{item1_key_proposition_kodaira_1_G-equivariant} holds true. To prove \cref{item3_key_proposition_kodaira_1_G-equivariant}, note that the induced morphism $g'\colon X'\rightarrow (C'\times E')^{\eta'}$ has at most finitely many multiple fibers. As the discriminant divisor of $g'$ consists again of finitely many fibers of $\pi'\colon(C'\times E')^{\eta'}\rightarrow C'$, $g'$ is smooth away from a smooth divisor. Therefore, by \cite[][Thm 3.3.3]{Nakayama_elliptic-fibrations} $g'$ is locally projective. By combining \cite[][Thm. 5.3.3]{Nakayama_elliptic-fibrations} and \cite[][Cor. 4.3.3]{Nakayama_elliptic-fibrations}, we see that $g'$ has local meromorphic sections over every point of $(C'\times E')^{\eta'}$. To prove \cref{item4_key_proposition_kodaira_1_G-equivariant}, note that by \Cref{step3_key_proposition_kodaira_1_G_equivariant}, the general fiber of $f_3\colon X_3\rightarrow C_2$ is a 2-torus. By construction, the general fibers of $f'\colon X'\rightarrow C'$ thus also a 2-torus. Finally, \cref{item2_key_proposition_kodaira_1_G-equivariant}, \cref{item5_key_proposition_kodaira_1_G-equivariant} and \cref{item6_key_proposition_kodaira_1_G-equivariant} hold true by construction.
    \end{case}
    \begin{case}\label{case2_proof_key_proposition_kodaira_1_G-equivariant}
        Suppose now $C\cong\mathbb{P}^1$. First, note that the image of the Albanese morphism $X\rightarrow Alb(X)$ is an elliptic curve. Indeed, by assumption, the general fiber of the Iitaka fibration is mapped to a translate of a fixed elliptic curve $E\subset Alb(X)$. Define $A:=Alb(X)/E$. Then there is a morphism $C\rightarrow A$ that has the same image as $X\rightarrow A$. However, since $C\cong\mathbb{P}^1$, the image of $C\rightarrow A$ must be a point. Therefore, the image of $X\rightarrow Alb(X)$ equals $E$. Applying the Stein factorization to $g\colon X\rightarrow E$ allows us to assume that $g$ has connected fibers. Step 1 in \cite[][Sec. 6.3]{Hao_Schreieder_3-folds} shows that there is a finite \'etale cover $E_1\rightarrow E$ inducing a finite \'etale cover $X_1\rightarrow X$ such that the fiber over a general $c\in C$ splits into a product of two elliptic curves:
        \begin{equation*}
            (X_1)_c\cong F_c\times (E_1)_c.
        \end{equation*}
        By step 2 in \cite[][Sec. 6.3]{Hao_Schreieder_3-folds}, we can find a finite cover $C_1\rightarrow C$ such that the morphism from the normalization $X_2$ of $X_1\times_CC_1\rightarrow X_1$ is finite \'etale and such that the morphism $X_2\rightarrow C_1\times E_1$ has connected fibers. It is thus an elliptic fibration with discriminant locus given by finitely many fibers of the projection $\pi_1\colon C_1\times E_1\rightarrow C_1$. If $C_1$ has positive genus or $C_1\cong\mathbb{P}^1$ and the orbifold structure induced by the multiple fibers of $X_2\rightarrow C_1\times E_1$ is good, we are in the same situation as in \cref{case1_proof_key_proposition_kodaira_1_G-equivariant} of the proof. It is thus enough to deal with the case $C_1\cong\mathbb{P}^1$ with bad orbifold structure. We can find a finite Galois cover $C_2\rightarrow C_1$ with Galois group $G$ that is branched over the orbifold points with the correct multiplicities, but has possibly other branch points. Define $X'$ as the normalization of $X_2\times_{C_1}C_2$, $C':=C_2$, $E':=E_1$, $\eta':=0$. We claim that these satisfy all conditions in \Cref{key_proposition_kodaira_1_G-equivariant}.\par
        By construction, $X'/G\cong X_1\rightarrow X$ is finite \'etale. This shows \cref{item1_key_proposition_kodaira_1_G-equivariant}. To prove \cref{item3_key_proposition_kodaira_1_G-equivariant}, we argue as in \cref{case1_proof_key_proposition_kodaira_1_G-equivariant}. Since the general fiber of $X_1\rightarrow C$ is a 2-torus, the construction of $X'$ shows that the general fiber of $X'\rightarrow C'$ is a 2-torus as well. This proves \cref{item4_key_proposition_kodaira_1_G-equivariant}. Finally, \cref{item2_key_proposition_kodaira_1_G-equivariant}, \cref{item5_key_proposition_kodaira_1_G-equivariant} and \cref{item6_key_proposition_kodaira_1_G-equivariant} hold true by construction.
    \end{case}
\end{proof}
Suppose from now on that we are in the situation described in \Cref{key_proposition_kodaira_1_G-equivariant}:
\[\begin{tikzcd}
    {X'} && {(C'\times E')^{\eta'}} \\
	& {C'} && {.}
	\arrow["{g'}", from=1-1, to=1-3]
	\arrow["{f'}"', from=1-1, to=2-2]
	\arrow["{\pi'}", from=1-3, to=2-2]
\end{tikzcd}\]
Note that in view of \Cref{Kodaira_dim_1_Weierstrass}, the elliptic fibrations are in the wrong order: the fibration $g'$ contracts the elliptic curves that do not contribute to the holomorphic 1-forms satisfying condition \ref{conditionC}, while the fibration $\pi'$ contracts the elliptic curves that do. What we are looking for is a diagram of the form
\[\begin{tikzcd}
	{(T\times E'')^\eta} && T \\
	& {C'/G} && {,}
	\arrow["{\pi''}", from=1-1, to=1-3]
	\arrow[from=1-1, to=2-2]
	\arrow[from=1-3, to=2-2]
\end{tikzcd}\]
where $T\rightarrow C'/G$ is an elliptic surface of Kodaira dimension 1 that does not satisfy condition \ref{conditionC} and such that there is a finite \'etale cover $(T\times E'')^\eta\rightarrow X'/G\rightarrow X$. In the following, we will prove that such a diagram exists and that it is induced by a finite \'etale cover $E''\rightarrow E'$.\par
To shorten the notation, set $Y':=(C'\times E')^{\eta'}$. Note that as the general fiber of $f'$ is smooth, $g'$ is smooth outside finitely many fibers of $\pi'$, i.e., the discriminant locus $\Delta\subset Y'$ of $g'$ is the pullback of an effective divisor on $C'$. Denote the respective complements by $j\colon Y^*\hookrightarrow Y'$ and $i\colon C^*\hookrightarrow C'$. Let $H:=R^1g'_*\Z\vert_{Y^*}$ be the variation of Hodge structures associated to $g'$ in \Cref{key_proposition_kodaira_1_G-equivariant}.
\begin{lemma}
    Let $W(\mathcal{L},\alpha,\beta)\rightarrow Y'$ be the unique $G$-equivariant minimal Weierstraß fibration associated to $H$, from \cite[][Cor. 2.7]{Nakayama_Weierstrass}. Then there is a $G$-equivariant variation of Hodge structures $K$ on $C^*$ and a $G$-equivariant line bundle $\mathcal{M}$ on $C'$, such that $H\cong(\pi')^*K$ and $\mathcal{L}\cong(\pi')^*\mathcal{M}$. In particular, by viewing $\alpha$ and $\beta$ as global sections of $\mathcal{M}^{-4}$, resp.~ $\mathcal{M}^{-6}$, there is a $G$-equivariant isomorphism
    \begin{equation*}
        W(\mathcal{L},\alpha,\beta)\cong W(\mathcal{M},\alpha,\beta)\times_{C'}Y'
    \end{equation*}
    over $Y'$.
\end{lemma}
\begin{proof}
    First, note that $H$ is constant along the fibers of $\pi'\colon Y'\rightarrow C'$. Indeed, over a fiber of $\pi'$ that is contained in $Y^*$, the elliptic fibration $g'$ is a surjective homomorphism from a 2-torus to an elliptic curve. In particular, the restriction of $H$ to this fiber is given by the Hodge structure of the kernel of the homomorphism. Therefore, there is a variation of Hodge structures $K$ on $C^*$ such that $H\cong(\pi'\vert_{Y^*})^*K$. The $G$-action on $H$ induces a $G$-action on $K$.\par
    Similarly, note that the line bundle $\mathcal{L}$ restricts on the fiber of $\pi'$ over a point in $C^*$ to the trivial line bundle. Indeed, since $H$ is isomorphic to the pullback of $K$, the Jacobian fibration $J_H\rightarrow Y^*$ is isomorphic to $J_K\times_{C^*}Y^*\rightarrow Y^*$. Since $W(\mathcal{L},\alpha,\beta)\vert_{Y^*}\rightarrow Y^*$ is isomorphic to $J_H\rightarrow Y^*$, the assertion follows. The set of points $c\in C'$, where $\mathcal{L}\vert_{(\pi')^{-1}(c)}\cong\OO_{(\pi')^{-1}(c)}$ is closed. Indeed, the subsets
    \begin{align*}
        \{c\in C\mid h^0((\pi')^{-1}(c),\mathcal{L}\vert_{(\pi')^{-1}(c)})\geq 1\},\\
        \{c\in C\mid h^0((\pi')^{-1}(c),\mathcal{L}^\vee\vert_{(\pi')^{-1}(c)})\geq 1\}
    \end{align*}
    are closed by upper semicontinuity \cite[][Thm. 12.8]{Hartshorne_AG} and their intersection is the locus where $\mathcal{L}\vert_{(\pi')^{-1}(c)}\cong\OO_{(\pi')^{-1}(c)}$. Hence, there is a line bundle $\mathcal{M}$ on $C'$ such that $\mathcal{L}\cong (\pi')^*\mathcal{M}$. The $G$-action on $\mathcal{L}$ induces a $G$-action on $\mathcal{M}$.
\end{proof}
Let us now fix the notation for the rest of this section. We abbreviate the $G$-equivariant Weierstraß models $W(\mathcal{L},\alpha,\beta)\rightarrow Y'$, resp.~ $W(\mathcal{M},\alpha,\beta)\rightarrow C'$, by $p\colon W\rightarrow Y'$, resp.~ $q\colon V\rightarrow C'$. Recall from \Cref{section_Weierstrass_models} that $\mathcal{W}(\mathcal{L},\alpha,\beta)^\#$ is the sheaf of local sections of $p\colon W\rightarrow Y'$ with image in the locus $W^\#\subset W$ where $p$ is smooth, and that $\mathcal{W}(\mathcal{L},\alpha,\beta)^{mer}$ is the sheaf of local meromorphic sections of $p\colon W\rightarrow Y'$. We will denote these sheaves by $\mathcal{W}^\#$, resp.~ $\mathcal{W}^{mer}$. Similarly, the sheaf $\mathcal{W}(\mathcal{M},\alpha,\beta)^\#$ of local sections of $q\colon V\rightarrow C'$ with image in the locus $V^\#\subset V$ where $q$ is smooth is denoted by $\mathcal{V}^\#$. Twists of the $G$-equivariant Weierstraß models $p\colon W\rightarrow Y'$, resp.~ $q\colon V\rightarrow C'$ by a cohomology class $\gamma\in H^1_G(Y',\mathcal{W}^\#)$, resp.~ $\delta\in H^1_G(C',\mathcal{V}^\#)$ are denoted by $p^\gamma\colon W^\gamma\rightarrow Y'$, resp.~ by $q^\delta\colon V^\delta\rightarrow C'$. Recall furthermore that $\mathcal{E}_G(Y',\Delta,H)$ was defined as the set of equivalence classes of $G$-equivariant elliptic fibrations $h\colon Z\rightarrow Y'$ that satisfy the following properties:
\begin{enumerate}[label=(\roman*)]
    \item $h\colon Z\rightarrow Y'$ has local meromorphic sections over every point of $Y'$,
    \item the elliptic fibration $h^{-1}(Y^*)\rightarrow Y^*:=Y'\setminus\Delta$ is bimeromorphic to a smooth elliptic fibration over $Y^*$, and
    \item this smooth elliptic fibration induces $H$ as its $G$-equivariant variation of Hodge structures;
\end{enumerate}
where two such elliptic fibrations are equivalent if they are $G$-equivariantly bimeromorphic. Moreover, recall that there is a sequence of maps of sets, \cref{composition_construction_injection_into_H1(Wmer)},
\begin{equation*}
    H_G^1(Y',\mathcal{W}^\#)\longrightarrow\mathcal{E}_G(Y',\Delta,H)\longhookrightarrow H_G^1(Y',\mathcal{W}^{mer})
\end{equation*}
that equals the map in cohomology induced by the inclusion $\mathcal{W}^\#\subset\mathcal{W}^{mer}$. Denote by $\theta\in H^1_G(Y',\mathcal{W}^{mer})$ the cohomology class defined by $g'$. By \Cref{tautological_models}, there is a positive integer $m$ such that $m\theta$ lifts to a cohomology class $\theta'\in H^1_G(Y',\mathcal{W}^\#)$, and we can construct the $G$-equivariant tautological model
\[\begin{tikzcd}
	{\mathcal{X}} && {W(\mathcal{L},\alpha,\beta)^{\theta'}} \\
	& {Y'}
	\arrow[from=1-1, to=1-3]
	\arrow[from=1-1, to=2-2]
	\arrow[from=1-3, to=2-2]
\end{tikzcd}\]
of $g'$. Denote by $(\pi')^{-1}\mathcal{V}^\#$ and $(\pi')^{-1}\mathcal{M}$ the sheaves on $Y'$ obtained by applying the inverse image functor to the sheaves $\mathcal{V}^\#$, resp.~ $\mathcal{M}$ on $C'$.
\begin{lemma}\label{no_invariant_global_sections}
    $H^0_G(C',\mathcal{M})=0=H^0_G(Y',\mathcal{L})$.
\end{lemma}
\begin{proof}
    Note that it suffices to prove $H^0_G(Y',\mathcal{L})=0$. Indeed, as $\pi'\colon Y'\rightarrow C'$ is dominant, the pullback map
    \begin{equation*}
        H_G^0(C',\mathcal{M})\longhookrightarrow H_G^0(Y',(\pi')^*\mathcal{M}\cong\mathcal{L})
    \end{equation*}
    is injective.\par
    Suppose that $\mathcal{L}$ has a non-trivial $G$-invariant global section. We then obtain a non-trivial $G$-equivariant map
    \begin{equation}\label{G-equivariant_isom_to_OO}
        \OO_Y\longrightarrow\mathcal{L}.
    \end{equation}
    By taking tensor products on both sides, we also obtain non-trivial maps $\OO_Y\rightarrow\mathcal{L}^4$, $\OO_Y\rightarrow\mathcal{L}^6$. Note that at least one of $\alpha\in H^0(Y',\mathcal{L}^{-4})$, $\beta\in H^0(Y',\mathcal{L}^{-6})$ is non-trivial, as otherwise $4\alpha^3+27\beta^2=0$, which is a contradiction, see \Cref{definition_Weierstrass_models}. Therefore, at least one of the maps $\OO_Y\rightarrow\mathcal{L}^4$, $\OO_Y\rightarrow\mathcal{L}^6$ is an isomorphism. Thus, also \cref{G-equivariant_isom_to_OO} must be an isomorphism, i.e. $\mathcal{L}\cong_G\OO_{Y'}$, where $\OO_{Y'}$ is endowed with the trivial $G$-action. Hence, $\alpha$, $\beta\in\C$ and the Weierstraß model is given by
    \begin{equation*}
        Y'\times F\longrightarrow Y',
    \end{equation*}
    where $F$ is the elliptic curve $\{y^2z-(x^3+\alpha xz^2+\beta z^3)=0\}\subset\mathbb{P}^2$. This shows that the variation of Hodge structures $H$ must be $G$-equivariantly isomorphic to $\Gamma\subset\C$, where $F\cong\C/\Gamma$. From \Cref{trivial_vhs_constant_j_all_invariant} it follows that $X'$ is $G$-equivariantly bimeromorphic to $(Y'\times F)^\theta\rightarrow Y'$, where $G$ acts via $g\times tr(\sigma(g))$ for some group homomorphism $\sigma\colon G\rightarrow F$. Therefore, any non-trivial holomorphic 1-form on $(Y'\times F)^\theta$ coming from $F$ is invariant under the $G$-action. Define $G'$ to be the kernel of $\sigma\colon G\rightarrow F$. Note that $Y'/G'$ is a smooth minimal K\"ahler surface by construction. Hence, $(Y'\times F)^\theta/G'\rightarrow Y'/G'$ is an elliptic fiber bundle over a complex manifold. By \Cref{classification_torus_fiber_bundles}, there is a class $\gamma\in H^1(Y'/G',\OO_{Y'/G'}/\Gamma)$ such that $(Y'\times F)^\theta/G'\cong (Y'/G'\times F)^{\gamma}$. Note that $G/G'$ acts fixed-point freely on $(Y'/G'\times F)^{\gamma}$ and that its quotient is a smooth compact K\"ahler threefold bimeromorphic to $X'/G$. Thus, by \Cref{reduction_to_bimeromorphic_problem},
    \begin{equation*}
        (Y'/G'\times F)^{\gamma}\longrightarrow (Y'/G'\times F)^{\gamma}/(G/G')\cong X'/G\longrightarrow X'
    \end{equation*}
    is a finite \'etale cover. However, the fibers of $(Y'/G'\times F)^{\gamma}\rightarrow Y'/G'\rightarrow C'/G'$ are not contracted to a curve by the Albanese morphism; this is a contradiction to \Cref{assumption}.
\end{proof}
\begin{lemma}\label{lift_of_theta_to_inverse_image}
    The following holds true.
    \begin{enumerate}[label=(\roman*)]
        \item\label{item_map_on_first_cohom_V_to_W_injective} The natural map
        \begin{equation}\label{map_on_first_cohom_V_to_W_injective}
            H^1_G(Y',(\pi')^{-1}\mathcal{V}^\#)\longhookrightarrow H^1_G(Y',\mathcal{W}^\#)
        \end{equation}
        is injective.
        \item\label{item_theta'_in_image_of_map_on_first_cohom_V_to_W} There is a unique cohomology class $\theta_K\in H^1_G(Y',(\pi')^{-1}\mathcal{V}^\#)$ that is mapped to $\theta'$ by \cref{map_on_first_cohom_V_to_W_injective}.
        \item\label{differential_H1_V_to_H2_K_sends_theta'_to_torsion} The map
        \begin{equation}\label{differential_H1_V_to_H2_K}
            H^1_G(Y',(\pi')^{-1}\mathcal{V}^\#)\longrightarrow H^2_G(Y',(\pi')^{-1}i_*K)
        \end{equation}
        sends $\theta_K$ to a torsion class.
    \end{enumerate}
\end{lemma}
\begin{proof}
    The Leray spectral sequences
    \begin{align*}
        E^{p,q}_2=H^p_G(C',R^q\pi'_*(\pi')^{-1}\mathcal{M})&\Longrightarrow H^{p+q}_G(Y',(\pi')^{-1}\mathcal{M})\\
        'E^{p,q}_2=H^p_G(C',R^q\pi'_*\mathcal{L})&\Longrightarrow H^{p+q}_G(Y',\mathcal{L})
    \end{align*}
    together with $H^0_G(C',\mathcal{M})=0=H^0_G(Y',\mathcal{L})$ show that the maps $H^i_G(Y',(\pi')^{-1}\mathcal{M})\stackrel{\sim}{\rightarrow}H^i_G(Y',\mathcal{L})$ are isomorphisms for $i=0$, 1 and that for $i=2$ this map is isomorphic to the map
    \begin{equation*}
        H^1_G(C',\mathcal{M})\otimes_\Z\Lambda'\longrightarrow H^1_G(C',\mathcal{M})\otimes_\C\C,
    \end{equation*}
    induced by $\Lambda'\subset\C$, where $E'\cong\C/\Lambda'$. Since $\Lambda'\cong\Z^2$, this map is surjective with kernel isomorphic to $H^1_G(C',\mathcal{M})$. Therefore,
    \begin{align}
        \label{H0_of_L_over_inverse_image_M} H^0_G(Y',\mathcal{L}/(\pi')^{-1}\mathcal{M})&=0,\\
        \label{H1_of_L_over_inverse_image_M} H^1_G(Y',\mathcal{L}/(\pi')^{-1}\mathcal{M})&\cong H^1_G(C',\mathcal{M}).
    \end{align}
    The commutative diagram
    \begin{equation}\label{diagram_KMV_HLW}
        \begin{tikzcd}
	       & 0 & 0 & 0 \\
	       0 & {(\pi')^{-1}i_*K} & {(\pi')^{-1}\mathcal{M}} & {(\pi')^{-1}\mathcal{V}^\#} & 0 \\
	       0 & {j_*H} & {\mathcal{L}} & {\mathcal{W}^\#} & 0 \\
	       & 0 & {\mathcal{L}/(\pi')^{-1}\mathcal{M}} & {\mathcal{W}^\#/(\pi')^{-1}\mathcal{V}^\#} & 0 \\
	       && 0 & 0
	       \arrow[from=1-2, to=2-2]
	       \arrow[from=1-3, to=2-3]
	       \arrow[from=1-4, to=2-4]
	       \arrow[from=2-1, to=2-2]
	       \arrow[from=2-2, to=2-3]
	       \arrow[from=2-2, to=3-2]
	       \arrow[from=2-3, to=2-4]
	       \arrow[from=2-3, to=3-3]
	       \arrow[from=2-4, to=2-5]
	       \arrow[from=2-4, to=3-4]
	       \arrow[from=3-1, to=3-2]
	       \arrow[from=3-2, to=3-3]
	       \arrow[from=3-2, to=4-2]
	       \arrow[from=3-3, to=3-4]
	       \arrow[from=3-3, to=4-3]
	       \arrow[from=3-4, to=3-5]
	       \arrow[from=3-4, to=4-4]
	       \arrow[from=4-2, to=4-3]
	       \arrow[from=4-3, to=4-4]
	       \arrow[from=4-3, to=5-3]
	       \arrow[from=4-4, to=4-5]
	       \arrow[from=4-4, to=5-4]
        \end{tikzcd}
    \end{equation}
    has exact rows and columns and therefore induces an isomorphism $\mathcal{L}/(\pi')^{-1}\mathcal{M}\stackrel{\sim}{\rightarrow}\mathcal{W}^\#/(\pi')^{-1}\mathcal{V}^\#$. Hence, the cohomology computations \cref{H0_of_L_over_inverse_image_M}, \cref{H1_of_L_over_inverse_image_M} imply
    \begin{align*}
        H^0_G(Y',\mathcal{W}^\#/(\pi')^{-1}\mathcal{V}^\#)&=0,\\
        H^1_G(Y',\mathcal{W}^\#/(\pi')^{-1}\mathcal{V}^\#)&\cong H^1_G(C',\mathcal{M}).
    \end{align*}
    Applying $G$-equivariant cohomology to the column on the right-hand side of diagram (\ref{diagram_KMV_HLW}) thus yields an exact sequence
    \begin{equation*}
        0\longrightarrow H^1_G(Y',(\pi')^{-1}\mathcal{V}^\#)\longrightarrow H^1_G(Y',\mathcal{W}^\#)\longrightarrow H^1_G(C',\mathcal{M}),
    \end{equation*}
    which proves \cref{item_map_on_first_cohom_V_to_W_injective}.\par
    To prove \cref{item_theta'_in_image_of_map_on_first_cohom_V_to_W}, consider the commutative diagram
    \[\begin{tikzcd}
    	{H^1_G(Y',(\pi')^{-1}\mathcal{M})} & {H^1_G(Y',(\pi')^{-1}\mathcal{V}^\#)} & {H^2_G(Y',(\pi')^{-1}i_*K)} \\
    	{H^1_G(Y',\mathcal{L})} & {H^1_G(Y',\mathcal{W}^\#)} & {H^2_G(Y',j_*H)}
    	\arrow[from=1-1, to=1-2]
    	\arrow["\cong"', from=1-1, to=2-1]
    	\arrow[from=1-2, to=1-3]
    	\arrow[hook, from=1-2, to=2-2]
    	\arrow["\cong", from=1-3, to=2-3]
    	\arrow[from=2-1, to=2-2]
    	\arrow[from=2-2, to=2-3]
    \end{tikzcd}\]
    with exact rows.
    \begin{claim}\label{claim_theta'_to_torsion}
        The horizontal map on the lower right-hand side of the diagram sends $\theta'$ to a torsion class.
    \end{claim}
    Suppose for a moment that we have proven \cref{claim_theta'_to_torsion}. A positive multiple of $\theta'$ thus comes from $H^1_G(Y',(\pi')^{-1}\mathcal{V}^\#)$ by a diagram chase, i.e. $\theta'$ defines a torsion element in the cokernel of $H^1_G(Y',(\pi')^{-1}\mathcal{V}^\#)\hookrightarrow H^1_G(Y',\mathcal{W}^\#)$. However, the cokernel is a subgroup of $H^1_G(C',\mathcal{M})$, which is torsion-free. This proves \cref{item_theta'_in_image_of_map_on_first_cohom_V_to_W}.\par
    Let us now prove \cref{claim_theta'_to_torsion}. Since $X'$ is K\"ahler, the class $\theta'$ is mapped to a torsion class in $H^2(Y',j_*H)$ by the composition
    \begin{equation*}
        H^1_G(Y',\mathcal{W}^\#)\longrightarrow H^1(Y',\mathcal{W}^\#)^G\subset H^1(Y',\mathcal{W}^\#)\longrightarrow H^2(Y',j_*H),
    \end{equation*}
    \Cref{Weierstrass_model_Kaehler}. The commutative diagram
    \[\begin{tikzcd}
	   {H^1_G(Y',\mathcal{W}^\#)} & {H^1(Y',\mathcal{W}^\#)^G} & {H^1(Y',\mathcal{W}^\#)} \\
	   {H^2_G(Y',j_*H)} & {H^2(Y',j_*H)^G} & {H^2(Y',j_*H)}
	   \arrow[from=1-1, to=1-2]
	   \arrow[from=1-1, to=2-1]
	   \arrow[hook, from=1-2, to=1-3]
	   \arrow[from=1-2, to=2-2]
	   \arrow[from=1-3, to=2-3]
	   \arrow[from=2-1, to=2-2]
	   \arrow[hook, from=2-2, to=2-3]
    \end{tikzcd}\]
    shows that a multiple of $\theta'$ is mapped by the left vertical arrow to a class in the kernel of $H^2_G(Y',j_*H)\rightarrow H^2(Y',j_*H)^G$. The seven-term exact sequence associated to the Hochschild--Serre spectral sequence
    \begin{equation*}
        ''E^{p,q}_2=H^p(G,H^q(Y',j_*H))\Rightarrow H^{p+q}_G(Y',j_*H)
    \end{equation*}
    gives rise to the exact sequence
    \begin{equation*}
        H^2(G,H^0(Y',j_*H))\longrightarrow\ker(H^2_G(Y',j_*H)\longrightarrow H^2(Y',j_*H)^G)\longrightarrow H^1(G,H^1(Y',j_*H)).
    \end{equation*}
    Since the higher cohomology groups of finite groups are torsion, $\ker(H^2_G(Y',j_*H)\rightarrow H^2(Y',j_*H)^G)$ is also torsion. In particular, a further multiple of $\theta'$ is mapped to zero via $H^1_G(Y',\mathcal{W}^\#)\rightarrow H^2_G(Y',j_*H)$. This concludes the proof of \cref{claim_theta'_to_torsion}.\par
    To show \cref{differential_H1_V_to_H2_K_sends_theta'_to_torsion}, consider the commutative diagram
    \[\begin{tikzcd}
	   {H^1_G(Y',(\pi')^{-1}\mathcal{V}^\#)} & {H^1_G(Y',\mathcal{W}^\#)} \\
	   {H^2_G(Y',(\pi')^{-1}i_*K)} & {H^2_G(Y',j_*H).}
	   \arrow[hook, from=1-1, to=1-2]
	   \arrow[from=1-1, to=2-1]
	   \arrow[from=1-2, to=2-2]
	   \arrow["\sim", from=2-1, to=2-2]
    \end{tikzcd}\]
    By \cref{item_map_on_first_cohom_V_to_W_injective}, the upper horizontal map is injective, and by \cref{claim_theta'_to_torsion} the vertical map on the right-hand side sends $\theta'$ to a torsion class. The assertion follows.
\end{proof}
\begin{lemma}\label{commutative_diagram_inverse_images_pullback}
    There exists a commutative diagram
    \begin{equation}\label{diagram_VVVotimesK_KKKotimesK}
        \begin{tikzcd}
	       0 & {H^1_G(C',\mathcal{V}^\#)} & {H^1_G(Y',(\pi')^{-1}\mathcal{V}^\#)} & {H^0_G(C',\mathcal{V}^\#\otimes\Lambda')} \\
	       0 & {H^2_G(C',i_*K)} & {H^2_G(Y',(\pi')^{-1}i_*K)} & {H^1_G(C',i_*K\otimes\Lambda'),}
	       \arrow[from=1-1, to=1-2]
	       \arrow[from=1-2, to=1-3]
	       \arrow[from=1-2, to=2-2]
	       \arrow[from=1-3, to=1-4]
	       \arrow[from=1-3, to=2-3]
	       \arrow[hook, from=1-4, to=2-4]
	       \arrow[from=2-1, to=2-2]
	       \arrow[from=2-2, to=2-3]
	       \arrow[from=2-3, to=2-4]
    \end{tikzcd}
    \end{equation}
    where $\Lambda'\subset\C$ denotes the lattice in \Cref{key_proposition_kodaira_1_G-equivariant}, the rows of this diagram are exact and the vertical maps are the boundary maps induced by the short exact sequences $0\rightarrow i_*K\rightarrow\mathcal{M}\rightarrow\mathcal{V}^\#\rightarrow 0$, $0\rightarrow(\pi')^{-1}i_*K\rightarrow (\pi')^{-1}\mathcal{M}\rightarrow (\pi')^{-1}\mathcal{V}^\#\rightarrow 0$, resp.~ $0\rightarrow i_*K\otimes\Lambda'\rightarrow\mathcal{M}\otimes\Lambda'\rightarrow\mathcal{V}^\#\otimes\Lambda'\rightarrow 0$. Moreover, the vertical map on the right-hand side is injective.
\end{lemma}
\begin{proof}
    First, consider the Leray spectral sequence
    \begin{equation*}
        E^{p,q}_2=H^p_G(C',R^q\pi'_*(\pi')^{-1}\mathcal{V}^\#)\Longrightarrow H^{p+q}_G(Y',(\pi')^{-1}\mathcal{V}^\#).
    \end{equation*}
    This gives rise to the short exact sequence
    \begin{equation*}
        0\longrightarrow E^{1,0}_\infty\longrightarrow H^1(Y',(\pi')^{-1}\mathcal{V}^\#)\longrightarrow E^{0,1}_\infty\longrightarrow 0.
    \end{equation*}
    For degree reasons we have $E^{1,0}_\infty=E^{1,0}_2$, and $E^{0,1}_\infty=\ker(d^{0,1}_2\colon E^{0,1}_2\rightarrow E^{2,0}_2)\subset E^{0,1}_2$. Now, consider the Leray spectral sequence
    \begin{equation*}
        'E^{p,q}_2=H^p_G(C',R^q\pi'_*(\pi')^{-1}i_*K)\Longrightarrow H^{p+q}_G(Y',(\pi')^{-1}i_*K).
    \end{equation*}
    The short exact sequence $0\rightarrow i_*K\rightarrow\mathcal{M}\rightarrow\mathcal{V}^\#\rightarrow 0$ yields an injection $H^0_G(C',i_*K)\hookrightarrow H^0_G(C',\mathcal{M})=0$ by \Cref{no_invariant_global_sections}. Thus, $'E^{0,1}_2\cong H^0_G(C',i_*K)\otimes\Lambda'=0$ and $'E^{0,2}_2\cong H^0_G(C',i_*K)=0$. Therefore, $'E^{2,0}_\infty='E^{2,0}_2=H^2_G(C',i_*K)$ and $'E^{0,2}_\infty=0$. The spectral sequence induces a filtration
    \begin{equation*}
        0=F^3\subset F^2\subset F^1\subset F^0=H^2(Y',(\pi')^{-1}i_*K)
    \end{equation*}
    with quotients $F^i/F^{i+1}\cong 'E^{i,2-i}_\infty$. Since $'E^{0,2}_\infty=0$, we get $F^1=F^0$. Furthermore, $F^1/F^2\cong E^{1,1}_\infty=\ker(d^{1,1}_2\colon E^{1,1}_2\rightarrow E^{3,0}_2)\subset E^{1,1}_2$. Hence, the short exact sequence $0\rightarrow F^2\rightarrow F^1\rightarrow F^1/F^2\rightarrow 0$ induces the lower row in diagram (\ref{diagram_VVVotimesK_KKKotimesK}). The functoriality of the Leray spectral sequence implies that the vertical arrows in diagram (\ref{diagram_VVVotimesK_KKKotimesK}) are given by the respective boundary maps. The injectivity of the vertical map on the right-hand side of diagram (\ref{diagram_VVVotimesK_KKKotimesK}) again follows from $H^0_G(C',\mathcal{M})=0$.
\end{proof}
\begin{lemma}
    The class $\theta_K$ is mapped to a torsion element via
    \begin{equation*}
        H^1_G(Y',(\pi')^{-1}\mathcal{V}^\#)\longrightarrow H^0_G(C',\mathcal{V}^\#\otimes\Lambda')
    \end{equation*}
    in the diagram (\ref{diagram_VVVotimesK_KKKotimesK}).
\end{lemma}
\begin{proof}
    Combine \Cref{lift_of_theta_to_inverse_image} and \Cref{commutative_diagram_inverse_images_pullback}.
\end{proof}
\begin{construction}\label{construction_X''}
    Let $n$ be a positive integer that annihilates the image of $\theta_K$ in $H^0(C',\mathcal{V}^\#\otimes\Lambda')$. Define $\Lambda'':=n\cdot\Lambda'$, and $E'':=\C/\Lambda''$. Denote the finite \'etale morphism $E''\rightarrow E'$ by $\varphi$, and choose a $G$-equivariant extension $\psi\colon (C'\times E'')^{\eta''}\rightarrow (C'\times E')^{\eta'}$ of $\varphi\colon E''\rightarrow E'$, see \Cref{definition_extension_of_etale_cover_of_elliptic_curves}. A priori, it is not clear why this extension exists $G$-equivariantly. However, by \Cref{key_proposition_kodaira_1_G-equivariant}, $G$ is either trivial or $C'/G\cong\mathbb{P}^1$. In the first case, note that since $(C'\times E')^{\eta'}$ is K\"ahler, the Chern class $c(\eta')\in H^2(C',\Lambda')$ is a torsion class by \Cref{fundamental_properties_special_torus_bundles}. However, since $H^2(C',\Lambda')\cong\Lambda'$ is torsion free, $c(\eta')=0$. Hence, we can apply \Cref{lemma_extension_of_etale_cover_of_elliptic_curves} to obtain the desired extension. In the second case, $C$ must have been isomorphic to $\mathbb{P}^1$. Thus, by \Cref{key_proposition_kodaira_1_G-equivariant}, $\eta'=0$, and hence we can take $\eta''=0$. Again, to simplify the notation, we define $Y'':=(C'\times E'')^{\eta''}$, and denote by $\pi''\colon Y''\rightarrow C'$ the natural map. Define
    \begin{equation*}
        X'':=X'\times_{Y'}Y'',
    \end{equation*}
    and note that the natural map $X''\rightarrow X'$ is a finite $G$-equivariant \'etale cover. In particular, $X''/G\rightarrow X'/G\rightarrow X$ is still finite \'etale. Since Weierstraß fibrations behave well under flat pullback, the elliptic fibration $g''\colon X''\rightarrow Y''$ has the Weierstraß model
    \begin{equation*}
        W(\psi^*\mathcal{L}\cong(\pi'')^*\mathcal{M},\alpha,\beta)\longrightarrow Y'',
    \end{equation*}
    and the variation of Hodge structures is given by $\psi^*H$.
\end{construction}
\begin{lemma}
    There is a class $\theta''\in H^1_G(C',\mathcal{V}^\#)$ and a $G$-equivariant isomorphism
    \begin{equation*}
        W(\psi^*\mathcal{L},\alpha,\beta)^{\psi^*\theta'}\cong V^{\theta''}\times_{C'}Y''
    \end{equation*}
    over $Y''$.
\end{lemma}
\begin{proof}
    The exact sequence from \Cref{commutative_diagram_inverse_images_pullback} applied to $X'\rightarrow Y'\rightarrow C'$ and $X''\rightarrow Y''\rightarrow C'$ induces the commutative diagram
    \[\begin{tikzcd}
	   0 & {H^1_G(C',\mathcal{V}^\#)} & {H^1_G(Y',(\pi')^{-1}\mathcal{V}^\#)} & {H^0_G(C',\mathcal{V}^\#\otimes\Lambda')} \\
	   0 & {H^1_G(C',\mathcal{V}^\#)} & {H^1_G(Y'',(\pi'')^{-1}\mathcal{V}^\#)} & {H^0_G(C',\mathcal{V}^\#\otimes\Lambda'').}
	   \arrow[from=1-1, to=1-2]
	   \arrow[from=1-2, to=1-3]
	   \arrow[from=1-2, to=2-2]
	   \arrow[from=1-3, to=1-4]
	   \arrow["\psi^*", from=1-3, to=2-3]
	   \arrow["n", from=1-4, to=2-4]
	   \arrow[from=2-1, to=2-2]
	   \arrow[from=2-2, to=2-3]
	   \arrow[from=2-3, to=2-4]
    \end{tikzcd}\]
    Since $n$ annihilates the image of $\theta_K$ in $H^0_G(C',\mathcal{V}^\#\otimes\Lambda')$, there is a class $\theta''\in H^1_G(C',\mathcal{V}^\#)$ that is mapped to $\psi^*\theta'$.
\end{proof}
Consider the Stein factorization
\[\begin{tikzcd}
	{\mathcal{X}\times_{Y'}Y''} & {W(\psi^*\mathcal{L},\alpha,\beta)^{\psi^*\theta'}\cong V^{\theta''}\times_{C'}Y''} & {V^{\theta''}} \\
	& S & {.}
	\arrow[from=1-1, to=1-2]
	\arrow[from=1-1, to=2-2]
	\arrow[from=1-2, to=1-3]
	\arrow[from=2-2, to=1-3]
\end{tikzcd}\]
\begin{lemma}\label{last_situation_we_consider}
    The morphism $s\colon S\rightarrow V^{\theta''}\rightarrow C'$ is a $G$-equivariant elliptic fibration. Moreover, the elliptic fibration $S\times_{C'}Y''\rightarrow Y''$ is $G$-equivariantly bimeromorphic to $g''\colon X''\rightarrow Y''$.
\end{lemma}
\begin{proof}
    The construction of the Stein factorization allows to endow $S$ with a natural $G$-action such that the diagram of the Stein factorization is $G$-equivariant. Therefore, the morphism $s\colon S\rightarrow C'$ is $G$-equivariant as well. Note that over $Y^*\subset Y'$ the tautological model $\mathcal{X}\rightarrow Y'$ is given by $J_H^\theta\rightarrow Y^*$, and $W(\mathcal{L},\alpha,\beta)^{\theta'}\rightarrow Y'$ is given by $J_H^{\theta'=m\theta}\rightarrow Y^*$, where we identified by abuse of notation the classes $\theta\in H_G^1(Y',\mathcal{W}^{mer})$, resp.~ $\theta'\in H_G^1(Y',\mathcal{W}^\#)$ with their images under the restriction maps
    \begin{align*}
        H_G^1(Y',\mathcal{W}^{mer})&\longrightarrow H_G^1(Y^*,\mathcal{W}^{mer}\vert_{Y^*}\cong\mathcal{J}_H),\\
        H_G^1(Y',\mathcal{W}^\#)&\longrightarrow H_G^1(Y^*,\mathcal{W}^\#\vert_{Y^*}\cong\mathcal{J}_H).
    \end{align*}
    Moreover, the map $\mathcal{X}\rightarrow W(\mathcal{L},\alpha,\beta)^{\theta'}$ looks locally like the multiplication by $m$-map
    \begin{equation*}
        J_H^\theta\stackrel{m}{\longrightarrow}J_H^{m\theta}.
    \end{equation*}
    After pulling back along $\psi$ we thus obtain the $G$-equivariant commutative diagram
    \[\begin{tikzcd}
	   {J_{\psi^*H}^\theta} & {J_{\psi^*H}^{m\theta}\cong J_K^{\theta''}\times_{C^*}(C^*\times E'')^\eta} & {J_K^{\theta''}} \\
	   & {S^*} & {C^*,}
	   \arrow[from=1-1, to=1-2]
	   \arrow[from=1-1, to=2-2]
	   \arrow[from=1-2, to=1-3]
	   \arrow[from=1-3, to=2-3]
	   \arrow[from=2-2, to=1-3]
	   \arrow["{s\vert_{S^*}}"', from=2-2, to=2-3]
    \end{tikzcd}\]
    where $S^*:=s^{-1}(C^*)$ and we identified $\eta$ with its image under the restriction map to $C^*$. It follows, that $S^*\rightarrow J_K^{\theta''}$ is also a multiplication by $m$-map on each fiber of $J_K^{\theta''}\rightarrow C^*$. In particular, $s\vert_{S^*}\colon S^*\rightarrow C^*$ is an elliptic fiber bundle with variation of Hodge structure $K$. Thus, there is a cocycle $\theta'''\in H_G^1(C^*,\mathcal{J}_K)$, such that $\theta''=m\theta'''$ and $S^*\rightarrow J_K^{\theta''}$ is the multiplication by $m$-map
    \begin{equation*}
        J_K^{\theta'''}\stackrel{m}{\longrightarrow}J_K^{m\theta'''}.
    \end{equation*}
    Therefore, $J_{\psi^*H}^\theta\cong J_K^{\theta'''}\times_{C^*}(C^*\times E'')^\eta$.
\end{proof}
We are now in the situation to prove \Cref{Kodaira_dim_1_Weierstrass} under \Cref{assumption}.
\begin{proof}[Proof of \Cref{Kodaira_dim_1_Weierstrass} under \Cref{assumption}]
    Suppose we are in the situation of \Cref{assumption}. By \Cref{key_proposition_kodaira_1_G-equivariant}, there is a finite group $G$ and a $G$-equivariant commutative diagram of compact connected K\"ahler manifolds
    \[\begin{tikzcd}
	   {X'} && {(C'\times E')^{\eta'}} \\
	   & {C'} && {,}
	   \arrow["{g'}", from=1-1, to=1-3]
	   \arrow["{f'}"', from=1-1, to=2-2]
	   \arrow["{\pi'}", from=1-3, to=2-2]
    \end{tikzcd}\]
    such that the following holds:
    \begin{enumerate}[label=(\roman*)]
        \item There is a finite \'etale cover $X'/G\rightarrow X$ and a finite cover $C'/G\rightarrow C$.
        \item $E'=\C/\Lambda'$ is an elliptic curve and $\eta'\in H^1_G(C',\OO_{C'}/\Lambda')$ is a cohomology class.
        \setcounter{enumi}{4}
        \item $G$ is either trivial or $C'/G\cong\mathbb{P}^1$.
        \item If $C\cong\mathbb{P}^1$ (for example if $G$ is non-trivial) then $\eta'=0$.
    \end{enumerate}
    Moreover, by \Cref{construction_X''} and \Cref{last_situation_we_consider} there is a finite \'etale cover $(C'\times E'')^{\eta''}\rightarrow (C\times E')^{\eta''}$ and a normal compact K\"ahler surface $S$ with elliptic fibration $s\colon S\rightarrow C'$ such that
    \begin{equation*}
        X''=X'\times_{(C'\times E')^{\eta'}}(C'\times E'')^{\eta''}
    \end{equation*}
    is $G$-equivariantly bimeromorphic to 
    \begin{equation*}
        S\times_{C'}(C'\times E'')^{\eta''}.
    \end{equation*}
    \setcounter{case}{0}
    \begin{case}
        Suppose that $G$ is trivial. Let $\mu\colon S_1\rightarrow S$ be a desingularization. Then $X''$ is bimeromorphic to
        \begin{equation*}
            S_1\times_{C'}(C'\times E'')^{\eta''}\cong (S_1\times E'')^{\mu^*s^*\eta''}.
        \end{equation*}
        We can thus apply \Cref{reduction_to_bimeromorphic_problem} to conclude that there is a smooth minimal K\"ahler surface $T$ such that $X''/G=X''$ is isomorphic to $(T\times E'')^{\eta'''}$ for some $\eta'''\in H^1(T,\OO_T/\Lambda'')$.
    \end{case}
    \begin{case}
        Suppose that $G$ is non-trivial. By \Cref{key_proposition_kodaira_1_G-equivariant} \cref{item5_key_proposition_kodaira_1_G-equivariant} and \cref{item6_key_proposition_kodaira_1_G-equivariant} and \Cref{construction_X''}, we can assume that $\eta''$ is trivial. Thus,
        \begin{equation*}
            S\times_{C'}(C'\times E'')^{\eta''}\cong S\times E''.
        \end{equation*}
        \setcounter{claim}{0}
        \begin{claim}\label{claim_diagonal_action}
            $G$ acts diagonally on $S\times E''$.
        \end{claim}
        Tho prove this claim, simply observe that the $G$ action on $S\times\{e\}\cong S$ is given by the one in \Cref{last_situation_we_consider}. In particular, it does not depend on $e\in E''$. As $G$ acts trivially on $E''$, the $G$-action on $S\times E''$ must be diagonal.\par
        Therefore, as $G$ acts trivially on $E''$,
        \begin{equation*}
            (S\times_{C'}(C'\times E'')^{\eta''})/G\cong (S\times E'')/G\cong S/G\times E''.
        \end{equation*}
        Let $\mu\colon S_1\rightarrow S/G$ be a desingularization. Then $X''/G$ is bimeromorphic to $S_1\times E''$. We can thus apply \Cref{reduction_to_bimeromorphic_problem} to conclude that there is a smooth minimal K\"ahler surface $T$ such that $X''/G$ is isomorphic to $T\times E''$.
    \end{case}
\end{proof}
\begin{remark}
    The authors from \cite{Hao_Schreieder_3-folds} informed me that their proof of step 3 in \cite[][Sec. 6.3]{Hao_Schreieder_3-folds} only covers a special case. They completed the proof in \cite[][Case 2]{Hao_Schreieder_erratum} by invoking Theorem 3.1 of \cite{Hao_Schreieder_BMY}. Note that the results in this paper yield an alternative fix of the aforementioned issue. Indeed, we arrive in a similar situation as in \cite[][Step 3, Sec. 6.3]{Hao_Schreieder_3-folds} in \cref{claim_diagonal_action} in the proof of \Cref{Kodaira_dim_1_Weierstrass} under \Cref{assumption}. The condition that $G$ acts diagonally follows from the fact that there is a prescribed action on both factors of the considered product.
\end{remark}
\subsection{The general fiber is contracted to a point via the Albanese morphism}
\begin{lemma}\label{Iitaka_fiber_contracted_cannot_occur}
    Let $X$ be a smooth minimal K\"ahler threefold of Kodaira dimension 1 and let $f\colon X\rightarrow C$ be its Iitaka fibration. Then the general fiber of $f$ is not contracted to a point by the Albanese morphism.
\end{lemma}
\begin{proof}
    We first follow \cite[][Sec. 6.4]{Hao_Schreieder_3-folds} and then show that this case cannot occur. As the general fiber of the Iitaka fibration $f\colon X\rightarrow C$ is contracted to a point via the Albanese morphism, the same must hold for every fiber. The Albanese morphism, therefore, factors through $f\colon X\rightarrow C\rightarrow Alb(C)$. The universal property of the Albanese hence implies $Alb(X)\cong Alb(C)$. In particular, $b_1(X)=b_1(C)$. Since $X$ admits the non-zero holomorphic 1-form $\omega$ that induces exactness on $H^\bullet(X,\C)$, $b_1(C)\neq 0$ and thus $C$ has positive genus. As $b_1(X)=b_1(C)$, $\omega$ must be a pullback of a holomorphic 1-form $\tau$ on $C$. The exactness of $\wedge\omega$ implies that also $\wedge\tau$ induces an exact complex $(H^\bullet(C,\C),\wedge\tau)$. Thus, $C$ is an elliptic curve.\par
    Suppose first $h^{2,0}(X)\neq 0$. Let $\xi\in H^0(X,\Omega^2_X)$ be a non-zero element. Exactness of $\wedge\omega$ implies $\xi\wedge\omega\neq 0$. As $\omega=f^*\tau$, this shows that $\xi$ is not vertical with respect to $f$. Therefore, by \Cref{smooth_fibers_after_cover}, $f$ is quasi-smooth with typical fiber $F$, which is an abelian surface. The multiplicities induce an orbifold structure on $C$, which is good since $C$ has positive genus. Let $C'\rightarrow C$ be a finite orbifold cover with trivial orbifold structure, see \Cref{classification_orbifolds_curves}. Denote by $X'$ the normalization of the fiber product $X\times_CC'$. Then $X'\rightarrow X$ is finite \'etale. In particular, the pullback $\omega'$ of $\omega$ to $X'$ still has the property that the complex $(H^\bullet(X',\C),\wedge\omega')$
    \begin{equation*}
        \cdots\stackrel{\wedge\omega'}{\longrightarrow} H^{i-1}(X',\C)\stackrel{\wedge\omega'}{\longrightarrow} H^i(X',\C)\stackrel{\wedge\omega'}{\longrightarrow} H^{i+1}(X',\C)\stackrel{\wedge\omega'}{\longrightarrow}\cdots
    \end{equation*}
    is exact. Let $\tau'$ be the pullback of $\tau$ to $C'$ and let $f'\colon X'\rightarrow C'$ be the induced fibration. Since $\omega'=(f')^*\tau'$, also $(H^\bullet(C',\C),\wedge\tau')$ is exact. Hence, $C'$ is elliptic, and thus $X'$ must be a 3-torus. In particular, the Kodaira dimension $\kappa(X')$ of $X'$ equals 0. Since $\kappa(X)=\kappa(X')$, we obtain a contradiction.\par
    In the other two cases, $h^{2,0}(X)=0$ and the general fiber of $f$ is a 2-torus, resp.~ bielliptic, the same argument as in \cite[][Sec. 6.4]{Hao_Schreieder_3-folds} shows that there is a finite \'etale cover $X'\rightarrow X$ with $h^{2,0}(X')\neq 0$. This is possible since Koll\'ar's torsion-free theorem \cite[][Thm. I.2.1]{Kollar_higher_direc_images} also holds in the K\"ahler case \cite[][Thm. 3.21, Rem. 2]{Saito_mixed_hodge_modules}. By the above argument, $X'$ cannot have Kodaira dimension 1, and thus these cases cannot occur either.
\end{proof}
We are now in the position the prove \Cref{Kodaira_dim_1_Weierstrass}.
\begin{proof}[Proof of \Cref{Kodaira_dim_1_Weierstrass}]
    Let $X$ be a smooth minimal K\"ahler threefold of Kodaira dimension 1 and let $f\colon X\rightarrow C$ be its Iitaka fibration. By \Cref{Iitaka_fiber_contracted_cannot_occur}, the general fiber of $f$ cannot be contracted to a point by the Albanese morphism. The general fiber of $f$ is thus either contracted to a curve or not contracted by the Albanese morphism.
    \setcounter{case}{0}
    \begin{case}\label{case_Iitaka_fiber_not_contracted}
        If there is a finite \'etale cover $X'\rightarrow X$ such that the general fiber of the Iitaka fibration $f'\colon X'\rightarrow C'$ is not contracted by the Albanese morphism, we replace $X$ by $X'$. In this case, \Cref{Kodaira_dim_1_Weierstrass} is proven in \Cref{Iitaka_fiber_not_contracted}.
    \end{case}
    \begin{case}
        Suppose that there is no finite \'etale cover $X'\rightarrow X$ such that the general fiber of the Iitaka fibration $f'\colon X'\rightarrow C'$ is not contracted by the Albanese morphism. Then, by \Cref{Iitaka_fiber_contracted_cannot_occur}, for all finite \'etale covers $X'\rightarrow X$, the general fiber of the Iitaka fibration $f'\colon X'\rightarrow C'$ is contracted to a curve by the Albanese morphism. Note that the Iitaka fibration $f'\colon X'\rightarrow C'$ is also given by the Stein factorization of $X'\rightarrow X\rightarrow C$:
        \[\begin{tikzcd}
	       {X'} & X & C \\
	       & {C'} && {.}
	       \arrow[from=1-1, to=1-2]
	       \arrow[from=1-1, to=2-2]
	       \arrow["f", from=1-2, to=1-3]
	       \arrow[from=2-2, to=1-3]
        \end{tikzcd}\]
        Thus, \Cref{assumption} holds true, and \Cref{Kodaira_dim_1_Weierstrass} is proven in \Cref{Iitaka_fiber_contracted_to_curve}.
    \end{case}
\end{proof}
\section{Main result}
\begin{theorem}\label{structure_theorem_Kaeler_3-folds}
    Let $X$ be a smooth compact K\"ahler threefold that satisfies condition \ref{conditionC} from \Cref{conjecture_Kotschick}. Then there is a finite \'etale cover $X'\rightarrow X$ and a sequence of blow-downs $X'=:X'_0\rightarrow\cdots\rightarrow X'_n$ along elliptic curves, a compact K\"ahler manifold $Y$, a positive-dimensional torus $A=\C^n/\Lambda$, a cohomology class $\eta\in H^1(Y,\OO_Y/\Lambda)$ such that $(Y\times A)^\eta$ is K\"ahler, and a smooth morphism
    \begin{equation*}
        X'_n\longrightarrow (Y\times A)^\eta.
    \end{equation*}
    The morphism $X'_n\longrightarrow (Y\times A)^\eta$ is either an isomorphism, or all fibers are isomorphic to $\mathbb{P}^1$, or to $\mathbb{P}^2$ or to a Hirzebruch surface. Moreover, if $\omega\in H^0(X,\Omega^1_X)$ is a holomorphic 1-form such that $(X,\omega)$ satisfies condition \ref{conditionC}, then the induced holomorphic 1-form $\omega_i$ on $X'_i$ restricts non-trivially on the center of the blow-up $X'_{i-1}\rightarrow X'_i$, and $\omega_n$ restricts non-trivially on the fibers of $(Y\times A)^\eta\rightarrow Y$.
\end{theorem}
For the proof, we need the following two statements.
\begin{proposition}\label{not_of_general_type}\cite[][Cor. 3.3]{Hao_Schreieder_3-folds}
    Let $X$ be a smooth projective threefold of general type. Then $X$ does not satisfy condition \ref{conditionC} from \Cref{conjecture_Kotschick}.
\end{proposition}
\begin{proposition}\label{geometry_conic_bundle}\cite[][Thm. 8.3]{Hao_Schreieder_3-folds}
    Let $X$ be a K\"ahler threefold with holomorphic 1-form $\omega\in H^0(X,\Omega^1_X)$ that satisfies condition \ref{conditionC} from \Cref{conjecture_Kotschick}. Suppose that $X$ admits the structure of a Mori fiber space $f\colon X\rightarrow S$ over a surface $S$. Then the following holds true:
    \begin{enumerate}[label=(\roman*)]
        \item $S$ is a smooth compact K\"ahler surface and there is a holomorphic 1-form $\tilde{\omega}\in H^0(S,\Omega^1_S)$ such that $f^*\tilde{\omega}=\omega$. Moreover, $\tilde{\omega}$ satisfies condition \ref{conditionC}.
        \item $f$ is a conic bundle of relative Picard rank 1, and its discriminant divisor is a disjoint union of elliptic curves in $S$.
    \end{enumerate}
\end{proposition}
\begin{proof}
    Note that the argument in \cite[][Thm. 8.3]{Hao_Schreieder_3-folds} still holds true since Mori fiber spaces are projective morphisms and hence described by \cite[][Thm. 3.5]{Mori_3-folds_not_nef}.
\end{proof}
\begin{proof}[Proof of \Cref{structure_theorem_Kaeler_3-folds}]
    First, we apply \Cref{reduction_minimal_model_Mori_fiber_space} to obtain a sequence of blow-downs
    \begin{equation*}
        X=:X_0\longrightarrow\cdots\longrightarrow X_m
    \end{equation*}
    along elliptic curves $E_i\subset X_i$ such that $X_m$ is a smooth minimal model, or a smooth Mori fiber space. Moreover, if $\omega_i\in H^0(X_i,\Omega^1_{X_i})$ denotes the holomorphic 1-form induced by $\omega$, then $\omega_i\vert_{E_i}\neq 0$ and each $(X_i,\omega_i)$ satisfies condition \ref{conditionC} as well.\par
    Note that by Moishezon's theorem, any smooth compact K\"ahler threefold of Kodaira dimension 3 is projective. By \Cref{not_of_general_type} we can thus assume $\kappa(X_m)\leq 2$.\par
    If $\kappa(X_m)=2$, \Cref{Kodaira_dim_2_Weierstrass} implies that there is a smooth compact K\"ahler surface $S$ of general type, an elliptic curve $E=\C/\Lambda$, a cohomology class $\eta\in H^1(S,\OO_S/\Lambda)$ and a finite \'etale cover
    \begin{equation*}
        (S\times E)^\eta\longrightarrow X_m.
    \end{equation*}
    In particular, $(S\times E)^\eta$ is K\"ahler. Define $X'_i:=X_i\times_{X_m}(S\times E)^\eta$. Then the natural morphism $X':=X'_0\rightarrow X$ is finite \'etale and the natural morphism $X'_{i-1}\rightarrow X'_{i}$ is a blow-up along the pullback $E'_{i}\subset X'_{i}$ of $E_{i}\subset X_{i}$. As $X'_{i}\rightarrow X_{i}$ is finite \'etale and finite \'etale covers of elliptic curves are disjoint unions of elliptic curves, we can factor $X'_{i-1}\rightarrow X'_{i}$ as a sequence of blow-ups along the components of $E'_{i}$. Thus, $X'\rightarrow (S\times E)^\eta$ satisfies \Cref{structure_theorem_Kaeler_3-folds}.\par
    If $\kappa(X_m)=1$, \Cref{Kodaira_dim_1_Weierstrass} implies that one of the following cases holds true:
    \setcounter{case}{0}
    \begin{case}
        There is either a smooth curve $C$ of genus at least 2, a 2-torus $A\cong\C^2/\Lambda$, a cohomology class $\eta\in H^1(C,\OO_C^{\oplus 2})$, and a finite \'etale cover
        \begin{equation*}
            (C\times A)^\eta\rightarrow X_m.
        \end{equation*}
        In this case, we define $X'_i:=X_i\times_{X_m}(C\times A)^\eta$. Then the natural morphism $X':=X'_0\rightarrow X$ is finite \'etale and the natural morphism $X'_{i-1}\rightarrow X'_{i}$ is a blow-up along the pullback $E'_{i}\subset X'_{i}$ of $E_{i}\subset X_{i}$. As $X'_{i}\rightarrow X_{i}$ is finite \'etale and finite \'etale covers of elliptic curves are disjoint unions of elliptic curves, we can factor $X'_{i-1}\rightarrow X'_{i}$ as a sequence of blow-ups along the components of $E'_{i}$. Thus, $X'\rightarrow (C\times A)^\eta$ satisfies \Cref{structure_theorem_Kaeler_3-folds}.
    \end{case}
    \begin{case}
        There is a smooth K\"ahler surface $S$ of Kodaira dimension 1 that does not satisfy condition \ref{conditionC}, an elliptic curve $E=\C/\Lambda$, a cohomology class $\eta\in H^1(S,\OO_S/\Lambda)$, and a finite \'etale cover
        \begin{equation*}
            (S\times E)^\eta\rightarrow X_m.
        \end{equation*}
        In this case, we define $X'_i:=X_i\times_{X_m}(S\times E)^\eta$. Then the natural morphism $X':=X'_0\rightarrow X$ is finite \'etale and the natural morphism $X'_{i-1}\rightarrow X'_{i}$ is a blow-up along the pullback $E'_{i}\subset X'_{i}$ of $E_{i}\subset X_{i}$. As $X'_{i}\rightarrow X_{i}$ is finite \'etale and finite \'etale covers of elliptic curves are disjoint unions of elliptic curves, we can factor $X'_{i-1}\rightarrow X'_{i}$ as a sequence of blow-ups along the components of $E'_{i}$. Thus, $X'\rightarrow (S\times E)^\eta$ satisfies \Cref{structure_theorem_Kaeler_3-folds}.
    \end{case}
    If $\kappa(X_m)=0$, there is a positive integer $d$ such that $dK_{X_m}$ is globally generated by the abundance theorem \cite[][Thm. 1.1]{Campana_Hoering_Peternell_Kaehler-abundance}, \cite{Erratum_Addendum_Kaehler-abundance}. Thus, $dK_{X_m}$ is trivial. In particular, $c_1(K_{X_m})=0$ in $H^2(X_m,\mathbb{R})$. The Beauville--Bogomolov decomposition \cite[][Thm. 2]{Beauville_decomposition} together with the fact $b_1(X_m)>0$ shows that there is a finite \'etale cover $Y\times A\rightarrow X_m$, where $A$ is a positive-dimensional torus and $Y$ is a compact connected K\"ahler manifold. Define $X':=X\times_{X_m}(Y\times A)$. Then we argue as before that $X'\rightarrow X$ is finite \'etale and that $X'\rightarrow Y\times A$ satisfies the condition in \Cref{structure_theorem_Kaeler_3-folds}.\par
    Suppose now $\kappa(X_m)=-\infty$. We adapt the arguments from \cite[][Thm. 8.3, Thm. 8.4]{Hao_Schreieder_3-folds}. Suppose first that $X_m$ admits the structure of a Mori fiber space $f\colon X_m\rightarrow S$ over a surface $S$. By \Cref{geometry_conic_bundle} and \Cref{structure_thm_surfaces} we can find a finite \'etale cover $S'\rightarrow S$ such that $S'$ is one of the following
    \begin{enumerate}[label=(\alph*)]
        \item $S'$ is a ruled surface $S'\rightarrow E$ over an elliptic curve $E$,
        \item $S'$ is a 2-torus, or
        \item $S'\cong (C\times E)^\eta\rightarrow C$ for a smooth compact curve $C$ of genus at least 2, an elliptic curve $E\cong\C/\Lambda$, and a cohomology class $\eta\in H^1(C,\OO_C/\Lambda)$ such that $(C\times E)^\eta$ is K\"ahler.
    \end{enumerate}
    Define $X':=X\times_SS'$. If the conic bundle $f'\colon X'\rightarrow S'$ is smooth, there is nothing to show. Suppose that $f'$ is not smooth. By \Cref{geometry_conic_bundle}, the discriminant locus $\Delta'\subset S'$ of $f'$ is a disjoint union of elliptic curves. Let $D'\subset\Delta'$ be a connected component. By \cite[][Prop. 7.1.8 (i)]{Parsin_Shafarevic_AG_V}, the fiber of $f'$ over a point in $D'$ consists of two distinct lines that meet in a point. The monodromy of the components is finite over $D'$. Thus, there is a finite \'etale cover $D''\rightarrow D'$ that trivializes the monodromy. Note that we can extend this to get a finite \'etale cover $S''\rightarrow S'$ that restricts over $D'$ to $D''\rightarrow D'$. Indeed, if $S'$ is of type (a) or (b), this is clear. If $S'$ is of type (c), this follows from \Cref{lemma_extension_of_etale_cover_of_elliptic_curves}. Indeed, as $C$ has genus at least 2, the map $D'\rightarrow C$ is constant and thus $D'$ is an elliptic fiber of $S'\cong (C\times E)^\eta\rightarrow C$. In any case, $S''$ is again of the same type as $S'$. Repeating this process, we may assume that the monodromy over each component of $\Delta'$ is trivial. Denote the resulting conic bundle by $f''\colon X''_m\rightarrow S''$ and denote its discriminant divisor by $\Delta''\subset S''$. Note that we have a divisorial contraction for each component of $\Delta''\subset S''$ that contracts exactly one of the two families of lines over the respective component. Define $X''_i:=X_i\times_{X_m}X''_m$. We hence obtain a sequence of contractions
    \begin{equation*}
        X''=:X''_0\longrightarrow\cdots\longrightarrow X''_m\longrightarrow\cdots\longrightarrow X''_n,
    \end{equation*}
    where $X''_n\rightarrow S''$ is a smooth conic bundle. The contractions $X''_{i-1}\rightarrow X''_{i}$ are blow-ups along the respective component of $\Delta''$ for $m+1\leq i\leq n$. Thus, $X''\rightarrow X$ is finite \'etale and $X''\rightarrow X''_n$ satisfies the condition in \Cref{structure_theorem_Kaeler_3-folds}. Depending on the type of $S''$, we obtain one of the following:
    \begin{enumerate}[label=(\alph*)]
        \item a smooth morphism $X''_n\rightarrow S''\rightarrow E'$ to the elliptic curve $E'$ such that the fibers are ruled surfaces over $\mathbb{P}^1$, i.e. Hirzebruch surfaces,
        \item a smooth morphism $X''_n\rightarrow S''$ to a 2-torus with fibers isomorphic to $\mathbb{P}^1$, or
        \item a smooth morphism $X''_n\rightarrow (C\times E')^{\eta'}$ with fibers isomorphic to $\mathbb{P}^1$, where $E'\rightarrow E$ is a finite \'etale cover and $(C\times E')^{\eta'}\rightarrow (C\times E)^{\eta}$ is an extension of $E'\rightarrow E$ due to \Cref{lemma_extension_of_etale_cover_of_elliptic_curves}.
    \end{enumerate}
    Suppose now that $X_m$ admits the structure of a Mori fiber space $f\colon X_m\rightarrow C$ over a smooth compact curve $C$. By \cite[][Sec. 1]{Hoering_Peternell_Kaehler-MFS}, $X_m$ must be projective. We can thus apply \cite[][Thm. 8.4]{Hao_Schreieder_3-folds} to conclude that the fibration $f\colon X_m\rightarrow C$ is a smooth del Pezzo fibration over an elliptic curve and that the local system $R^2f_*\C$ has finite monodromy. We can hence find a finite \'etale cover $E\rightarrow C$ that trivializes the monodromy. Denote by $X'_m\rightarrow X_m$ the induced finite \'etale cover and define $X'_i:=X_i\times_{X_m}X'_m$. The induced smooth del Pezzo fibration $X'_m\rightarrow E$ may not be of relative Picard rank 1. However, applying \Cref{reduction_minimal_model_Mori_fiber_space} yields a sequence of contractions
    \begin{equation*}
        X'=:X'_0\longrightarrow\cdots\longrightarrow X'_m\longrightarrow\cdots\longrightarrow X'_n,
    \end{equation*}
    that satisfies the condition in \Cref{structure_theorem_Kaeler_3-folds} such that the induced fibration $f'\colon X'_n\rightarrow E$ is a Mori fiber space. As $R^2f'_*\C$ has trivial monodromy and $f'$ has relative Picard rank 1, it follows that $f'\colon X'_n\rightarrow E$ is smooth with fibers isomorphic to $\mathbb{P}^2$. The case of a Mori fiber space over a point cannot occur since, in this case, $X$ would have to be Fano, and hence, it does not satisfy condition \ref{conditionC}.
\end{proof}
\printbibliography%[heading=bibintoc, title={References}]
\end{document}